\documentclass[a4paper]{amsart}
\usepackage[latin1]{inputenc}
\usepackage{amsmath}
\usepackage{amsfonts}
\usepackage{amssymb}
\usepackage{graphicx}
\usepackage{tikz}
\numberwithin{equation}{section}
\newtheorem{theorem}{Theorem}[section]
\newtheorem{lemma}[theorem]{Lemma}
\newtheorem{proposition}[theorem]{Proposition}
\newtheorem{corollary}[theorem]{Corollary}
\theoremstyle{definition}
\newtheorem{definition}[theorem]{Definition}
\newtheorem{example}[theorem]{Example}

\theoremstyle{remark}
\newtheorem{remark}[theorem]{\bf{Remark}}

\newcommand{\lcross}{{>\!\!\!\triangleleft}}

\newcommand{\bicross}{{\blacktriangleright\!\!\!\triangleleft}}

\newcommand{\lrbicross}{{\blacktriangleright\!\!\!\triangleleft}}
\newcommand{\dcross}{{\bowtie}}
\newcommand{\lbiprod}{{>\!\!\!\triangleleft\kern-.33em\cdot}}
\newcommand{\rbiprod}{{\cdot\kern-.33em\triangleright\!\!\!<}}

\newcommand{\codcross}{{\blacktriangleright\!\!\blacktriangleleft}}

\newcommand{\cop}{\mathrm{\rm cop}}
\newcommand{\op}{\mathrm{\rm op}}

\makeatletter
\newcommand{\mathleft}{\@fleqntrue\@mathmargin0pt}
\newcommand{\mathcenter}{\@fleqnfalse}
\makeatother

\renewcommand{\o}{{}_{\scriptscriptstyle(1)}} 
\renewcommand{\t}{{}_{\scriptscriptstyle(2)}}
\renewcommand{\th}{{}_{\scriptscriptstyle(3)}}
\newcommand{\fo}{{}_{\scriptscriptstyle(4)}}
\newcommand{\fiv}{{}_{\scriptscriptstyle(5)}}
\newcommand{\six}{{}_{\scriptscriptstyle(6)}}
\newcommand{\sev}{{}_{\scriptscriptstyle(7)}}
\newcommand{\ei}{{}_{\scriptscriptstyle(8)}}

\newcommand{\C}{{\Bbb C}}

\newcommand{\Z}{{\Bbb Z}}

\newcommand{\bt}{\mathbf{t}}
\newcommand{\bs}{\mathbf{s}}
\newcommand{\bx}{\mathbf{x}}
\newcommand{\bu}{\mathbf{u}}

\newcommand{\eps}{\epsilon}
\newcommand{\CR}{{\mathcal R}}
\newcommand{\CC}{{\mathcal C}}

\newcommand{\CM}{{\mathcal M}}
\newcommand{\tens}{\otimes}
\newcommand{\id}{{\rm id}}
\newcommand{\extd}{{\rm d}}

\renewcommand{\(}{{(}}

\newcommand{\la}{{\triangleright}}
\newcommand{\ra}{{\triangleleft}}

\newcommand{\uunderline}[1]{\mkern3mu\underline{\mkern-3mu #1\mkern-3mu}\mkern3mu}
\newcommand{\ooverline}[1]{\mkern3mu\overline{\mkern-3mu #1\mkern-3mu}\mkern3mu}

\newcommand{\baro}{\ooverline{{}_{\scriptscriptstyle(1)}}}
\newcommand{\barnot}{\ooverline{{}_{\scriptscriptstyle(0)}}}
\newcommand{\barinfi}{\ooverline{{}_{\scriptscriptstyle(\infty)}}}

\newcommand{\tilo}{\widetilde{{}_{\scriptscriptstyle(1)}}}
\newcommand{\tilnot}{\widetilde{{}_{\scriptscriptstyle(0)}}}

\newcommand{\undo}{\uunderline{\o}}
\newcommand{\undt}{\uunderline{\t}}

\allowdisplaybreaks

\title[Quantum differentials on cross product Hopf algebras]{Quantum differentials on cross product Hopf algebras}
\author{Ryan Aziz and Shahn Majid}
\address{School of Mathematical Sciences\\ Queen Mary University of London \\ Mile End Rd, London E1 4NS }
\email{r.aziz@qmul.ac.uk, s.majid@qmul.ac.uk}
\thanks{LPDP Indonesian Endowment Fund For Education No. 1571072104910081}
\begin{document}
\begin{abstract} We construct canonical strongly bicovariant differential graded algebra structures on all four flavours of cross product Hopf algebras, namely double cross products  $A\hookrightarrow A\dcross H\hookleftarrow H$, double cross coproducts $A\twoheadleftarrow A\codcross H\twoheadrightarrow H$, biproducts $A{\buildrel\hookrightarrow\over \twoheadleftarrow}A\rbiprod B$ and bicrossproducts $A\hookrightarrow A\bicross H\twoheadrightarrow H$  on the assumption that the factors have strongly bicovariant calculi $\Omega(A),\Omega(H)$  (or a braided version $\Omega(B)$). We use super versions of each of the  constructions. Moreover, the latter three quantum groups all coact canonically on one of their factors and we show that this coaction is differentiable. In the case of the Drinfeld double  $D(A,H)=A^{\rm op}\dcross H$ (where $A$ is dually paired to $H$), we show that its canonical actions on $A,H$ are differentiable. Examples include a canonical $\Omega(\C_q[GL_2\ltimes \C^2])$ for the quantum group of affine transformations of the quantum  plane  and  $\Omega(\C_\lambda[{\rm Poinc_{1,1}}])$ for the bicrossproduct Poincar\'e quantum group in 2 dimensions. We also show that  $\Omega(\C_q[GL_2])$ itself is uniquely determined by differentiability of the canonical coaction on the quantum plane and of the determinant subalgebra.  \end{abstract}
\maketitle

\section{Introduction}\label{sec 1} 

A modern `quantum groups' approach to noncommutative geometry starts with differential structures on quantum groups or Hopf algebras and associated comodule algebras, expressed in the form of a differential graded algebra (DGA), see \cite{QRG} and an extensive literature. In general, there will be many such DGAs on a given Hopf algebra even if we demand left and right translation invariance (i.e., bicovariance). The most general result here, an immediate consequence of the Hopf module lemma first pointed out by  Woronowicz\cite{Wor}, is that bicovariant calculi at first order correspond to (say) right ideals of the augmentation ideal that are stable under the right adjoint coaction. This, however, only translates the problem to finding such ideals. If the Hopf algebra is coquasitriangular and essentially factorisable then a general classification was obtained in \cite{Ma:cla} and later in \cite{BS} in terms of the irreducible corepresentations, but for other types of Hopf algebra there are no such general results. Of particular interest are inhomogeneous quantum groups, normally some form of cross product of `translations' and `rotations' coacting on another copy of the former. This is needed in physics for a quantum Poincar\'e group coordinate algebra coacting on quantum spacetime, but until now there has generally been no canonical choice of differential structure needed further to develop the geometry of such inhomogeneous quantum groups. 

Here we solve this problem for all main flavours  of inhomogeneous quantum groups in a way which, rather than classifying all bicovariant calculi, selects out a canonical calculus for given calculi on the factors with the property  that the inhomogeneous quantum group coacts {\em differentiably} on its  canonically associated comodule algebra. This applies to biproduct Hopf algebras $A\rbiprod B$ where $A$ acts and coacts on a braided-Hopf algebra $B$ from the same side\cite{Rad,Ma:skl,Ma:book} and the biproduct coacts on $B$, bicrossproducts $A\bicross H$ where $H$ acts on $A$ and $A$ coacts on $H$ \cite{Ma:phy,Ma:book} and the bicrossproduct coacts on $H$, and double cross coproducts $A\codcross H$ where each factor coacts on the other and the double cross coproduct coacts on each factor. The fourth flavour is in principle the double cross product $A\dcross H$ where each factor acts on the other and the double acts canonically on $A^*,H^*$,  but differentiability would need us to have calculi on $A^*,H^*$ which is further data that would need to be supplied. The special case of the Drinfeld double $D(A,H)=A^{\op}\dcross H$ in the generalised form for dually paired Hopf algebras $A,H$ (as introduced by the 2nd author\cite{Ma:phy,Ma:book}), however, has a canonical action on $H,A$ and these we show are differentiable {\em as actions}. (This is a  slightly unusual concept for a noncommutative geometer if we think of our algebras as `coordinate algebras', but arises naturally here.)  
We also note that first order differential calculi on finite group bicrossproducts $\C(M)\bicross \C G$, where a finite group factorises into $G,M$, were classified in \cite{NME} if one wants the full moduli of these in isolation. 

Our approach starts with the notion of a {\em strongly bicovariant} exterior algebra $\Omega(A)$ introduced in \cite{MaTao}. Here an exterior algebra means $\Omega(A)=\oplus_i\Omega^i$ is a graded algebra with a graded derivation $\extd$ with $\extd^2=0$ (a DGA) which is generated by $\Omega^0=A$ and $\Omega^1=A\extd A$. This is strongly bicovariant if $\Omega(A)$ is a super-Hopf algebra with coproduct $\Delta_*$ (say) where $\Delta_*|_A=\Delta$, the coproduct of $A$, in degree zero and $\extd$ is a graded coderivation in the sense
\[ \Delta_*\extd = (\extd\tens \id+ (-1)^{|\ |}\tens\extd)\Delta_*,\]
where $(-1)^{|\ |}\omega=(-1)^{|\omega|}$ for $\omega$ of degree $|\omega|$. It was shown in \cite{MaTao} that the canonical exterior algebra due to Woronowicz\cite{Wor} of a first order bicovariant calculus is strongly bicovariant with $\Delta_*|_{\Omega^1}=\Delta_L+\Delta_R$ in terms of the coactions on $\Omega^1$ induced by $\Delta$ (that the Woronowicz construction gives a super-Hopf algebra here was already known\cite{Brz} while the supercoderivation property was new). Conversely, if $\Omega(A)$ is strongly bicovariant then $\Delta_*|_{\Omega^i}\subseteq (A\tens\Omega^i)\oplus\cdots\oplus (\Omega^i\tens A)$ recovers the left and right coactions on each  $\Omega^i$ as the outer components, making $\Omega(A)$ bicovariant as a left and right comodule algebra, but contains a lot more information in the intermediate components. One can add to this the idea  that an algebra $B$ with say a right coaction $\Delta_R:B\to B\tens A$ has an exterior algebra $\Omega(B)$ such that $\Delta_R$ extends to $\Delta_{R*}:\Omega(B)\to \Omega(B)\underline{\tens}\Omega(A)$ as a super comodule algebra. Then $\Delta_R$ is said to be {\em differentiable}. This is used in \cite{QRG} in a different context from the present paper, namely in the theory of quantum principal bundles and fibrations.  Exact definitions and any further conditions are provided in the preliminaries Section~\ref{secpre}. 

Our canonical construction in this context will be to construct the required exterior algebra as a super extension of the same flavour of cross (co)product. Thus we define $\Omega(A\dcross H):=\Omega(A)\dcross \Omega(H)$ and $\Omega(A\codcross H):=\Omega(A)\codcross\Omega(H)$ in Section~\ref{secdouble},  $\Omega(A\rbiprod B):=\Omega(A)\rbiprod \Omega(B)$ in Section~\ref{secbiprod} and $\Omega(A\lrbicross H):=\Omega(A)\lrbicross \Omega(H)$ in Section~\ref{secbicross}, using the obvious super-Hopf versions of the cross products.  In each case, the new part is to check that the exterior derivative $\extd$ extends correctly so that the result is strongly bicovariant, and that the canonical (co)actions where applicable are differentiable. 

We will see that this approach is remarkably effective across a range of otherwise rather complicated examples in the literature. Thus, bicovariant calculi on the Sweedler-Taft algebra $U_q(b_+)$ (a sub-Hopf algebra of $U_q(su_2)$) were studied in \cite{Oec} and our construction lands on the most natural one in the classification there, albeit in the dual language of $\C_q[B_+]=\C[t,t^{-1}]\rbiprod \C[x]$ where $\C[x]$ is developed as a braided-line in the category of $\Z$-graded spaces or $\C[t,t^{-1}]$-comodules. The exterior algebra $\Omega(\C_q[B_+])$ is in Proposition~\ref{ExB+} as a warm-up to Proposition~\ref{propexP} for $\Omega(\C_q[GL_2]\rbiprod\C_q^2)$, where the quantum plane $\C_q^2$ is viewed as a braided Hopf algebra in the category of $\C_q[GL_2]$-comodules. Moreover, our general construction ensures that the canonical coaction of $\C_q[B_+]$ on $\C[x]$ and the canonical coaction of  $\C_q[GL_2]\rbiprod\C_q^2$ on $\C_q^2$ are both differentiable for their natural $q$-differential calculi.  Similarly, one of the simplest bicrossproducts beyond finite groups is the Planck-scale Hopf algebra $\C[g,g^{-1}]\bicross\C[p]$ and Proposition~\ref{explanck} recovers a known exterior algebra $\Omega(\C[g,g^{-1}]\bicross\C[p])$  in \cite{MaOec} but now uniquely determined by the Hopf algebra coacting differentiably on $\C[x]$ with its unique finite-difference form of bicovariant calculus. This is a warm-up for the  Poincar\'e  quantum group $\C_\lambda[{\rm Poinc}_{1,1}]= \C[SO_{1,1}] \bicross U(b_+)$  in a well-known family of models of quantum spacetime, see \cite[Chapter~9.1]{QRG} for a recent account. We show in Proposition~\ref{expoinc} that this has a uniquely determined exterior algebra $\Omega(\C_\lambda[{\rm Poinc}_{1,1}])$ such that the natural coaction of the Poincar\'e quantum  group on the associated quantum spacetime $U(b_+)$ is differentiable. Here  $U(b_+)$ with primitive generators $x,t$ and relations $[x,t]=\lambda x$ is regarded as the algebra of spacetime coordinates in 1+1 dimensions, equipped with its unique Poincar\'e covariant calculus from\cite{Sit}. 

The larger picture that emerges in this work is that we should not choose differential structures independently on each algebra but in a coherent manner such that natural maps between them are differentiable. Thus, it is well known following Manin\cite{Manin} that the structure of $\C_q[GL_2]$ itself is largely determined by its coaction on the quantum plane (and on its Koszul dual). In passing, in Proposition~\ref{new 4D calc}, we go a step further and show that the standard but rather complicated relations of its 4D differential calculus are similarly determined uniquely by requiring differentiablity of this coaction on the quantum plane and of the determinant subalgebra. Also, without giving details, it follows from Example~\ref{ex dcross B(R)} in Section~\ref{sectrans} that a previously known coaction of $\C_q[SO_{1,3}]$   on $q$-Minkowski spacetime is differentiable, for a suitable  choice of $\Omega(\C_q[SO_{1,3}])$. The general Hopf algebraic picture here is that if $A$ is any coquasitriangular Hopf algebra then the coquasitriangular structure provides a (skew) pairing of $A$ with itself and hence a  canonical `double' coquasitriangular Hopf algebra $D(A,A)= A\dcross A$, leading to natural calculi $\Omega'(A)\dcross\Omega(A)$ on $A\dcross A$ in Section~\ref{secAdcrossA}, where we can choose the calculi on each factor independently. Section~\ref{sectrans} brings these results together with the theory of transmutation.

\section{Preliminaries on strongly bicovariant differential calculi}\label{secpre}

Here we recall what is meant by a strongly bicovariant differential calculus on a bialgebra or Hopf algebra, following \cite{MaTao}, and what it means for a coaction by it to be differentiable. We also look at what it means for an action to be differentiable. The section includes some first lemmas that will be needed later as well the example of the FRT bialgebra $A(R)$ in the $q$-Hecke case coacting differentiably on the associated quantum plane. 

\subsection{Super, braided and super-braided Hopf algebras}\label{secprehop}
Let $k$ be a field. We recall that a Hopf algebra $A$ is a unital algebra together with two algebra maps, the coproduct $\Delta : A\to A \tens A$ and counit $\epsilon : A \to k$, forming a coalgebra, and equipped with an antipode map $S : A\to A$ defined by $(Sa\o)a\t = \epsilon a = a\o Sa\t$ for all $a\in A$. We use the Sweedler notation  $\Delta a =a\o \tens a\t$ (summation understood). More generally, a super-Hopf algebra is a $\Z_2$-graded Hopf algebra $A$, where $A=A_0\oplus A_1$ with grade $|a|=i$ for $a\in A_i$, and $i=0,1$. Here the coproduct $\Delta:A\to A\underline{\tens} A$ respects the total grades and is an algebra map to the super tensor product algebra
\[(a\tens b)(c\tens d) = (-1)^{|b||c|}(ac \tens bd)\]
for $a,b,c,d\in A$. The counit also respects the grade and hence $\eps|_{A_1}=0$ if we work over a field $k$.  If $A$ is a super-Hopf algebra then $A\underline{\tens} A$ is also a super-Hopf algebra, with 
\[\Delta(a\tens b) = (-1)^{|a\t||b\o|}a\o \tens b\o \tens a\t \tens b\t\]
for $a,b\in A$ of homogeneous grade. The terms in the implicit summation are also taken with factors of homogeneous grade understood.  

The notions of action and coaction under a Hopf algebra similarly have super versions respecting the grade. Thus $V = V_0 \oplus V_1$ is a super right $A$-module algebra if $V$ is a super right $A$-module by $\ra : V\tens A \to A$ and additionally 
\[(v w)\ra a = (-1)^{|\omega||a\o|} (v\ra a\o)(w\ra a\t)\]
for  $v, w \in V$ and $a \in A$ with  $w,a$ of homogeneous grade. A super right $A$-module $V$ is a super right $A$-module coalgebra if additionally
\[\Delta(v\ra a) = (-1)^{|v\t||a\o|} v\o \ra a\o \tens v\t \ra a\t \]
for $v\in V$ and $a\in A$ of homogeneous grade. Similarly, $V$ is a super right $A$-comodule algebra if $V$ is a super right $A$-comodule by $\Delta_R : V \to V\tens A$, denoted by $\Delta_R v = v^{\barnot}\tens v^{\overline{\o}}$ (summation understood) and additionally
\[\Delta_R(vw) = (-1)^{|w^{\barnot}||v^{\baro}|}v^{\barnot}w^{\barnot} \tens v^{\baro}w^{\baro}\]
for $v,w\in V$ with the relevant components of homogeneous grade. A super right $A$-comodule $V$ is a super right $A$-comodule coalgebra if additionally
\[(\Delta\tens \id)\Delta_R v = (-1)^{|v\t{}^{\barnot}||v\o{}^{\baro}|}v\o{}^{\barnot}\tens v\t{}^{\barnot}\tens v\o{}^{\baro}v\t{}^{\baro}.\]
There is also an equally good left-handed notion of super $A$-(co)module (co)algebras.

The usual notion of crossed or Drinfeld-Radford-Yetter modules (cf. \cite{Ma:book, Ma:skl, Rad}) also has a super version: $V$ is a super right $A$-crossed module over a super-Hopf algebra $A$ if $V$ is both a super right $A$-module and a super right $A$-comodule, such that
\[\Delta_R(v\ra a) = (-1)^{|v^{\overline{\o}}|(|a\o|+|a\t|)+|a\o||a\t|}v^{\barnot}\ra a\t \tens (Sa\o)v^{\baro}a\th\]
for all  $v\in V$ and $a \in A$ with components of homogeneous grade. The category $\CM^A_A$ of super right $A$-crossed modules is a prebraided category with the braiding $\Psi: V \tens W \to W\tens V$ given by 
\[\Psi(v\tens w) =(-1)^{|v||w^{\barnot}|}w^{\barnot}\tens(v\ra w^{\overline{\o}})\]
for $v\in V$ and $w\in W$ with components of homogenous degree. This is invertible (so the category is braided) if $A$ has invertible antipode, which we will assume throughout in this context.

Next we recall that a braided Hopf algebra (or `braided group') $B$ as introduced in \cite{Ma:bra} is a Hopf algebra in a braided category $\CC$, where we view the unit element as a morphism $\eta : \underline{1} \to B$, and the product, coproduct, counit $\eps:B\to\underline{1}$ and antipode are morphisms. In a concrete setting, we denote the  coproduct of a braided Hopf algebra as $\underline{\Delta} b = b\underline{\o}\tens b\underline{\t}$ (for $b\in B$, summation understood) where $\underline{\Delta}$ is an algebra morphism to the braided tensor product algebra so that $\underline{\Delta}(bc) = b\underline{\o}\Psi(b\underline{\t} \tens c\underline{\o})c\underline{\t}$, with the braiding $\Psi$ of the category on $B\tens B$. In the category of super or $\Z_2$-graded vector spaces, this just means a super-Hopf algebra since the categorical braiding on $B$ is  $\Psi(b\tens c) = (-1)^{|b||c|}c\tens b$. We will also need the notion of super braided-Hopf algebra satisfying the same axioms as a braided-Hopf algebra but $\Z_2$-graded  and with an extra factor $(-1)^{| \ | | \ |}$ in every braiding. This assumes direct sums in the category and can be viewed as a braided Hopf algebra in a super version of the braided category where objects are of the form $V_0\oplus V_1$ and the braiding has the additional signs according to the degree.  

When $A$ is an ordinary Hopf algebra and $B$ is a braided-Hopf algebra in $\CM^A_A$, the  category of ordinary right $A$-crossed modules, there is an ordinary Hopf algebra $A\rbiprod B$. This is the bosonisation cf.\cite{Ma:bos} or more precisely the biproduct\cite{Rad,Ma:skl} of $B$ built on $A \tens B$ with product and coproduct
\[(a \tens b)(c \tens d)= ac\o \tens (b \ra c\t)d\]
\[\Delta(a \tens b) = a\o \tens b\underline{\o}^{\barnot}\tens a\t b\underline{\o}^{\baro}\tens b\underline{\t}\]
for $a,c\in A$, $b,d\in B$. This structure was found in \cite{Rad} from a study of Hopf algebras with projection (they are of this form) prior to the theory of braided-Hopf algebras, while the above formulation is due to the second author in \cite{Ma:skl}. If $A$ is (co)quasitriangular then a functor\cite{Ma:dou, Ma:book} from $A$-(co)modules to $A$-crossed modules includes  (co)bosonisation in the more strict original sense of \cite{Ma:bos}. There is also an equally good left-handed version for braided-Hopf algebra $B$ in the category of left $A$-crossed modules ${}^A_A\CM$ which gives a bosonisation $B\lbiprod A$. At this point it should be clear that these constructions too have obvious super versions, which we will use in what follows.

\subsection{Differentials on Hopf algebras}\label{secpredif}
A differential calculus on an algebra $A$ is an $A$-$A$-bimodule $\Omega^1 = \mathrm{span}\{a\extd b\}$ where $\extd: A \to \Omega^1$ is a linear map obeying the Leibniz rule $\extd(ab)=(\extd a) b + a \extd b$ for all $a,b\in A$. The calculus $\Omega^1$ is said to be inner if there is an element $\theta \in \Omega^1$ such that $\extd = [\theta, \ ]$.
If, moreover,  $A$ is a bialgebra or Hopf algebra then $\Omega^1$ is said to be left-covariant if $\Omega^1$ is a left $A$-comodule with left coaction $\Delta_L : \Omega^1 \to A \tens \Omega^1$ such that $\Delta_L \extd = (\id \tens \extd )\Delta$. It is said to be right covariant if $\Omega^1$ is a right $A$-comodule with right coaction $\Delta_R : \Omega^1 \to \Omega^1 \tens A$ such that $\Delta_R \extd= (\extd \tens \id)\Delta$. If $\Omega^1$ is both left and right covariant, it is said to be bicovariant.

By {\em exterior algebra} on $A$, we mean that $A,\Omega^1$ as above extend to a differential graded algebra $\Omega = \oplus_{n\geq 0}\Omega^n$ generated by $A=\Omega^0,\Omega^1$ and $\extd$ extends  to $\extd : \Omega^n \to \Omega^{n+1}$ such that $\extd^2=0$ and the graded Leibniz rule $\extd(\omega \eta) = (\extd\omega)\eta + (-1)^{|\omega|}\omega\extd \eta$ holds for all $\omega,\eta \in \Omega$. Here $|\omega|$ denotes the degree of an element $\omega$ of homogeneous degree. Again, when $A$ is a bialgebra or Hopf algebra, $\Omega$ is left covariant if it is a left $A$-comodule algebra with degree-preserving $\Delta_L$ commuting with $\extd$, and similarly for the right covariant and bicovariant cases. Finally, an exterior algebra $\Omega$ on $A$ is {\em strongly bicovariant}\cite{MaTao} if it is respectively a super-bialgebra or super-Hopf algebra with the super grade given by the form degree mod 2,  the super-coproduct $\Delta_*$ is degree-preserving and restricts to the coproduct of $A$, and  $\extd$ is a super-coderivation in the sense
\begin{equation}\label{super-coderivation}
\Delta_* \extd \omega = (\extd \tens \id + (-1)^{| \ |}\id \tens \extd)\Delta_* \omega
\end{equation}
for all $\omega\in \Omega$ as discussed in the introduction. It is proved in \cite{MaTao} that any strongly bicovariant exterior algebra is bicovariant, justifying the terminology, with $\Delta_L,\Delta_R$ on $\Omega^i$ extracted from the relevant degree component of $\Delta_*|_{\Omega^i}$. For example, $\Delta_*|_{\Omega^1}=\Delta_L+\Delta_R$ for the coactions on $\Omega^1$.

Given an algebra with differential calculus $(A,\Omega^1,\extd)$, one has the maximal prolongation exterior algebra  $\Omega_{\max}$ where we impose the quadratic relation $\sum_i \extd b_i \extd c_i + \sum_j \extd r_j \extd s_j$ whenever $\sum_i \extd b_i.c_i - \sum_j r_j\extd s_j =0$ is a relation in $\Omega^1$, where $b_i, c_i, r_j, s_j \in A$. For a bicovariant $(A,\Omega^1,\extd)$ on a Hopf algebra one has a canonical exterior algebra $\Omega_{\rm wor}$ in \cite{Wor}, see \cite{Ma:hod} for a modern treatment. In this case both $\Omega_{\rm max}$ and $\Omega_{\rm wor}$ are strongly bicovariant, see \cite{Ma:hod} and \cite{Brz} respectively.

Finally, in the strongly bicovariant Hopf algebra case it follows\cite{MaTao} by a super version of the Radford-Majid biproduct theorem\cite{Rad,Ma:skl}  that $\Omega \cong A\rbiprod\Lambda$ is  a super-bosonisation, where $\Lambda = \oplus_{i\geq 1} \Lambda^i$ is the subalgebra of left-invariant differential forms and forms a super braided-Hopf algebra in category of right $A$-crossed modules $\CM^A_A$. 

\begin{lemma}\label{lemma diff tens prod}
Let $A,H$ be Hopf algebras and $\Omega(A), \Omega(H)$ strongly bicovariant exterior algebras on them. Then $\Omega(A\tens H):=\Omega(A)\underline{\tens} \Omega(H)$ is a strongly bicovariant exterior algebra on $A\tens H$ with super tensor product algebra, super tensor coproduct coalgebra and differential
\[\extd (\omega \tens \eta) = \extd_A \omega \tens \eta + (-1)^{|\omega|} \omega \tens \extd_H \eta\]
for all $\omega \in \Omega(A)$ and $\eta \in \Omega(H)$. Furthermore, $\Omega(A\tens H)\cong (A \tens H)\rbiprod (\Lambda_A \tens \Lambda_H)$. 
\end{lemma}
\begin{proof}

 The graded Leibniz rule holds since $\extd$ restricts to $\extd_A$ and $\extd_H$ on $\Omega(A)$ and $\Omega(H)$, and the algebra is just a super tensor product algebra. For the super-coderivation property, it again suffices to
check this separately on $\omega\tens 1$ and $1\tens \eta$, where it clearly reduces to the same property for $\Omega(A)$ and $\Omega(H)$ respectively. 

Furthermore, $A\tens H$ acts and coacts on $\Lambda_A \tens \Lambda_H$ by $(v\tens w)\ra (a\tens h) = v\ra a \tens w\ra h$ and $\Delta_R(v\tens w) = v^{\barnot}\tens w^{\barnot}\tens v^{\baro}\tens w^{\baro}$ for all $a\in A, h \in H$ and $v\in \Lambda_A, w\in \Lambda_H$, making $\Lambda_A \tens \Lambda_H$ an $A\tens H$-crossed module since
\begin{align*}
\Delta_R((v\tens w)\ra (a\tens h)) =& (v^{\barnot}\tens w^{\barnot})\ra (a\t\tens h\t) \tens (S(a\o \tens h\o))(v^{\baro}\tens w^{\baro})(a\th \tens h\th)\\
=& (v^{\barnot}\ra a\t)\tens (w^{\barnot}\ra h\t) \tens (Sa\o)v^{\baro}a\th \tens (Sh\o)w^{\baro}h\th\\
=& (v\ra a)^{\barnot}\tens (w\ra h)^{\barnot}\tens (v\ra a)^{\baro}\tens (w\ra h)^{\baro}\\
=&\Delta_R((v\ra a)\tens (w\ra h)).
\end{align*}
Thus there is a bosonisation $(A\tens H)\rbiprod (\Lambda_A \tens \Lambda_H)$ and it is easy to show that $\Omega(A)\tens \Omega(H)\cong A\rbiprod \Lambda_A \tens H \rbiprod \Lambda_H \cong (A\tens H) \rbiprod (\Lambda_A \tens \Lambda_H)$ by a swap of the middle tensor factors. 
\end{proof}

\subsection{Differentiable coactions and actions}

A map between algebras with differential structure 
as {\em differentiable} if it extends to a map of DGAs (possibly up to some degree). Let $A$ be bialgebra or Hopf algebra with exterior algebra $\Omega(A)$ (not necessarily bicovariant), and let $B$ be an algebra with exterior algebra $\Omega(B)$. We first formulate the concept of differentiable coaction as also discussed in \cite[Chapter 4.2]{QRG}. We also note the graded tensor product of DGAs as already used above, which then leads to following.

\begin{definition}\label{defdif} If $B$ is an $A$-comodule algebra, the coaction $\Delta_R b = b^{\barnot}\tens b^{\baro}$ is called {\em differentiable} if it extends to a degree-preserving map  $\Delta_{R*} : \Omega(B) \to \Omega(B) \underline{\tens} \Omega(A)$ of differential exterior algebras. \end{definition}
 Explicitly, we require $\Delta_{R*}$ a total degree preserving map of superalgebras such that $\Delta_{R*}\extd_B = \extd_{B\tens A} \Delta_{R*}$, where the latter means
\begin{equation}\label{diff coact}
\Delta_{R*}\extd_B \eta = \extd_B \eta^{\barnot *} \tens \eta^{\baro *} +(-1)^{|\eta^{\barnot *}|} \eta^{\barnot *}\tens \extd_A \eta^{\baro *}
\end{equation} 
with notation $\Delta_{R*}\eta = \eta^{\barnot *}\tens \eta^{\baro *} \in \Omega(B)\underline{\tens}\Omega(A)$ for all $\eta \in \Omega(B)$.

If $\Delta_{R*}$ exists then it is uniquely determined from $\Delta_R$. For instance, on $\Omega^1(B)$ we would need
\begin{align}\label{diff coact1}
\Delta_{R*}(b\extd_B c) = b^{\barnot}\extd_B c^{\barnot} \tens b^{\baro}c^{\baro} + b^{\barnot}c^{\barnot} \tens b^{\baro} \extd_A c^{\baro}
\end{align}
for all $b,c\in B$, where the first term entails a well-defined map  $\Delta_R : \Omega^1(B)\to \Omega^1(B) \tens A$ (which is then necessarily a coaction of $A$) and the second term entails a well-defined map 
\[\delta_R(b \extd_B c) = b^{\barnot}c^{\barnot} \tens b^{\baro} \extd_A c^{\baro}, \quad  \delta_R : \Omega^1(B) \to B \tens \Omega^1(A)\]
for all $b,c\in B$. In general degree the assumption similarly entails that there is a degree preserving map $\Delta_R:\Omega(B)\to \Omega(B)\tens A$ which one can show is a coaction making $\Omega(B)$ an $A$-covariant exterior algebra, while the stronger property of being differentiable entails all the other components of $\Delta_{R*}$ also being well defined. If (as often happens) we are already given a covariant calculus, i.e.  $\Omega(B)$ as an $A$-comodule algebra with $\extd_B$ a comodule map, we also say that the coaction $\Delta_R$ on $\Omega(B)$ is `differentiable' if the further data for the other components of $\Delta_{R*}$ exist.  

\begin{lemma}\label{maxDeltaR}\cite[Lemma 4.29]{QRG} Let $B$ be a right $A$-comodule algebra and assume that the map $\Delta_{R*} : \Omega^1(B) \to (\Omega^1(B)\tens A) \oplus (B \tens \Omega^1(A))$  in (\ref{diff coact1}) is well defined and that $\Omega(B)$ is the maximal prolongation of $\Omega^1(B)$.  Then $\Delta_R$ is differentiable.
\end{lemma}
\begin{proof} This is in \cite{QRG} but to be self-contained, we give our own short proof. In fact, it suffices to prove that $\Delta_{R*}$ extends to $\Omega^2(B)$ since the maximal prolongation is quadratic. So we need to check that $\Delta_{R*}(\xi\eta) = \Delta_{R*}(\xi)\Delta_{R*}(\eta)$ is well defined for $\xi, \eta \in \Omega^1(B)$. Suppose we have the relation $b \extd_B c =0$ in $\Omega^1(B)$ (sum of such terms understood) which implies that $\extd_B b\extd_B c =0 \in \Omega^2(B)$. Applying $\Delta_{R*}$ to the relation in  $\Omega^1(B)$ we have
\[b^{\barnot} \extd_B c^{\barnot} \tens b^{\baro}c^{\baro}=0, \quad b^{\barnot} c^{\barnot} \tens b^{\baro}\extd_A c^{\baro}=0.\]
Applying $\id \tens \extd_A$ to the first equation and $\extd_B \tens \id$ to the second equation then subtracting them gives us
\begin{align*}
b^{\barnot}&\extd_B c^{\barnot} \tens (\extd_A b^{\baro}) c^{\baro} - (\extd_B b^{\barnot})c^{\barnot} \tens b^{\baro}\extd_A c^{\baro}=0,
\end{align*}
which is the $\Omega^1(B) \tens \Omega^1(A)$ part of $\Delta_{R*}(\extd_B b \extd_B c)$. Applying $\extd_B \tens \id$ to the first equation gives
\[\extd_B b^{\barnot} \extd_B c^{\barnot} \tens b^{\baro}c^{\baro}=0\]
which is the $\Omega^2(B)\tens \id$ part. Finally, applying $\id \tens \extd_A$ to the second equation gives
\[ b^{\barnot} c^{\barnot} \tens \extd_A b^{\baro}\extd_A c^{\baro} =0\]
which is the $B \tens \Omega^2(A)$ part. Since all relations in the maximal prolongation are sent to zero then $\Delta_{R*}$ extends to $\Omega^2(B)$, which completes the proof.
\end{proof}

There is an equally good left-handed definition of differentiable coaction, where the left coaction $\Delta_L :  B \to A\tens B$ extends to a degree preserving map $\Delta_{L*} : \Omega(B)\to \Omega(A)\underline{\tens} \Omega(B)$ of exterior algebra. 

Although the above applies more generally, we now specialise to the case where $\Omega(A)$ is {\em strongly bicovariant}, which says in terms of Definition~\ref{defdif} that the coproduct of $A$ viewed as a left or right coaction, is differentiable. In this case $\Omega(A)$ is a super-bialgebra or super-Hopf algebra (the latter if $A$ is a Hopf algebra) and $\Delta_{R*}$ in Lemma~\ref{maxDeltaR} then makes $\Omega(B)$ an $\Omega(A)$-supercomodule algebra. At least in this context, we are now motivated to introduce the partially dual notion of a differentiable action on an algebra $B$ with exterior algebra $\Omega(B)$. This time the extension entails an action of $A$ making $\Omega(B)$ into an $A$-module algebra. Unlike the previous case, we take this in its entirely as the initial data. 

\begin{definition}\label{defactdif} Let $A$ be a bialgebra or Hopf algebra with $\Omega(A)$ strongly bicovariant. If an exterior algebra $\Omega(B)$ is an $A$-module algebra,  the action $\ra : \Omega(B)\tens A \to \Omega(B)$ is called {\em differentiable} if it extends to a degree-preserving map $\ra : \Omega(B)\underline{\tens} \Omega(A)\to \Omega(B)$  making $\Omega(B)$ an $\Omega(A)$-supermodule algebra such that  $\extd_B \ra=\ra \extd_{B\tens A}$. 
\end{definition}

The last condition explicitly is
\begin{align}
\extd_B (\eta \ra \omega) = (\extd_B \eta) \ra \omega + (-1)^{|\eta|}\eta \ra (\extd_A \omega) \label{diff act form}
\end{align}
for all $\eta \in \Omega(B)$ and $\omega\in \Omega(A)$.  

Given an $A$-module algebra structure of $\Omega(B)$, the extension to $\Omega(A)$ for a differentiable action is uniquely determined. For instance, on $\Omega^1(B)$  we require 
\[\extd_B(b\ra a) = (\extd_B b) \ra a + b\ra \extd_A a, \quad \extd_B((\extd_B b)\ra a) = -(\extd_B b)\ra \extd_A a\]
where $(\extd_B b) \ra a$ is given as is $\extd_B(b\ra a)$, hence $b\ra \extd_A a$ is determined, and hence also $\ra : B \tens \Omega^1(A)\to \Omega^1(B)$. Similarly, the second equation specifies $\ra : \Omega^1(B)\underline{\tens}\Omega^1(A)\to \Omega^2(B)$. We also require the supermodule algebra axiom, for example
\begin{align*}
(\eta \xi) \ra \extd_A a = (\eta \ra a \o) (\xi \ra \extd_A a\t) +(-1)^{|\xi|} (\eta\ra \extd_A a\o) (\xi \ra a\t)
\end{align*}
for all $\eta,\xi \in \Omega(B)$ and $a\in A$, hence $\Omega(B)\underline{\tens}\Omega^1(A)\to \Omega(B)$ is specified. We could have said that an action of $A$ on $B$ was differentiable if all of the above exists but we have not done so since the extension would not be uniquely determined from such a starting point. Also note that $\extd_B$ is not required to be an $A$-module map, so this data already differs from the idea that $\Omega(B)$ is covariant in a dual sense to the usual comodule notion. 

\begin{lemma}\label{diff right action} Let $A$ be a bialgebra or Hopf algebra and $\Omega(A)$ a maximal prolongation of a bicovariant $\Omega^1(A)$. If an exterior algebra $\Omega(B)$ is an $A$-module algebra and its action $\ra$ extends to a well-defined map $\ra : \Omega(B)\tens \Omega^1(A) \to \Omega(B)$  by
\[ \eta \ra ((\extd_A a)c) := (\eta \ra \extd_A a)\ra c,\quad \eta \ra \extd_A a := (-1)^{|\eta|}\big(\extd_B(\eta \ra a) -(\extd_B\eta) \ra a \big)\]
for all $\eta \in \Omega(A)$ and $a,c \in A$, then the original action $\ra$ is differentiable. \end{lemma}
\begin{proof}
First we check that $\ra : \Omega(B)\tens \Omega^1(A) \to \Omega(B)$ if defined as above gives an action of $\Omega^1(A)$ in the sense
\begin{align*}
\eta \ra (a\extd_A c) =& \eta \ra (\extd_A(ac)-(\extd_A a)c)=(-1)^{|\eta|}\extd_B(\eta \ra ac) - (-1)^{|\eta|}(\extd_B \eta)\ra ac - (\eta\ra \extd_A a)\ra c\\
=& (-1)^{|\eta|}\extd_B((\eta \ra a)\ra c)- (-1)^{|\eta|}((\extd_B \eta)\ra a)\ra c - (\eta\ra \extd_A a)\ra c\\
=&(-1)^{|\eta|}(\extd_B(\eta\ra a))\ra c + (\eta \ra a)\ra \extd_A c -(-1)^{|\eta|}((\extd_B \eta)\ra a)\ra c - (\eta \ra \extd_A a)\ra c\\
=&(-1)^{|\eta|}((\extd_B \eta)\ra a)\ra c + (\eta \ra \extd_A a)\ra c + (\eta \ra a)\ra \extd_A c -(-1)^{|\eta|}((\extd_B \eta)\ra a)\ra c - (\eta \ra \extd_A a)\ra c \\
=& (\eta \ra a ) \ra \extd_A c.
\end{align*}

We now suppose that $\Omega(A)$ is the maximal prolongation of $\Omega^1(A)$ and show that we can extend the above to a right action of $\Omega(A)$ of all degrees. The higher relations relations are quadratic of the form $\extd_A a \extd_A c=0$ for  $a\extd_A c=0$ (a sum of such terms understood) and we check that
\begin{align*}
0 = & \extd_B (\eta \ra (a\extd_A c)) = (\extd_B \eta) \ra (a\extd_A c) + (-1)^{|\eta|} \eta \ra (\extd_A a \extd_A c) = (-1)^{|\eta|} \eta \ra (\extd_A a \extd_A c)
\end{align*}
as required. Indeed this gives an action of $\Omega^2(A)$ in the sense
\begin{align*}
\eta \ra (\extd_A a \extd_A c) =& \eta \ra \extd_A(a\extd_A c) = (-1)^{|\eta|}(\extd_B(\eta\ra a\extd_A c) - (\extd_B \eta)\ra a\extd_A c)\\
=&(-1)^{|\eta|}\big(\extd_B((\eta\ra a)\ra \extd_A c) - ((\extd_B \eta)\ra a)\ra \extd_A c\big) = (\eta\ra \extd_A a)\ra \extd_A c,
\end{align*}
and since $\Omega(A)$ is the maximal prolongation of $\Omega^1(A)$, then there is no new relations in higher degree and thus $\ra$ can be extended further to be an action of $\Omega(A)$ of all degrees, making $\Omega(H)$ a super right $\Omega(A)$-module.

Next we check that $\Omega(B)$ is a super right $\Omega(A)$-module algebra with regard to the action of $\Omega^1(A)$,
\begin{align*}
(\eta\xi) \ra \extd_A a =& (-1)^{|\eta\xi|} (\extd_B((\eta\xi) \ra a)-(\extd_B(\eta\xi))\ra a)\\
=& (-1)^{|\eta|+|\xi|} \big(\extd_B((\eta \ra a\o)(\xi \ra a\t)) - ((\extd_B \eta)\xi) \ra a -(-1)^{|\eta|}(\eta \extd_B \xi) \ra a \big)\\
=&(-1)^{|\eta|+|\xi|}(\extd_B(\eta \ra a\o))(\xi \ra a\t) + (-1)^{|\xi|}(\eta \ra a\o)\extd_B(\xi \ra a\t)\\
& -(-1)^{|\eta|+|\xi|}((\extd_B\eta)\xi) \ra a-(-1)^{|\xi|}(\eta\extd_B \xi)\ra a\\
=&(-1)^{|\eta|+|\xi|}((\extd_B \eta) \ra a\o)(\xi \ra a\t) + (-1)^{|\xi|}(\eta \ra \extd_A a\o)(\xi \ra a\t)\\
&+(-1)^{|\xi|}(\eta \ra a\o)((\extd_B \xi) \ra a\t)+(\eta \ra a\o)(\xi \ra \extd_A a\t)\\
&-(-1)^{|\eta|+|\xi|}((\extd_B \eta)\ra a\o)(\xi \ra a\t) - (-1)^{|\xi|}(\eta \ra a\o)((\extd_B \xi)\ra a\t)\\
=&(\eta \ra a\o)(\xi \ra \extd_A a\t)+(-1)^{|\xi|}(\eta \ra \extd_A a\o)(\xi \ra a\t)
\end{align*}
where we see the coproduct $\Delta_*(\extd_A a)=a\o\tens\extd_A a\t+\extd_A a\o\tens a\t$ of $\Omega(A)$. 
As a direct consequence, one can find that
\begin{align*}
(\eta \xi)\ra ((\extd_A a)c) =& (\eta \ra (a\o c\o))(\xi \ra ((\extd_A a\t)c\t))+(-1)^{|\xi|}(\eta \ra ((\extd_A a\o)c\o))(\xi \ra (a\t c\t))\\
(\eta \xi) \ra (a \extd_A c) =& (\eta \ra (a\o c\o))(\xi \ra (a\t \extd_A c\t))+(-1)^{|\xi|}(\eta \ra (a\o \extd_A c\o))(\xi \ra (a\t c\t))
\end{align*}
as required for the action of general elements of $\Omega^1(A)$. Since $\Omega(A)$ is the maximal prolongation of $\Omega^1(A)$, then by taking $a\extd_A c =0$, one can check that
\begin{align*}
0=&\extd_B \big((\eta\xi)\ra (a\extd_A c)\big)=(\extd_B(\eta\xi)) \ra (a\extd_A c) + (-1)^{|\eta|+|\xi|}(\eta\xi)\ra (\extd_A a \extd_A c)
\end{align*}
implying $(\eta\xi)\ra (\extd_A a \extd_A c)=0$, making $\Omega(B)$ a super right $\Omega(A)$-module algebra.
\end{proof}

There is an equally good left-handed definition of differentiable action, where a left action $\la : A\tens B \to B$ extends to $\la : \Omega(A)\underline{\tens} \Omega(B)\to \Omega(B)$ making $\Omega(B)$ a left $\Omega(A)$-super-module algebra.

\subsection{Differentials on the FRT bialgebra}\label{secFRT}
We recall that the FRT-bialgebra\cite{FRT} $A(R)$ over a field $k$ is generated by $\bt = (t^i{}_j)$ such that
\begin{equation*}
R\bt_1\bt_2 = \bt_2\bt_1 R, \quad \Delta \bt = \bt \tens \bt, \quad \epsilon\bt = \id, 
\end{equation*}
where $R=(R^i{}_j{}^r{}_s)\in M_n(k)\tens M_n(k)$ where $(i,j)$ label the first copy and $(r,s)$ the second. In  the compact notation used, the numerical suffixes indicate the position in the tensor matrix product, e.g. $\bt_1 = \bt\tens \id$, $\bt_2= \id \tens \bt$, $R_{23} = \id \tens R$ etc. We ask for $R$ to satisfy the Yang-Baxter equation $R_{12}R_{13}R_{23} = R_{23}R_{13}R_{12}$ (equivalent to the braid relation) and for the present purpose to be  {\em $q$-Hecke}, which means that it satisfies
\begin{align}
(PR-q)(PR+q^{-1})=0,\label{R q-Hecke cond}
\end{align}
where $q\in k$, $q\ne 0$, and $P=(P^i{}_j{}^r{}_s)$ is the permutation matrix $P^i{}_j{}^r{}_s = \delta^i{}_s \delta^r{}_j$ in terms of the Kronecker delta or identity matrix. The $q$-Hecke condition is equivalent to  $R_{21}R = \id + (q-q^{-1})PR$ where  $R_{21}=PRP$ has the tensor factors swapped. Note that $-R_{21}^{-1}$ is also $q$-Hecke and gives the same $A(R)$ but `conjugate' constructions to those below.

It is already proven in \cite[Prop.~10.5.1]{Ma:book} that $A(R)$ in the $q$-Hecke case is an additive  braided Hopf algebra in the braided category of $A(R)^{\rm cop}\tens A(R)$-right comodules (or $A(R)$-bicomodules) with coproduct $\underline{\Delta}\bt=\bt\tens 1+1\tens\bt$, or in a compact notation  $\bt'' = \bt' + \bt$ where $\bt'$ is a second copy of $\bt$, and $\bt''$ obeys FRT-bialgebra relation provided $\bt'_1 \bt_2 = R_{21} \bt_2 \bt'_1 R$. This expresses the braided Hopf algebra homomorphism property of the coproduct with respect to the relevant braiding $\Psi(\bt_1\tens\bt_2)=  R_{21} \bt_2\tens \bt_1 R$, see \cite{Ma:book} for details. As a consequence of the additive braided-Hopf algebra structure, $A(R)$  has a bicovariant exterior algebra $\Omega(A(R))$ generated by $\bt$ and $\extd\bt$ with bimodule relations as in the next lemma. The new part is that  this makes $\Omega(A(R))$ strongly bicovariant. 

\begin{lemma}\label{FRTcalc}
Let $A(R)$ be the FRT-bialgebra with $R$ $q$-Hecke. The exterior algebra $\Omega(A(R))$ with bimodule and exterior algebra relations  $(\extd \bt_1)\bt_2 = R_{21}\bt_2\extd \bt_1 R$ and $\extd \bt_1\extd \bt_2 = -R_{21}\extd \bt_2\extd \bt_1 R$ as in \cite[Chapter~10.5.1]{Ma:book} is strongly bicovariant with  $\Delta_* \extd \bt = \extd \bt \tens \bt + \bt \tens \extd \bt$.
\end{lemma}
\begin{proof} The exterior algebra was already constructed in \cite{Ma:book}, but we provide a short check of the Leibniz rule so as to be self-contained.  Thus, on generators, 
\begin{align*}
\extd (R &\bt_1 \bt_2) = R ((\extd \bt_1)\bt_2 +\bt_1 \extd \bt_2)= RR_{21}\bt_2 \extd \bt_1 R + R\bt_1 \extd \bt_2\\
=&\bt_2 \extd \bt_1 R + (q-q^{-1})RP\bt_2 \extd\bt_1 R + R\bt_1 \extd \bt_2=R\bt_1 \extd \bt_2 (\id + (q-q^{-1})PR)+\bt_2 \extd \bt_1 R\\
=& R \bt_1 \extd \bt_2 R_{21}R + \bt_2 \extd \bt_1 R=(\extd \bt_2)\bt_1 R + \bt_2 \extd \bt_1  R=\extd(\bt_2 \bt_1 R)
\end{align*}
where we used $R_{21} = R^{-1}+(q-q^{-1})P$ for the 3rd equality and  $\bt_1 \extd \bt_2 P = P\bt_2 \extd \bt_1$ for the 4th. Applying $\extd$ once more to the stated bimodule relation on degree 1 gives the stated
relations in degree 2 and there are no further relations in higher degree, which  means that $\Omega(A(R))$ is the maximal prolongation of $\Omega^1(A(R))$. 

The new part is the super-coproduct $\Delta_*$, which is uniquely determined by the super-coderivation property but we need to check that it is well-defined. Thus, 
\begin{align*}
\Delta_*((\extd \bt_1) \bt_2)=&(\extd \bt_1)\bt_2 \tens \bt_1\bt_2 + \bt_1\bt_2 \tens  (\extd\bt_1)\bt_2 = R_{21}\bt_2 \extd \bt_1 R \tens \bt_1\bt_2 + \bt_1 \bt_2 \tens R_{21}\bt_2\extd \bt_1 R\\
=&R_{21}\bt_2\extd\bt_1 \tens \bt_2 \bt_1 R + R_{21}\bt_2\bt_1 \tens \bt_2 \extd \bt_1 R = \Delta_*(R_{21}\bt_2 \extd \bt_1 R)
\end{align*}
on $\Omega^1(A(R))$. Since $\Omega(A(R))$ is the maximal prolongation, we do not in principle need to
check the relations in higher degrees due to arguments similar to the proof of Lemma~\ref{maxDeltaR}. 
In practice, however, we check the degree 2 relations explicitly. Thus
\begin{align*}
\Delta_*(-\extd\bt_1&\extd\bt_2)= -\extd\bt_1\extd\bt_2\tens \bt_1\bt_2 - (\extd\bt_1)\bt_2 \tens \bt_1\extd\bt_2 + \bt_1\extd\bt_2 \tens(\extd\bt_1)\bt_2 - \bt_1\bt_2\tens\extd\bt_1\extd\bt_2\\
=&R_{21}\extd\bt_2\extd\bt_1 R\tens \bt_1\bt_2 -R_{21}\bt_2\extd\bt_1 R\tens \bt_1\extd\bt_2 + \bt_1\extd\bt_2 \tens R_{21}\bt_2\extd\bt_1 R+\bt_1\bt_2\tens R_{21}\extd\bt_2\extd\bt_1 R\\
=&R_{21}\extd\bt_2\extd\bt_1 \tens \bt_2\bt_1 R -R_{21}\bt_2\extd\bt_1 R\tens \bt_1\extd\bt_2 + \bt_1\extd\bt_2 \tens R_{21}\bt_2\extd\bt_1 R+R_{21}\bt_2\bt_1\tens\extd\bt_2\extd\bt_1 R
\end{align*}
\begin{align*}
\Delta_*(R_{21}\extd&\bt_2\extd\bt_1 R)= R_{21}\big(\extd\bt_2\extd\bt_1\tens \bt_2\bt_1  + \extd\bt_2.\bt_1 \tens \bt_2\extd\bt_1  -\bt_2\extd\bt_1\tens \extd\bt_2.\bt_1 + \bt_2\bt_1\tens \extd\bt_2\extd\bt_1\big) R\\
=& R_{21}\extd\bt_2\extd\bt_1\tens \bt_2\bt_1 R + R_{21}R\bt_1\extd\bt_2 R_{21}\tens \bt_2\extd\bt_1 R - R_{21}\bt_2\extd\bt_1\tens R\bt_1\extd\bt_2 R_{21}R\\
&+R_{21}\bt_2\bt_1\tens \extd\bt_2\extd\bt_1R\\
=&R_{21}\extd\bt_2\extd\bt_1\tens \bt_2\bt_1 R + (\id + (q-q^{-1})PR)\bt_1\extd\bt_2 R_{21}\tens \bt_2\extd\bt_1 R\\
&-R_{21}\bt_2\extd\bt_1\tens R\bt_1\extd\bt_2(\id+(q-q^{-1})PR)+R_{21}\bt_2\bt_1\tens \extd\bt_2\extd\bt_1 R\\
=& R_{21}\extd\bt_2\extd\bt_1\tens \bt_2\bt_1 R +\bt_1\extd\bt_2\tens R_{21}\bt_2\extd\bt_1R + (q-q^{-1})PR\bt_1\extd\bt_2R_{21}\tens \bt_2\extd\bt_1 R\\
&-R_{21}\bt_2\extd\bt_1R \tens \bt_1\extd\bt_2 -(q-q^{-1})R_{21}\bt_2\extd\bt_1  \tens R\bt_1\extd\bt_2 PR +R_{21}\bt_2\bt_1\tens \extd\bt_2\extd\bt_1 R. 
\end{align*}
The two expressions are equal since 
\begin{align*}
PR\bt_1\extd\bt_2 R_{21}\tens \bt_2\extd\bt_1 R =&R_{21}P\bt_1\extd\bt_2\tens \extd\bt_1.\bt_2 = R_{21}\bt_2\extd\bt_1\tens P\extd\bt_1.\bt_2 = R_{21}\bt_2\extd\bt_1\tens \extd\bt_2.\bt_1 P 
\\=& R_{21}\bt_2\extd\bt_1\tens R\bt_1\extd\bt_2R_{21}P = R_{21}\bt_2\extd\bt_1\tens R\bt_1\extd\bt_2 PR
\end{align*}
so that $(q-q^{-1})(PR\bt_1\extd\bt_2 R_{21}\tens \bt_2\extd \bt_1 R - R_{21}\bt_2\extd\bt_1  \tens R\bt_1\extd\bt_2 PR)$ vanishes. \end{proof}

Now consider the additive braided-Hopf algebra $V(R)$ on which $A(R)$ right coacts. This is 
the $q$-Hecke case of the general construction of `braided linear spaces' in \cite[Prop.~10.2.8]{Ma:book} and is generated by
$\bx = (x_i)$  regarded as a vector row with relations $q \bx_1 \bx_2 = \bx_2 \bx_1 R$, coaction $\Delta_R \bx = \bx\tens \bt$ and coproduct $\underline{\Delta}\bx=\bx\tens 1+1\tens\bx$, or in a compact notation  $\bx'' = \bx' + \bx$ where $\bx'$ is a second copy of $\bx$ and $\bx''$ obeys the relation of $V(R)$ provided $\bx'_1 \bx_2 = \bx_2 \bx'_1 qR$. This expresses the braided Hopf algebra homomorphism property of the coproduct with respect to the relevant braiding $\Psi(\bx_1\tens\bx_2)= \bx_2\tens \bx_1 qR$ induced by the coquasitriangular structure $\CR(\bt_1\tens\bt_2)=qR$, see \cite[Thm.~10.2.6]{Ma:book} for details.  As before, the additive braided-Hopf algebra theory implies an exterior algebra $\Omega(V(R))$ given at the end of \cite[Chapter 10.4]{Ma:book} with the relations shown in the next lemma.  A more formal treatment of the exterior algebras on an additive braided Hopf algebra, which underlies both $\Omega(A(R))$ and $\Omega(V(R))$, recently appeared in \cite[Prop.~2.9]{Ma:hod}. The new part now is that $\Delta_R$ is differentiable. 

\begin{lemma}\label{diff coact A(R) on V(R)} Let $V(R)$ be the right $A(R)$-covariant braided Hopf algebra as above with $R$ $q$-Hecke. The exterior algebra $\Omega(V(R))$ with bimodule and exterior algebra relations $(\extd \bx_1) \bx_2=\bx_2 \extd \bx_1 qR$ and $-\extd \bx_1 \extd \bx_2 = \extd \bx_2\extd \bx_1 qR$ as in \cite{Ma:book} has differentiable right coaction with $\Delta_{R*} \extd \bx = \extd \bx \tens \bt + \bx \tens \extd \bt$.
\end{lemma}
\begin{proof}
That $\Delta_{R*}$ is well-defined on degree 1 is
\begin{align*}
\Delta_{R*}((\extd \bx_1) \bx_2) =&(\extd \bx_1) \bx_2 \tens \bt_1\bt_2 + \bx_1\bx_2 \tens (\extd \bt_1)\bt_2= \bx_2\extd\bx_1 qR\tens \bt_1\bt_2 + \bx_1\bx_2 \tens R_{21}\bt_2\extd \bt_1 R\\
=& \bx_2\extd \bx_1 \tens qR\bt_1\bt_2 + \bx_1\bx_2 R_{21} \tens \bt_2\extd \bt_1 R= \bx_2 \extd \bx_1 \tens \bt_2\bt_1 qR + \bx_2\bx_1 \tens \bt_2\extd\bt_1 qR\\
=& \Delta_{R*}(\bx_2 \extd \bx_1 qR).
\end{align*}
This is sufficient by Lemma~\ref{maxDeltaR} since $\Omega(V(R))$ is the maximal prolongation of $\Omega^1(V(R))$. If one wants to see it explicitly on degree 2, this is
\begin{align*}
\Delta_{R*}(-\extd \bx_1&\extd \bx_2)=-\extd \bx_1 \extd \bx_2 \tens \bt_1 \bt_2 - (\extd \bx_1) \bx_2 \tens \bt_1 \extd \bt_2 + \bx_1 \extd \bx_2 \tens (\extd \bt_1)\bt_2 - \bx_1\bx_2 \tens \extd \bt_1\extd \bt_2\\
=& \extd \bx_2 \extd \bx_1 qR \tens \bt_1 \bt_2 - q\bx_2 \extd\bx_1 R \tens \bt_1\extd\bt_2+\bx_1\extd\bx_2\tens R_{21}\bt_2\extd\bt_1 R + \bx_1\bx_2 \tens R_{21}\extd \bt_2 \extd \bt_1 R\\
=&\extd \bx_2 \extd \bx_1  \tens \bt_2 \bt_1 qR - q\bx_2 \extd\bx_1 R \tens \bt_1\extd\bt_2+\bx_1\extd\bx_2\tens R_{21}\bt_2\extd\bt_1 R + \bx_2\bx_1 \tens\extd \bt_2 \extd \bt_1 qR.
\end{align*}
\begin{align*}
\Delta_{R*}(\bx_2 &\extd \bx_1 qR)= (\extd \bx_2 \extd \bx_1 \tens \bt_2 \bt_1 + \extd \bx_2.\bx_1 \tens \bt_2 \extd \bt_1 - \bx_2 \extd \bx_1 \tens (\extd \bt_2)\bt_1 + \bx_2 \bx_1 \tens \extd \bt_2 \extd \bt_1)qR\\
=&\extd \bx_2 \extd \bx_1 \tens \bt_2 \bt_1 qR + q^2\bx_1 \extd \bx_2\tens R_{21}\bt_2\extd\bt_1 R - q\bx_2\extd \bx_1 \tens R\bt_1 \extd \bt_2 R_{21}R+ \bx_2 \bx_1 \tens \extd \bt_2 \extd \bt_1 qR\\
=&\extd \bx_2 \extd \bx_1 \tens \bt_2 \bt_1 qR + q^2\bx_1 \extd \bx_2\tens R_{21}\bt_2\extd\bt_1 R -q \bx_2 \extd \bx_1 \tens R\bt_1\extd \bt_2 (\id + (q-q^{-1})PR)\\
&+ \bx_2 \bx_1 \tens \extd \bt_2 \extd \bt_1 qR\\
=&\extd \bx_2 \extd \bx_1 \tens \bt_2 \bt_1 qR + q^2\bx_1 \extd \bx_2\tens R_{21}\bt_2\extd\bt_1 R -q \bx_2 \extd \bx_1 \tens R\bt_1\extd \bt_2 + \bx_2 \bx_1 \tens \extd \bt_2 \extd \bt_1 qR \\
&-(q^2-1)\bx_2\extd \bx_1 PR_{21} \tens\bt_2 \extd\bt_1 R
\end{align*}
which simplifies to the first expression. 
\end{proof}

We will suppose later that $A(R)$ has a central grouplike element $D$ such that $A=A(R)[D^{-1}]$ is a Hopf algebra. For the standard $\C_q[GL_n]$ $R$-matrix, this is just the $q$-determinant allowing $S\bt$ to be constructed as $D^{-1}$ times a $q$-matrix of cofactors. (Another approach, which we will not take, is to assume that $R$ is `bi-invertible' and define a Hopf algebra by reconstruction from a rigid braided category defined by $R$.) It is easy to see that the above $\Omega(A(R))$ extends to a strongly bicovariant exterior algebra $\Omega(A)$ with $\extd D^{-1}=-D^{-1}(\extd D)D^{-1}$. 


\section{Differentials on double (co)cross products}\label{secdouble}

Here we briefly cover the quantum group construction $\dcross$ from \cite{Ma:phy} and its dual $\codcross$ in the notation of \cite{Ma:book}. The former includes the Drinfeld double of a Hopf as generalised in \cite{Ma:phy} to the case of dually paired Hopf algebras.  

\subsection{Exterior algebras by super double cross product}

Let $A, H$ be bialgebras or Hopf algebras with $H$ a right $A$-module coalgebra by $\ra : H \tens A \to H$ and $A$ a left $H$-module coalgebra by $\la : H\tens A \to A$. We suppose that $\ra$ and $\la$ are compatible as in \cite{Ma:phy, Ma:book} such that they form a double cross product Hopf algebra $A\dcross H$.

Now let $\Omega(A)$ and $\Omega(H)$ be strongly bicovariant exterior algebras, and let $\ra$ and $\la$ extend to $\ra : \Omega(H)\underline{\tens}\Omega(A)\to \Omega(H)$ and $\la : \Omega(H)\underline{\tens}\Omega(A)\to \Omega(A)$ as module coalgebras such that
\begin{align}
\extd_H(\eta \ra \omega) =& (\extd_H \eta) \ra \omega + (-1)^{|\eta|}\eta \ra \extd_A \omega\label{d_H dcross},\\
\extd_A(\eta \la \omega)=&(\extd_H \eta)\la \omega + (-1)^{|\eta|}\eta \la \extd_A \omega\label{d_A dcross}
\end{align}
for all $\eta \in \Omega(H)$ and $\omega \in \Omega(A)$. If $\ra$ and $\la$ obey the super double cross product conditions
\begin{align}
1_H\ra \omega = & 1_H\epsilon(\omega), \quad  \eta \la 1_A =  1_A\epsilon(\eta) \label{dcross1}\\
(\eta \xi)\ra \omega = & (-1)^{|\omega\o||\xi\t|} (\eta \ra (\xi\o \la \omega\o))(\xi\t \ra \omega\t)\label{dcross2} \\
\eta \la (\omega \tau) = &(-1)^{|\omega\o||\eta\t|}(\eta\o \la \omega\o)((\eta\t\ra \omega\t)\la \tau) \label{dcross3}\\
(-1)^{|\omega\o||\eta\t|}&\eta\o \ra \omega\o \tens \eta\t \la \omega\t \label{dcross4}\\
=& (-1)^{|\eta\o|(|\eta\t|+|\omega\t|)+|\omega\o||\omega\t|}\eta\t \ra \omega\t \tens \eta\o \la \omega\o \notag 
\end{align}
then we have a super double cross product bialgebra or Hopf algebra $\Omega(A)\dcross \Omega(H)$ with super tensor product coalgebra and product
\[(\omega\tens \eta)(\tau \tens \xi) = (-1)^{|\eta\t||\tau\o|}\omega(\eta\o \la \tau\o)\tens (\eta\t \ra \tau\t)\xi\]
for all $\eta,\xi \in \Omega(H)$ and $\omega,\tau \in \Omega(A)$. We omit the proof that $\Omega(A)\dcross\Omega(H)$ is a super Hopf algebra since this is similar to the usual version\cite{Ma:phy,Ma:book} with extra signs.

\begin{theorem}\label{thm super dcross}
Let $A,H$ be bialgebras or Hopf algebras forming a double cross product $A\dcross H$ and let $\Omega(A)$ and $\Omega(H)$ be strongly bicovariant with $\ra, \la$ obeying (\ref{d_H dcross})-(\ref{dcross4}). Then $\Omega(A\dcross H):= \Omega(A)\dcross \Omega(H)$ is a strongly bicovariant exterior algebra on $A\dcross H$ with differential
\begin{align*}
\extd(\omega\tens \eta)= \extd_A \omega \tens \eta + (-1)^{|\omega|}\omega\tens \extd_H \eta
\end{align*} 
for all $\omega\in \Omega(A)$ and $\eta\in \Omega(H)$.
\end{theorem}
\begin{proof} 

Since $\extd(\omega\tens 1)=\extd_A \omega$ and $\extd (1\tens \eta)= \extd_H \eta$ for all $\eta \in \Omega(H)$ and $\omega \in \Omega(A)$, we show that the graded Leibniz rule holds as  
\begin{align*}
\extd(\eta\omega)=& \extd((1\tens \eta)(\omega\tens 1))\\
=&(-1)^{|\eta\t||\omega\o|}\big((\extd_H\eta\o)\la \omega\o \tens \eta\t \ra \omega\t + (-1)^{|\eta\o|+|\omega\o|} \eta\o \la\omega\o \tens (\extd_H\eta\t)\ra \omega\t \big)\\
&+(-1)^{|\eta\t||\omega\o| +|\eta\o|}\big(\eta\o \la \extd_A \omega\o \tens \eta\t \ra \omega\t + (-1)^{|\eta\t|+|\omega\o|} \eta\o \la \omega\o \tens \eta\t \ra \extd_A \omega\t \big)\\
=& (-1)^{|(\extd_H \eta)\t||\omega\o|} (\extd_H \eta)\o \la \tens \omega\o \tens (\extd_H \eta)\t \ra \omega\t \\
&+ (-1)^{|\eta|+|\eta\t||(\extd_A\omega)\o|} \eta\o \la (\extd_A \omega)\o \tens \eta\t \ra (\extd_A \omega)\t\\
=&(1\tens \extd_H \eta)(\omega\tens 1)+(-1)^{|\eta|}(1\tens \eta)(\extd_A \omega\tens 1) = (\extd \eta)\omega + (-1)^{|\eta|}\eta\extd\omega.
\end{align*}
Clearly $\extd^2=0$ and thus $\Omega(A)\dcross \Omega(H)$ is a DGA. Finally since $\Delta_*$ is a super tensor coproduct as in Lemma \ref{lemma diff tens prod}, $\extd$ is a super-coderivation.
\end{proof}

\begin{remark}
If $A$ is finite dimensional, it is known cf.\cite{Ma:book} that $A\dcross H$ acts on $A^*$ as a module algebra by
\[(\phi\ra h)(a) = \phi(h\la a), \quad \phi \ra a = \langle \phi\o, a \rangle \phi\t,\]
for all $\phi \in A^*$, $a\in A$, and $h\in H$. Similarly for a left action on $H^*$. However, for differentiability, we would need $\Omega(A^*)$ or $\Omega(H^*)$ to be specified. We focus in the next section on a special case where this is not a problem. 
\end{remark}

\subsection{Exterior algebra on generalised quantum doubles $D(A,H)$}
Let $A, H$ be dually paired Hopf algebras with duality pairing $\langle \ , \ \rangle : H\tens A \to k$. Then, $A$ acts on $H$ and $H$ acts on $A$ by the following actions
\[h \ra a = h\t \langle Sh\o, a\o\rangle \langle h\th, a\t \rangle, \quad h\la a = a\t \langle Sh\o, a\o \rangle \langle h\t , a\th \rangle,\]
and one has a generalised quantum double $D(A,H) = A^{\op}\dcross H$ \cite{Ma:phy,Ma:book}. In this case, the product becomes $(a\tens h)(b\tens g) = \langle Sh\o, b\o \rangle b\t a \tens h\t g \langle h\th, b\th \rangle$, for all $a,b\in A$ and $h,g\in H$. The finite-dimensional case with $A=H^*$ is Drinfeld's $D(H)$ from \cite{Dri}.

Note that if $\Omega(A)$ is a DGA, $\Omega(A)^{\op}$ remains a DGA with the same differential $\extd_A$ and we define $\Omega(A^{\op}):=\Omega(A)^{\op}$. Now let $\Omega(H), \Omega(A)$ be strongly bicovariant exterior algebras and note that the above pairing can be extended to a super Hopf algebra pairing $\langle \ , \ \rangle : \Omega(H)\tens \Omega(A)\to k$ by 0 for degree $\geq 1$. So we have
\begin{align*}
\eta \ra \omega =&\begin{cases} \eta\t \langle S\eta\o, \omega\o \rangle \langle \eta\th, \omega\t \rangle &\text{if $\omega \in A$}\\
0 &\text{otherwise},
\end{cases}\\
\eta \la \omega =&\begin{cases} \omega\t \langle S\eta\o, \omega\o \rangle \langle \eta\t ,\omega\th \rangle &\text{if $\eta \in H$}\\
0 &\text{otherwise}
\end{cases}
\end{align*}
for $\omega\in \Omega(A)$ and $\eta\in \Omega(A)$, where only the parts of $\Delta_*^2 \eta$ in $H \tens \Omega(H) \tens H$ and $\Delta_*^2 \omega$ in $ A\tens \Omega(A)\tens A$ contribute in the respective actions. 

One can check that the above actions obey conditions $(\ref{d_H dcross})-(\ref{dcross4})$ for a super double cross product. Therefore, there is a super double cross product $\Omega(A)^{\op}\dcross \Omega(H)$ with product
\[(\omega \tens \eta)(\tau \tens \xi) = (-1)^{(|\eta|+|\omega|)|\tau|}\langle S\eta\o, \tau\o \rangle \tau\t\omega \tens \eta\t \xi \langle \eta\th, \tau\th \rangle\]
for all $\omega, \tau \in \Omega(A)$, $\eta, \xi \in \Omega(H)$ and super tensor product coalgebra. Note that in the above product, $\tau\t$ and $\omega$ are crossed with $\eta\t$ and  so we should have generated a factor $(-1)^{(|\eta\t|+|\omega|)|\tau\t|}$ but $\langle \ , \ \rangle$ is 0 on degree $\geq 1$, and only the parts of  $\Delta_*^2 \eta$ in $H\tens \Omega(H)\tens H$ and  $\Delta_*^2 \tau$ in $A \tens \Omega(A)\tens A$ contribute, so $|\eta\t|=|\eta|$ and $|\tau\t|=|\tau|$.

\begin{proposition}\label{super D(H)}
Let $A,H$ be dually paired Hopf algebras and $D(A,H)=A^{\op}\dcross H$. Let $\Omega(A)$ and $\Omega(H)$ be dually paired strongly bicovariant exterior algebras as above. Then $ \Omega(D(A,H)):=\Omega(A)^{\op}\dcross \Omega(H)$ is a strongly bicovariant exterior algebra on the generalised quantum double $D(A,H)$ with differential
\[\extd(\omega\tens \eta) = \extd_{A} \omega\tens \eta + (-1)^{|\omega|}\omega\tens \extd_H \eta\]
for all $\eta\in \Omega(H)$ and $\omega\in \Omega(A)$. Moreover, the action of $D(A,H)$ on $H$ defined by $g\ra (a\tens h) = \langle g\t, a \rangle (Sh\o)g\o h\t$ for $g,h\in H$, $a\in A$ is differentiable and extends to an action of $\Omega(D(A,H))$ on $\Omega(H)$ by
\begin{align*}
\xi \ra (\omega \tens \eta) =& \begin{cases}(-1)^{|\eta\o||\xi|}\langle \xi\t, \omega \rangle (S\eta\o)\xi\o \eta\t & \text{if $\omega \in A$ }\\
0 & {otherwise}\end{cases}
\end{align*}
for $\xi,\eta\in \Omega(H)$, where only the part of $\Delta_* \xi$ in $\Omega(H)\tens H$ contributes. Explicitly
\[\xi \ra (a\tens \eta) = (-1)^{|\eta\o||\xi|} \langle \xi^{\baro}, a \rangle (S\eta\o)\xi^{\barnot}\eta\t,\]
where $\Delta_R \xi = \xi^{\barnot}\tens \xi^{\baro}$ is the right coaction of $H$ on $\Omega(H)$. Similarly with left-right reversal for 
a differentiable  left action of $D(A,H)$ on $A$. 
\end{proposition}
\begin{proof}
The first part of the proposition follows from Theorem \ref{thm super dcross}. One can check the stated action makes $\Omega(H)$ a super right $\Omega(D(A,H))$-module algebra. Moreover, $\Delta_* \xi = \Delta_R \xi + \text{terms}$  of higher degree on the second factor, giving the explicit formula stated. We also have
\begin{align*}
\extd_H(\xi \ra (a\tens \eta))=& (-1)^{|\eta\o||\xi|} \langle \xi^{\baro},a\rangle \extd_H((S\eta\o)\xi^{\barnot}\eta\t)\\
=&(-1)^{|\eta\o||\xi|} \langle \xi^{\baro}, a \rangle  \Big((\extd_H S\eta\o)\xi^{\barnot} \eta\t  + (-1)^{|\eta\o|}(S\eta\o)(\extd_H \eta^{\barnot})\eta\t \\
&+ (S\eta\o)\xi^{\barnot}\extd_H \eta\t  \Big)\\
=&(\extd_H \xi)\ra a\tens \eta + (-1)^{|\xi|} \xi \ra \extd(a\tens \eta), 
\end{align*}
where in the last equation we have $\xi \ra \extd(a\tens \eta) = \xi \ra (a \tens \extd_H \eta)$ since $\xi \ra (\extd_A a\tens \eta) =0$. The formulae with $A,H$ and left-right swapped are left to the reader. 
\end{proof}

\begin{example}\label{ex D(U(b_+))}
Let $H = U_q(b_+)$ be a self-dual Hopf algebra generated by $x,t$ with relations, comultiplication, and duality pairing
\[tx=q^2xt, \quad \Delta t = t\tens t, \quad \Delta x = 1\tens x + x\tens t\]
\[\langle t,s \rangle = q^{-2}, \quad \langle x,s \rangle = \langle t,y \rangle=0, \quad \langle x,y \rangle= \dfrac{1}{1-q^2}\]
where $y,s$ are another copies of $x,t$, regarded as generators of $A^{\op} = U_q(b_+)^{\op} = U_{q^{-1}}(b_+)$ and $q\in k$ with $q^2\ne 1$.  
Let $\Omega(U_q(b_+))$ be strongly bicovariant (it will be constructed later in Proposition \ref{ExB+}) with the following bimodule relations and comultiplication
\[(\extd t)t = q^2 t\extd t, \quad (\extd x) x= q^2 x\extd x, \quad (\extd x)t = t\extd x, \quad (\extd t) x = q^2 x\extd t+ (q^2-1)t\extd x\]
\[(\extd t)^2 = (\extd x)^2 =0, \quad \extd t\extd x = -\extd x \extd t, \quad \Delta_* \extd t = \extd t \tens t+ t\tens \extd t, \quad \Delta_* \extd x = 1\tens \extd x + \extd x \tens t + x\tens \extd t.\]
Then $\Omega(D(U_q(b_+)))$ contains $\Omega(U_q(b_+))$ and $\Omega(U_{q^{-1}}(b_+))$ as sub-strongly bicovariant exterior algebras, and the following cross-relations
\[ts=st, \quad ty=q^{-2}yt, \quad xs = q^{-2}sx, \quad xy = q^{-2}yx+\dfrac{1-st}{1-q^2}\]
\[(\extd t) s = s\extd t, \quad (\extd s)t = t\extd s, \quad (\extd x)s= q^{-2}s\extd s, \quad (\extd s) x = q^2 x\extd s, \quad (\extd t)y = q^{-2}y\extd t,\]
\[(\extd y)t = q^2t\extd y, \quad (\extd x)y = q^{-2} y\extd x -\dfrac{s\extd t}{1-q^2}, \quad (\extd y) x =q^2 x\extd y +  \dfrac{t \extd s}{q^{-2}-1},\]
\[\extd t \extd s = -\extd s \extd t, \quad \extd t \extd y = - q^{-2}\extd y \extd t, \quad \extd x \extd s = -q^{-2} \extd s \extd x , \quad \extd x \extd y = -q^{-2}\extd y \extd x - \frac{\extd s \extd t}{1-q^2}.\]
Moreover, $D_q(U(b_+))$ acts differentiably on $U_q(b_+)$.
\end{example}
\begin{proof}
One can check that $\langle \ , \ \rangle$ extends by 0 on degree $\geq 1$. Then the stated crossed relations can be found by direct calculation and one can check that the graded Leibniz rule holds, and $\extd$ is a super-coderivation. By Proposition~\ref{super D(H)}, $U_{q^{-1}}(b_+)\dcross U_q(b_+)$ acts differentiably by
\[t\ra t = t, \quad t\ra x = (1-q^{-2})tx, \quad x \ra t = q^{-2}x, \quad x \ra x = (1-q^{-2})x^2,\]
\[(\extd t) \ra t = q^2 \extd t, \quad (\extd t) \ra x = (q^2-1)t\extd x, \quad (\extd x) \ra t = \extd x, \quad (\extd x) \ra x = (q^2-1)x\extd x, \]
\[t\ra \extd t = (1-q^2)\extd t, \quad t\ra \extd x = (q^2-1)x\extd t, \quad x \ra \extd t = (q^{-2}-1)\extd x, \quad x \ra \extd x = (1-q^{-2})x\extd x,\]
\[(\extd t) \ra \extd t = (\extd x)\ra \extd t = (\extd x) \ra \extd x=0, \quad (\extd t)\ra \extd x = (1-q^2)\extd t \extd x,\]
\[t\ra s = q^{-2}t, \quad t \ra y = 0, \quad x \ra s = x, \quad x \ra y = \frac{t}{1-q^2},\]
\[(\extd t)\ra s = q^{-2}\extd t, \quad (\extd t) \ra y = 0, \quad (\extd x) \ra s = \extd x,\quad  (\extd x) \ra y = \frac{\extd t}{1-q^2},\]
\[t\ra \extd s = t \ra \extd y = x \ra \extd s = x \ra \extd y = 0.\]
\end{proof}

\begin{remark}
It is known\cite{Dri} that $D(U_q(b_+))/(st^{-1}-1) \cong U_q(sl_2)$ by
\[x \mapsto x_+ K, \quad y \mapsto x_- K, \quad t \mapsto K^2,\]
where  $U_q(sl_2)$ is generated by $K, x_\pm$ with the  relations
\[KK^{-1} = K^{-1}K, \quad Kx_\pm = q^{\pm 1}x_\pm K, \quad [x_+, x_-] = \dfrac{K^2-K^{-2}}{q-q^{-1}}.\]
However, $\Omega(D(U_q(b_+)))$ in Example \ref{ex D(U(b_+))} above  does not descend to $\Omega(U_q(sl_2))$ since $\extd(st^{-1}) = 0$ gives $\extd s =q^{-2} \extd t$. 
\end{remark}

\begin{example}\label{exDsu2} Here we work over $\C$ due to the physics context. Let  $U(su_2)$ be the enveloping algebra of $su_2$ with generators $x_a$, for $a=1,2,3$ with primitive coproducts and relations $[x_a, x_b] = 2 \lambda \epsilon_{abc}x_c$ where $\lambda$ is purely imaginary and $\epsilon_{abc}$ is the totally antisymmetric tensor with $\eps_{123}=1$. Let $\Omega(U(su_2))$ be a 4D strongly bicovariant exterior algebra with the following bimodule relations\cite{BaMa}
\[[\extd x_a, x_b] = \lambda \epsilon_{abc} \extd x_c - \lambda^2 \delta^a{}_b\theta, \quad [\theta, x_a] = \extd x_a\]
 \[\extd \theta = 0, \quad \{\extd x_a, \extd x_b\} = 0, \quad \{\theta, \extd x_a \} = 0,\]
and primitive coproducts on $\extd x_a$ and $\theta$. Let $\C[SU_2]$ be the commutative Hopf algebra generated as usual by $\bt = (t^i{}_j)$, with determinant $t^1{}_1 t^2{}_2 - t^1{}_2 t^2{}_1 =1$ and $\Delta t^i{}_j = t^i{}_k  \tens t^k{}_j$. Also let $\Omega(\C[SU_2])$ be the classical 3D strongly bicovariant exterior algebra with
\[ [\extd t^i{}_j , t^k{}_l] = 0, \quad \{\extd t^i{}_j , \extd t^k{}_l\} = 0, \quad \extd t^2{}_2 = (t^2{}_2) (t^1{}_2 \extd t^2{}_1 + t^2{}_1 \extd t^1{}_2 - t^2{}_2 \extd t^1{}_1),\]
 \[ \Delta_* \extd t^i{}_j= \extd t^i{}_k \tens t^k{}_j + t^i{}_k \tens \extd t^k{}_j,\]
where $S\bt = \bt^{-1}$ as usual. 

Then $\Omega(\C[SU_2])\lcross \Omega(U(su_2))$ is a strongly bicovariant exterior algebra on $\C[SU_2]\lcross U(su_2)$ and contains $\Omega(\C[SU_2])$ and $\Omega(U(su_2))$ as sub-strongly bicovariant exterior algebras, and the following cross bimodule relations
\[[x_a, t^i{}_j] =  -\imath\lambda(t^i{}_k (\sigma_a)^k{}_j  - (\sigma_a)^i{}_k t^k{}_j), \quad [x_a,dt^i{}_j] = -\imath\lambda(\extd t^i{}_k (\sigma_a)^k{}_j  - (\sigma_a)^i{}_k \extd t^k{}_j)\]
 \[ [\extd x_a, t^i{}_j] = 0, \quad [\theta, t^i{}_j] = 0, \quad \{\extd x_a , \extd t^i{}_j\} =0, \quad \{\theta, \extd t^i{}_j\}=0,\]
where $(\sigma_a)^i{}_j$ is the $(i,j)$-th entry of the standard Pauli matrix $\sigma_a$ for $a=1,2,3$. This is a  $*$-differential calculus with the usual $*$-structure of $U(su_2)$ and $\C[SU_2]$.  Moreover, $\C[SU_2]\lcross U(su_2)$ acts differentiably on $U(su_2)$. 
\end{example}
\begin{proof}
First note that $x_a \ra t = x_a\eps(t)$ for all $t \in \C[SU_2]$, so $\C[SU_2]\dcross U(su_2) = \C[SU_2]\lcross U(su_2)$. The duality pairing between $\C[SU_2]$ and $U(su_2)$ is given by $\langle t^i{}_j, x_a \rangle = -\imath\lambda (\sigma_a)^i{}_j$, and this gives the stated cross relation on degree 0 as found previously in \cite{BaMa}. One can check that $\langle \ , \ \rangle$ extends to the pairing between $\Omega(\C[SU_2])$ and $\Omega(U(su_2))$ by 0 for degree $\geq 1$, giving the rest of the stated crossed relations on $\Omega(\C[SU_2])\lcross \Omega(U(su_2))$. One can also check that the graded Leibniz rule holds and that $\extd$ is a super-coderivation. Each factor is a  $*$-calculus (in the usual sense that $*$ commutes with $\extd$ and is a graded antilinear order-reversing involution), with the usual $*$-structure given by
\[x_a^* = x_a, \quad \theta^* = -\theta , \quad (\extd x_a)^* = \extd x_a, \quad (t^i{}_j)^* = St^j_i, \quad (\extd t^i{}_j)^*  = S\extd t^j{}_i.\]
 One can check that these  result in a $*$-calculus, e.g.
\begin{align*}
[x_a, \extd t^i{}_j]^* =& [(\extd t^i{}_j)^*, x_a^*] = [S\extd t^i{}_j, x_a] = -\imath \lambda((\sigma_a)^j{}_k S\extd t^k{}_i - (S\extd t^j{}_k )(\sigma_a)^k{}_i)\\
=&\big(-\imath\lambda (\extd t^i{}_k (\sigma_a)^k{}_j - (\sigma_a)^i{}_k \extd t^k{}_j) \big)^*.
\end{align*}

By Proposition \ref{super D(H)}, the action of $\Omega(\C[SU_2])\lcross \Omega(U(su_2))$ on $\Omega(U(su_2))$ as a super module algebra is given by
\[x_a \ra x_b = [x_a,x_b], \quad (\extd x_a) \ra x_b = [\extd x_a, x_b], \quad x_a \ra \extd x_b = [x_a, \extd x_b], \quad (\extd x_a) \ra \extd x_b = \{\extd x_a, \extd x_b\},\]
\[x_a \ra t^i{}_j = x_a \delta^i{}_j -\imath\lambda (\sigma_a)^i{}_j, \quad (\extd x_a )\ra t^i{}_j = \extd x_a \delta^i{}_j, \quad x_a \ra \extd t^i{}_j = 0\]
\[\theta \ra t^i{}_j = \theta \delta^i{}_j, \quad (\extd x_a) \ra \extd t^i{}_j = 0, \quad \theta \ra \extd t^i{}_j = 0.\]
\end{proof}

\begin{remark}\label{remark sigma}
One can replace $A^{\op}$ by $A$ in Proposition~\ref{super D(H)}  and regard the Hopf algebra pairing $\langle \ , \ \rangle$ as a skew pairing $\sigma : H\tens A\to k$ defined as a convolution-invertible map satisfying
\[\sigma(hg, a) = \sigma(h, a\o)\sigma(g, a\t), \quad \sigma(h, ab) = \sigma(h\t, a)\sigma(h\o,b)\]
for $h,g\in H$, $a,b\in A$. Here $\langle S( \ ), \ \rangle$ provides $\sigma^{-1}$ if we start with a Hopf algebra pairing. The above is then equivalent to a generalised quantum double $A \dcross_{\sigma} H$ \cite[Prop.~7.2.7]{Ma:book}. By extending to $\sigma : \Omega(H)\underline{\tens}\Omega(A)\to k$ by 0 for degree $\geq 1$, we have a super double cross product $\Omega(A)\dcross_{\sigma}\Omega(H)$, and by Theorem \ref{thm super dcross} we have $\Omega(A\dcross_\sigma H):= \Omega(A)\dcross_\sigma \Omega(H)$. In this equivalent approach, we work with a strongly bicovariant exterior algebra $\Omega(A)$ rather than with $\Omega(A)^{\op}$.
\end{remark}

\subsection{Exterior algebra on $A\dcross_{\CR} A$} \label{secAdcrossA}

Recall\cite[Chapter~2.2]{Ma:book} that a coquasitriangular Hopf algebra is a Hopf algebra $A$ equipped with involution-invertible $\CR : A \tens A \to k$ such that 
\[ \mathcal{R}(ab,c)=\mathcal{R}(a,c{\o})\mathcal{R}(b,c{\t}),\quad \mathcal{R}(a,bc)=\mathcal{R}(a{\o},c)\mathcal{R}(a{\t},b),\]
\[ a{\o}b{\o}\mathcal{R}(b{\t},a{\t})=\mathcal{R}(b{\o},a{\o})b{\t}a{\t}\]
for all $a,b,c\in A$. In this case we can view $A$ as skew-paired with itself by $\sigma=\CR$ in Remark \ref{remark sigma} and have $D(A,A)=A\dcross_\CR A$ with $A$ left and right acting on itself by
\[b\ra a = b\t \CR(Sb\o, a\o)\CR(b\th, a\t), \quad b\la a = a\t \CR(Sb\o, a\o)\CR(b\t, a\th)\]
for all $a,b \in A$ to give the double cross product structure. The product is $(a \tens b)(c \tens d)= \CR(Sb\o, c\o) a c\t \tens b\t  d\CR(b\th, c\th)$ for all $a,b,c,d\in A$ as in \cite[Sec.~7.2]{Ma:book}.  Hence by the above, if each copy of $A$ has a strongly  bicovariant exterior algebra (they do not need to be the same), say $\Omega'(A)$ and $\Omega(A)$ respectively, we extend $\CR$ by 0 as a skew pairing on degree $\geq 1$ and the double cross product actions extend to actions of each exterior algebra on the other
\begin{align*}
\eta \ra \omega =&\begin{cases} \eta\t \CR(S\eta\o, \omega\o)\CR(\eta\th, \omega\t) & \text{if $\omega \in A$}\\
0 & \text{otherwise}
\end{cases}\\
\eta \la \omega =&\begin{cases} \omega\t \CR(S\eta\o, \omega\o)\CR(\eta\t, \omega\th) & \text{if $\eta \in A$}\\
0 & \text{otherwise}
\end{cases}
\end{align*}
for $\omega\in\Omega'(A)$, $\eta\in\Omega(A)$, where only the parts of $\Delta_*^2 \omega \in A \tens \Omega'(A)\tens A$ and $\Delta_*^2 \eta \in A \tens \Omega(A)\tens A$ contribute in the first and second action respectively.  One can check that these actions obey (\ref{d_H dcross})--(\ref{dcross4}), so we have a super double version $\Omega'(A)\dcross_{\CR} \Omega(A)$ with product and coproduct 
\[(\omega \tens \eta)(\tau \tens \xi)=(-1)^{|\eta||\tau|} \CR(S\eta\o, \tau\o) \omega \tau\t \tens \eta\t  \xi\CR(\eta\th, \tau\th) \]
\[\Delta_*(\omega \tens \eta) = (-1)^{|\eta\o||\omega\t|}\omega\o \tens \eta\o \tens \omega\t \tens \eta\t.\]
for all $\omega, \tau\in \Omega'(A)$, $ \eta, \xi \in \Omega(A)$. In the above product, the crossing between $\eta\t$ and $\tau\t$  should have generated a factor $(-1)^{|\eta\t||\tau\t|}$ but $\CR$ is 0 on degree $\geq 1$, so again only the parts of $\Delta_*^2 \eta, \Delta_*^2 \tau$  with outer factors in $A$ contribute, resulting in $|\eta\t|=|\eta|$ and $|\tau\t|=|\tau|$.

\begin{corollary}\label{super dcross coquasi}
Let $\Omega'(A),\Omega(A)$ be strongly bicovariant exterior algebras on  a coquasitriangular Hopf algebra  $A$. Then $\Omega(A\dcross_{\CR} A):=\Omega'(A)\dcross_{\CR}\Omega(A)$ is a strongly bicovariant exterior algebra on $A\dcross_{\CR} A$ with  differential
\[\extd(\omega \tens \eta)= \extd'_A \omega \tens \eta + (-1)^{|\omega|}\omega \tens \extd_A \eta\]
for $\omega\in\Omega'(A)$ and $\eta\in \Omega(A)$. Moreover,  the action of $A\dcross_{\CR} A$ on $A$ by $a\ra (b\tens c)  = \CR(a\t, b) (Sc\o)a\o c\t$ for all $a,b,c \in A$ is differentiable, extending to an action of $\Omega(A\dcross_{\CR} A)$ on $\Omega(A)$ by 
\begin{align*}
\tau \ra (\omega \tens \eta) =& \begin{cases}
(-1)^{|\eta\o||\tau|}\CR(\tau\t, \omega)(S\eta\o)\tau\o\eta\t & \text{if $\omega \in A$}\\
0 & \text{otherwise}
\end{cases}
\end{align*}
for $\tau\in\Omega(A)$. Here only the part of $\Delta_* \tau \in A \tens \Omega(A)$ contributes. \end{corollary}
\begin{proof}
This is immediate from Proposition \ref{super D(H)}, where $\langle \ , \ \rangle$ is replaced by $\CR$, and $\langle S( \ ), \ \rangle$ is replaced by $\CR^{-1}$ as in Remark \ref{remark sigma}. \end{proof}

Later on, in Section~\ref{sectrans}, we will see that a natural coaction of $A\dcross_{\CR}A$ on a certain transmutation $\underline{A}$ of $A$ is also differentiable under certain circumstances. 

\begin{example}\label{ex dcross A(R)} Let $R \in M_n(k)\tens M_n(k)$ be a $q$-Hecke $R$-matrix and $A(R)$ be the FRT bialgebra over a field $k$,  generated by $\bt=(t^i{}_j)$ as in Section~\ref{secFRT}. We assume that there is a grouplike central $D\in A(R)$ and we let $A=A(R)[D^{-1}]$ with $\Omega(A)$ its strongly bicovariant exterior algebra from Lemma \ref{FRTcalc}. Remembering that $-R_{21}^{-1}$ is also $q$-Hecke and gives the same algebra $A$ but a conjugate calculus $\Omega'(A)$, we have four natural strongly bicovariant calculi on $A\dcross_\CR A$ amounting, without loss of generality, to two distinct constructions:

Case (i):  $\Omega(A\dcross_{\CR} A):=\Omega(A)\dcross_{\CR} \Omega(A)$ with the cross relations and coproducts
\[R\bt_1\bs_2 = \bs_2\bt_1 R, \quad (\extd\bs_2)\bt_1R = R\bt_1\extd\bs_2, \quad R(\extd \bt_1)\bs_2 = \bs_2 \extd\bt_1 R, \quad R\extd\bt_1 \extd\bs_2 = -\extd\bs_2 \extd\bt_1 R\]
\[\Delta\bt = \bt \tens \bt, \quad \Delta \bs = \bs\tens \bs, \quad \Delta_* \extd\bt= \extd\bt\tens \bt + \bt\tens \extd\bt, \quad \Delta_* \extd\bs = \extd\bs\tens \bs + \bs \tens \extd\bs,\]
where $\bs,\bt$ respectively generate the two copies of $A$. 

Case (ii): $\Omega(A\dcross_{\CR} A):=\Omega'(A)\dcross_{\CR} \Omega(A)$ with the same cross relations and coproducts as above but  $(\extd \bs_1)\bs_2=R^{-1}\bs_2\extd\bs_1 R_{21}^{-1}$ and $\extd \bs_1\extd \bs_2=-R^{-1}\extd \bs_2\extd \bs_1R_{21}^{-1}$ for the calculus on the first copy of $A$ (the other copy remains as in Lemma~\ref{FRTcalc}). 

In both cases, $A\dcross_{\CR} A$ acts on $A$ differentiably. By Corollary~\ref{super dcross coquasi}, the action is 
\[\bt_1 \ra \bt_2 = (S\bt_2)\bt_1 \bt_2, \quad (\extd \bt_1) \ra \bt_2 = (S\bt_2)(\extd \bt_1) \bt_2, \quad \bt_1 \ra \extd \bt_2 =  (S\extd\bt_2)\bt_1\bt_2 + (S\bt_2)\bt_1 \extd \bt_2\]
\[\bt_1 \ra \bs_2 = \bt_1 R, \quad (\extd \bt_1) \ra \bs_2 =  (\extd \bt_1)R, \quad \bt_1 \ra \extd \bs_2 = 0.\]
\end{example}

\subsection{Exterior algebras by super double cross coproduct}
Let $H$ and $A$ be bialgebras or Hopf algebras with $A$ a right $H$-comodule algebra with right coaction $\alpha : A\to A\tens H$ and $H$ a left $A$-comodule algebra with left coaction $\beta : H \to A\tens H$. Suppose that $\alpha$ and $\beta$ are compatible as in \cite[Ex.~7.2.15 ]{Ma:book} such that they form a double cross coproduct $H\codcross A$.

Now let $\Omega(H)$ and $\Omega(A)$ be strongly bicovariant exterior algebras and let $\alpha$ and $\beta$ be differentibale with extensions  $\alpha_* : \Omega(A)\to \Omega(A)\underline{\tens} \Omega(H)$ and $\beta_* : \Omega(H)\to \Omega(A)\underline{\tens} \Omega(H)$ as super comodule algebras (commuting with $\extd$ for the graded tensor product differential). Suppose further that $\alpha_*$ and $\beta_*$ are compatible in the sense
\begin{align}
(\Delta_{A*}\tens \id )\circ\alpha_*(\omega)  =& ((\id \tens \beta_*)\circ \alpha_*(\omega\o)) (1_A\tens \alpha_* (\omega\t))\label{codcross cond1}\\
(\id\tens \Delta_{H*})\circ\beta_*(\eta) =& (\beta_*(\eta\o)\tens 1_H)((\alpha_* \tens \id)\circ \beta_*(\eta\t))\label{codcross cond2}\\
\alpha_*(\omega)\beta_*(\eta) =& (-1)^{|\omega||\eta|}\beta_*(\eta)\alpha_*(\omega).\label{codcross cond3}
\end{align}
Then there is a super double cross coproduct $\Omega(H)\codcross \Omega(A)$ with super tensor product algebra structure, counit, and 
\[\Delta_*(\eta \tens \omega) = (-1)^{|\omega\o||\eta\t|}\eta\o \tens \alpha_*(\omega\o)\beta_*(\eta\t)\tens \omega\t\]
for $\eta\in\Omega(H)$ and $\omega\in \Omega(A)$. We omit the proof that $\Omega(H)\codcross\Omega(A)$ is a super Hopf algebra since this is just a super version of the usual double cross coproduct\cite{Ma:book}.

\begin{theorem}\label{thm super codcross}
Let $A, H$ be bialgebras or Hopf algebras forming a double cross coproduct $H\codcross A$ and $\Omega(A), \Omega(H)$ strongly bicovariant exterior algebras with $\alpha,\beta$ differentiable and $\alpha_*, \beta_*$ satisfying (\ref{codcross cond1})-(\ref{codcross cond3}). Then $\Omega(H\codcross A):=\Omega(H)\codcross \Omega(A)$ is a strongly bicovariant exterior algebra on $H\codcross A$ with differential
\[\extd (\eta\tens \omega) = \extd_H \eta \tens \omega + (-1)^{|\eta|}\eta \tens \extd_A \omega.\]
\end{theorem}
\begin{proof}
Since the algebra structure is the super tensor product, the graded Leibniz rule is already proved in Lemma \ref{lemma diff tens prod}. That $\extd$ is a super-coderivation is
\begin{align*}
\Delta_* \extd(\eta&\tens \omega)=\Delta_*(\extd_H \eta \tens \omega)+(-1)^{|\eta|}\Delta_*(\eta\tens \extd_A \omega)\\
=&(-1)^{|\omega\o||\eta\t|}\extd_H \eta\o \tens \alpha_*(\omega\o)\beta_*(\eta\t)\tens \omega\t\\
&+(-1)^{|\omega\o||\eta\t|+|\omega\o|+|\eta\o|}\eta\o\tens \alpha_*(\omega\o)\beta_*(\extd_H\eta\t)\tens \omega\t\\
&+(-1)^{|\omega\o||\eta\t|+|\eta\o|}\eta\o \tens \alpha_*(\extd_A \omega\o)\beta_*(\eta\t)\tens \omega\t\\
&+(-1)^{|\omega\o||\eta\t|+|\eta|+|\omega\o|}\eta\o \tens \alpha_*(\omega\o)\beta_*(\eta\t)\tens \extd_A \omega\t\\
=&(-1)^{|\omega\o||\eta\t|}\extd_H \eta\o \tens \alpha_*(\omega\o)\beta_*(\eta\t)\tens \omega\t\\
&+(-1)^{|\omega\o||\eta\t|+|\eta\o|}\eta\o \tens \extd'(\alpha_*(\omega\o)\beta_*(\eta\t))\tens \omega\t\\
&+(-1)^{|\omega\o||\eta\t|+|\eta|+|\omega\o|}\eta\o \tens \alpha_*(\omega\o)\beta_*(\eta\t)\tens \omega\t\\
=&\big((\extd_H \tens \id + (-1)^{| \ |}\id \tens \extd_A)\tens \id\big) ((-1)^{|\omega\o||\eta\t|}\eta\o \tens \alpha_*(\omega\o)\beta_*(\eta\t)\tens \omega\t)\\
&+(-1)^{| \ |}\big(\id \tens(\extd_H \tens \id + (-1)^{| \ |} \id \tens \extd_A) \big)((-1)^{|\omega\o||\eta\t|}\eta\o \tens \alpha_*(\omega\o)\beta_*(\eta\t)\tens \omega\t)\\
=&(\extd \tens \id + (-1)^{| \ |}\id \tens \extd)\Delta_*(\eta\tens \omega).
\end{align*}
In the first equality we expand $\extd(\eta\tens \omega)$ by its definition, with its comultiplication in the second equality. Notice that $\beta_*(\extd_H \eta\t) = \extd'(\beta_*(\eta\t))$ and $\alpha_*(\extd_A \omega\o) = \extd'(\alpha_*(\omega\o))$, so that  we obtain the third equality by the Leibniz rule for $\extd'$ on $\alpha_*(\omega\o)\beta_*(\eta\t)$. Expanding $\extd' = \extd_A \tens \id + (-1)^{| \ |}\id \tens \extd_H$,  one can rewrite so as to obtain the fourth equality, which is equivalent the required expression. 
\end{proof}

\begin{corollary}\label{codcrosscoact}
The double cross coproduct $H\codcross A$ left-coacts on $H$ as a comodule algebra by $\Delta_L h = h\o\tens\beta(h\t)$ for all $h\in H$ and right-coacts on $A$ as a comodule algebra by $\Delta_R a = \alpha(a\o)\tens a\t$ for all  $a\in A$. Both coactions are differentiable in the context of Theorem \ref{thm super codcross}.
\end{corollary}
\begin{proof}
We write $\alpha(a) = a^{\tilnot} \tens a^{\tilo}$ and $\beta(h) = h^{\baro}\tens h^{\barinfi}$ (summations understood) and  check that $\Delta_L:H\to H\codcross A\tens H$ is a coaction,
\begin{align*}
(\id \tens \Delta_L)\Delta_L h =& h\o \tens h\t{}^{\baro}\tens h\t{}^{\barinfi}\o \tens h\t{}^{\barinfi}\t{}^{\baro}\tens h\t{}^{\barinfi}\t{}^{\baro}\\
=&h\o \tens h\t{}^{\baro} h\th{}^{\baro \tilnot}\tens h\t{}^{\barinfi} h\th{}^{\baro \tilo} \tens h\th{}^{\barinfi \baro}\tens h\th{}^{\barinfi \barinfi}\\
=&h\o \tens h\t{}^{\baro} h\th{}^{\baro}\o{}^{\tilnot} \tens h\t{}^{\barinfi} h\th{}^{\baro}\o{}^{\tilo} \tens h\th{}^{\baro}\t \tens h\th{}^{\barinfi}\\
=& h\o \tens \beta(h\t)\alpha(h\th{}^{\baro}\o)\tens h\th{}^{\baro}\t \tens h\th{}^{\barinfi}\\
=& h\o \tens \alpha(h\th{}^{\baro}\o)\beta(h\t)\tens h\th{}^{\baro}\t \tens h\th{}^{\barinfi}\\
=&\Delta(h\o \tens h\t{}^{\baro})\tens h\t{}^{\barinfi} = (\Delta \tens \id)\Delta_L h. 
\end{align*}
We applied the definition of $\Delta_L$ twice in the first equality, and we used the condition (\ref{codcross cond2}) on $h\t{}^{\barinfi}$ in our notation of $\alpha$ and $\beta$ to obtain the second equality. Then since $\beta$ is a coaction, we rewrote $(\id \tens \beta)\beta (h\th{}^{\barinfi}) = (\Delta_H \tens \id)\beta(h\th{}^{\barinfi})$ to obtain the third equality, which is equivalent to the fourth equality. Then we used condition (\ref{codcross cond3}) to get the fifth equality,  and hence the final expression. We also check that $\Delta_L$ is an algebra map:
\begin{align*}
\Delta_L(hg)=& (h\o g\o) \tens \beta(h\t g\t) = h\o g\o \tens \beta(h\t)\beta(g\t) \\
=& (h\o \tens \beta(h\t))(g\o \tens \beta(g\t))= (\Delta_L h)(\Delta_L g).
\end{align*}
Thus, $H$ is a left $H\codcross A$-comodule algebra. The proof for $\Delta_R$ is similar. Since $\beta_*$ and $\alpha_*$ globally exist as assumed in Theorem \ref{thm super codcross}, it is clear that we can define
\[\Delta_{L*}\eta = \eta\o \tens \beta_*(\eta\t), \quad \Delta_{R*}\omega = \alpha_*(\omega\o)\tens \omega\t\]
for all $\eta \in \Omega(H)$ and $\omega \in \Omega(A)$, and that they have the required properties by our assumptions on $\alpha_*, \beta_*$. For example, on degree 1 we have
\[\Delta_{L*}\extd_H h = \extd_H h\o \tens \beta(h\t) + h\o \tens \beta_*(\extd_H h\t),\]
\[\Delta_{R*}\extd_A a = \alpha_*(\extd_A a\o) \tens a\t + \alpha(a\o)\tens \extd_A a\t.\]
\end{proof}

An application it so provide in principle a strongly bicovariant exterior algebra on the quantum codouble $coD(U_q(su_2)) = U_q(su_2)^{\cop}\codcross \C_q[SU_2]$ as a version of $\C_q[SO_{1,3}]$. We omit details since 
this case is essentially equivalent to $\C_q[SU_2]\dcross\C_q[SU_2]$ in  Example~\ref{ex dcross A(R)} as a different  version of $\C_q[SO_{1,3}]$. 

\section{Differentials on biproduct Hopf algebras}\label{secbiprod}

Here we cover quantum differentials on one of two cases where the Hopf algebra has a cross coproduct coalgebra as typical for an inhomogeneous quantum group coordinate algebra, namely biproducts or  bosonisations\cite{Rad,Ma:skl,Ma:bos,Ma:book}. 

\subsection{Exterior algebras by super bosonisation}\label{secbosthm}
Let  $A$ be a Hopf algebra with a strongly bicovariant exterior algebra $\Omega(A)$ and  $B$ a right $A$-comodule algebra with $\Omega(B)$ such that the right coaction $\Delta_R b= b^{\barnot}\tens b^{\baro}$ is differentiable with extension $\Delta_{R*} : \Omega(B)\to \Omega(B)\underline{\tens}\Omega(A)$ denoted by $\Delta_{R*}\eta = \eta^{\barnot *} \tens \eta^{\baro *}$.  Next we suppose the more specific data in succession: (i)  that $A$ also acts on $\Omega(B)$ so as to make this an $A$-crossed module algebra (an algebra in $\CM^A_A$); (ii) that this action  extends differentiably to $\ra:\Omega(B)\underline{\tens} \Omega(A)\to \Omega(B)$; (iii) that $B$ is a braided Hopf algebra in $\CM^A_A$; (iv) that  $\underline{\Delta}$ extends to a degree preserving super braided coproduct $\underline{\Delta}_*$ making $\Omega(B)$ a super-braided Hopf algebra in $\CM^{\Omega(A)}_{\Omega(A)}$ and $\extd_B$  a super-coderivation in the sense of (\ref{super-coderivation}). We say in this case that $\Omega(B)$ is {\em braided strongly bicovariant}. 

Also note that given a super-braided Hopf algebra $\Omega(B)\in \CM^{\Omega(A)}_{\Omega(A)}$, the usual bosonisation formulae extend with signs to define a super bosonisation $\Omega(A)\rbiprod \Omega(B)$ with 
\[(\omega \tens \eta)(\tau \tens \xi) = (-1)^{|\eta||\tau\o|}\omega \tau\o \tens (\eta\ra \tau \t)\xi\]
\[\Delta_*(\omega \tens \eta)=(-1)^{|\omega\t||\eta\underline{\o}^{\barnot *}|}\omega\o \tens \eta\underline{\o}^{\barnot *}\tens \omega\t \eta\underline{\o}^{\baro *}\tens \eta\underline{\t}\]
for all $\omega, \tau \in \Omega(A)$ and $\eta, \xi \in \Omega(B)$. 

\begin{theorem}\label{calcbos}
Let $A$ be a Hopf algebra, $B$ a braided Hopf algebra in $\CM^A_A$, $\Omega(A)$ strongly bicovariant and $\Omega(B)$  braided strongly bicovariant  in $\CM^{\Omega(A)}_{\Omega(A)}$. Then $\Omega(A\rbiprod B):=\Omega(A)\rbiprod \Omega(B)$ is a strongly bicovariant exterior algebra on $A\rbiprod B$ with differential
\[\extd(\omega\tens \eta)=\extd_A\omega\tens \eta + (-1)^{|\omega|}\omega \tens \extd_B \eta\]
for all $\omega \in \Omega(A)$, $\eta \in \Omega(B)$. \end{theorem}
\begin{proof} 
The graded Leibniz rule holds since
\begin{align*}
\extd(\eta\omega)= &\extd((1\tens \eta)(\omega\tens 1))\\
=&(-1)^{|\eta||\omega\o|}\big((-1)^{|\omega\o|}\omega\o \tens (\extd_B \eta)\ra\omega\t + \extd_A \omega\o \tens \eta\ra \omega\t +(-1)^{|\eta|+|\omega\o|}\omega\tens \eta \ra \extd_A \omega\t\big)\\
=&(1\tens \extd_B \eta)(\omega\tens 1)+(-1)^{|\eta|}(1\tens \eta)(\extd_A \omega \tens 1)= (\extd \eta)\omega + (-1)^{|\eta|}\eta\extd\omega
\end{align*}
for all $\eta\in \Omega(B)$ and $\omega \in \Omega(A)$. Clearly $\extd^2=0$ and thus $\Omega(A)\rbiprod \Omega(B)$ is a DGA. It is a super Hopf algebra by superbosonisation and we check that $\extd$ is a super-coderivation. On $\omega\in\Omega(A)$ this just reduces to the same property there. For $\eta\in\Omega(B)$, 
\begin{align*}
\Delta_*\extd(1\tens\eta)&=\Delta_*(1\tens\extd_B\eta)=1\tens(\extd_B\eta)\underline{\o}^{\barnot *}\tens(\extd_B\eta)\underline{\o}^{\baro *}\tens (\extd_B\eta)\underline{\t}\\
&=1\tens(\extd_B\eta\underline{\o})^{\barnot *}\tens(\extd_B\eta\underline{\o})^{\baro *}\tens \eta\underline{\t} + (-1)^{|\eta\underline{\o}|}1\tens \eta\underline{\o}^{\barnot *}\tens \eta\underline{\o}^{\baro *}\tens\extd_B\eta\underline{\t}\\
&=1\tens\extd_B(\eta\underline{\o}^{\barnot *})\tens \eta\underline{\o}^{\baro *}\tens \eta\underline{\t} + (-1)^{|\eta\underline{\o}^{\barnot *}|}1\tens \eta\underline{\o}^{\barnot *}\tens \extd_A\eta\underline{\o}^{\baro *}\tens\eta\underline{\t}\\
&\quad+ (-1)^{|\eta\underline{\o}|}1\tens \eta\underline{\o}^{\barnot *}\tens \eta\underline{\o}^{\baro *}\tens\extd_B\eta\underline{\t}\\
&=1\tens \extd_B(\eta\underline{\o}^{\barnot *})\tens  \eta\underline{\o}^{\baro *}\tens\eta\underline{\t}+(-1)^{|\eta\underline{\o}^{\barnot *}|} 1\tens\eta\underline{\o}^{\barnot *}\tens  \extd(\eta\underline{\o}^{\baro *}\tens\eta\underline{\t})\\
&=(\extd \tens \id +(-1)^{| \ |}\id \tens \extd)\Delta_*(1\tens \eta). \end{align*}
We used the braided super coderivation property in $\Omega(B)$ for the 3rd equality and differentiability of the coaction so that $\Delta_{R*}$ commutes with the total exterior derivative for the 4th equality. We then recognised the answer.
\end{proof}

In practice, since $\Omega(A)$ and $\Omega(B)$ are generated by their elements of degree 0 and 1, we can construct the bosonisation $\Omega(A)\rbiprod \Omega(B)$ providing  we know the action and coaction on degree 0 and degree 1.  

\begin{lemma} \label{extension ra}
(i) If $B$ is an $A$-crossed module with differentiable action and coaction obeying 
\[\Delta_{R*}((\extd_B b) \ra a) = (\extd_B b^{\barnot}) \ra a\t \tens (Sa\o)b^{\baro}a\th + b^{\barnot} \ra a\t \tens Sa\o (\extd_A b^{\baro}) a\th\]
for all $b\in B$ and for all $a\in A$, then $\Omega(B)$ is a super $\Omega(A)$-crossed module algebra. 

(ii) If furthermore $B$ is a braided Hopf algebra in $\CM^A_A$, 
\[\underline{\Delta} (b\extd_Bc) = b\underline{\o}\extd_B c\underline{\o}{}^{\barnot} \tens (b\underline{\t} \ra c\underline{\o}{}^{\baro})c\underline{\t} + b\underline{\o}c\underline{\o}{}^{\barnot} \tens \big( (b\underline{\t}\ra \extd_A c\underline{\o}{}^{\baro})c\underline{\t} +(b\underline{\t}\ra c\underline{\o}{}^{\baro})\extd_B c\underline{\t} \big)\]
for all $b,c\in B$ is well-defined and $\Omega(B)$ is the maximal prolongation of $\Omega^1(B)$,
then $\Omega(B)$ is  braided strongly bicovariant calculus in  $\CM^{\Omega(A)}_{\Omega(A)}$. 
\end{lemma}
\begin{proof}
(i) We first check,
\begin{align*}
\Delta_{R*}(b \ra \extd_A a) =& \Delta_{R*}(\extd_B (b \ra a)  - (\extd_B b) \ra a)\\
=&  \extd_B(b \ra a)^{\barnot} \tens (b \ra a)^{\baro} +(b \ra a)^{\barnot}\tens \extd_A(b \ra a)^{\baro}-\Delta_{R*}((\extd_B \eta)\ra a) \\
=& \extd_B (b^{\barnot}\ra a\t)\tens (Sa\o)b^{\baro}a\th + b^{\barnot} \ra a\t \tens \extd_A((Sa\o)b^{\baro}a\th)\\
&-(\extd_B b^{\barnot}) \ra a\t \tens (Sa\o)b^{\baro}a\th -b^{\barnot}\ra a\t \tens Sa\o (\extd_A b^{\baro})a\th\\
=&  b^{\barnot}\ra \extd_A a\t \tens (Sa\o)b^{\baro} a\th + b^{\barnot} \ra a\t \tens (\extd_A Sa\o) b^{\baro}a\th \\
&+b^{\barnot}\ra a\t \tens (Sa\o)b^{\baro}\extd_A a\th\\
=& b^{\barnot}\ra (\extd_A a)\t \tens S(\extd_A a)\o b^{\baro}(\extd_A a)\th
\end{align*}
for all $b\in B$, $a\in A$. This makes $B$ is a crossed module as regards the action and coaction of $\Omega^1(A)$. Similarly for $\Omega^1(B)$,  where we check using (\ref{diff coact}),
\begin{align*}
\Delta_{R*}((\extd_B b)&\ra \extd_A a ) = \Delta_{R*}(\extd_B(b\ra \extd_A a))\\
=& \extd_B(b\ra \extd_A a)^{\barnot} \tens (b\ra \extd_A a)^{\baro} + (-1)^{|(b \ra \extd_A a)^{\barnot}|}(b \ra \extd_A a)^{\barnot}\tens \extd_A (b \ra \extd_A a)^{\baro} \\
=&\extd_B(b^{\barnot}\ra \extd_A a\t) \tens (Sa\o)b^{\barnot}a\th + \extd_B(b^{\barnot} \ra a\t) \tens (\extd_A Sa\o)b^{\baro}a\th \\
&+ \extd_B(b^{\barnot} \ra a\t)\tens (Sa\o)b^{\baro}\extd_A a\th - b^{\barnot} \ra \extd_A a\t \tens \extd_A ((Sa\o)b^{\baro}a\th)\\
&+b^{\barnot}\ra a\t \tens \extd_A ((\extd_A Sa\o) b^{\baro}a\th) + b^{\barnot}\ra a\t \tens \extd_A (Sa\o b^{\baro}\extd_A a\th)\\
=&(\extd_B b^{\barnot}) \ra \extd_A a\t \tens (Sa\o)b^{\baro}a\th + (\extd_B b^{\barnot}) \ra a\t \tens (S\extd_A a\o)b^{\baro} a\th \\
&+(\extd_B b^{\barnot})\ra a\t \tens (Sa\o)b^{\baro}\extd_A a\th - b^{\barnot}\ra \extd_A a\t \tens (Sa\o)(\extd_A b^{\baro})a\th\\
&-b^{\barnot}\ra a\t \tens (S\extd_Aa\o)(\extd_A b^{\baro})a\th + b^{\barnot}\ra a\t \tens (Sa\o)(\extd_A b^{\baro})a\th\\
=& (\extd_B b^{\barnot}) \ra (\extd_A a)\t \tens S(\extd_A a)\o b^{\baro} (\extd_A a)\th \\
&+ (-1)^{|(\extd_A a)\o|+|(\extd_A a)\t|}b^{\barnot}\ra (\extd_A a)\t \tens S(\extd_A a)\o (\extd_B b^{\baro}) (\extd_A a)\th\\
=& (-1)^{|(\extd_B b)^{\baro}| (|(\extd_A a)\o|+|(\extd_A a)\t|)}(\extd_B b)^{\barnot}\ra (\extd_A a)\t \tens S(\extd_A a)\o (\extd_B b)^{\baro}(\extd_A a)\th.
\end{align*}
Since the action and coaction extend to $\ra : \Omega(B)\tens \Omega(A)\to \Omega(B)$ and $\Delta_{R*}$ differentiably, and both exterior algebras are generated by degrees 0,1, it follows that $\Delta_{R*}(\eta\ra\omega)$ obeys the crossed module condition in general, making $\Omega(B)$ a super  $\Omega(A)$-crossed module.

(ii) We need to prove that $\underline{\Delta}_* (\xi \eta) = (\underline{\Delta}_* \xi)(\underline{\Delta}_* \eta)$ for all $\xi,\eta\in \Omega(B)$, for which it suffices to prove that $\underline{\Delta}_*$ extends to $\Omega^2(B)$ since $\Omega(B)$ is the maximal prolongation of $\Omega^1(B)$. Applying $\underline{\Delta}_*$ to $b\extd_B c =0$ (for $b,c\in B$, a sum of such terms understood), we have
\[b\underline{\o}\extd_B c\underline{\o}^{\barnot}\tens (b\underline{\t} \ra c\underline{\o}^{\baro})c\underline{\t}=0, \quad b\underline{\o}c\underline{\o}^{\barnot} \tens ((b\underline{\t} \ra \extd_A c\underline{\o}^{\baro})c\underline{\t} +(b\underline{\t}\ra c\underline{\o}^{\baro})\extd_B c\underline{\t})=0.\]

Applying $\extd_B \tens \id$ to the first equation, we have
\[(\extd_B b\underline{\o}) \extd_B c\underline{\o}^{\barnot}\tens (b\underline{\t}\ra c\underline{\o}^{\baro})c\underline{\t} = 0\]
which is the $\Omega^2(B)\tens B$ part of $\underline{\Delta}_* ((\extd_B b)\extd_B c)$. Applying $\id \tens \extd_B$ to the second equation, we have
\[b\underline{\o}c\underline{\o}^{\barnot} \tens ((\extd_B b\underline{\t})\ra \extd_A c\underline{\o}^{\baro})c\underline{\t} + b\underline{\o}c\underline{\o}^{\barnot}\tens ((\extd_B b\underline{\t})\ra c\underline{\o}^{\baro})\extd_B c\underline{\t} = 0\]
which is the $B\tens \Omega^2(B)$ part of $\underline{\Delta}_* (\extd_B b \extd_B c)$. Finally, applying $\extd_B \tens \id$ to the second equation and $\id \tens \extd_B$ to the first equation and subtracting them, we have
\[(\extd_B b\underline{\o})c\underline{\o}^{\barnot}\tens \big((b\underline{\t}\ra \extd_A c\underline{\o}^{\baro})c\underline{\t} + (b\underline{\t}\ra c\underline{\o}^{\baro})\extd_B c\underline{\t}\big) - b\underline{\o}\extd_B c\underline{\o}^{\barnot}\tens ((\extd_B b\underline{\t})\ra c\underline{\o}^{\baro})c\underline{\t} = 0\]
which is the $\Omega^1(B)\underline{\tens} \Omega^1(B)$ part of $\underline{\Delta}_*(\extd_B b \extd_B c)$ and thus completes the proof.
\end{proof}

This lemma assists with the data needed for  Theorem~\ref{calcbos}. Finally, we note that $B$ is canonically an  $A\rbiprod B$-comodule algebra  by $\Delta_R b = b\underline{\o}{}^{\barnot}\tens b\underline{\o}{}^{\baro}\tens b\underline{\t}$. 
\begin{corollary}\label{cobosdif}
Under the condition of Theorem \ref{calcbos},  $\Delta_R : B \to B \tens A\rbiprod B$ is differentiable.
\end{corollary}
\begin{proof}
Since $\Delta_{R*} : \Omega(B)\to \Omega(B)\tens \Omega(A)$ and $\underline{\Delta}_*$ are assumed in Theorem \ref{calcbos}, it is clear that the extension now to $\Delta_{R*} \eta = \eta\undo{}^{\barnot *} \tens \eta\undo{}^{\baro *} \tens \eta\undt$ is well-defined and gives a coaction of $\Omega(A\rbiprod B)$ on $\Omega(B)$ as a restriction of the coproduct of $\Omega(A\rbiprod B)$. For example, on degree 1 we have
\[\Delta_{R*}(\extd_B b) = \extd_B b\underline{\o}{}^{\barnot} \tens b\underline{\o}{}^{\baro}  \tens b\underline{\t} + b\underline{\o}{}^{\barnot} \tens \extd_A b\underline{\o}{}^{\baro}  \tens b\underline{\t} + b\underline{\o}{}^{\barnot} \tens b\underline{\o}{}^{\baro}  \tens \extd_B b\underline{\t} .\]
\end{proof}

 \subsection{Construction by transmutation}\label{sectrans} 
 
 It is known that when $A$ is coquastriangular then $A\dcross_\CR A$ in Corollary~\ref{super dcross coquasi} is isomorphic as
 a Hopf algebra to $A\rbiprod \underline{A}$, \cite[Thm~7.4.10]{Ma:book} c.f.\cite{Ma:bos}. Here the transmutation $\underline{A}$ of $A$  \cite{Ma:bra}  has the same coalgebra as $A$ but a modified product $a\bullet b = a\t b\t \CR((Sa\o)a\th, Sb\o)$ for all $a,b\in \underline{A}$ and is a braided-Hopf algebra in $\CM^A_A$   by the right adjoint coaction $\mathrm{Ad}_R a = a\t \tens (Sa\o)a\th$ and an induced right action of $A$ on $\underline{A}$ given by $b\ra a = b^{\barnot}\CR(b^{\overline{\o}},a) = b\t \CR((Sb\o)b\th, a)$  (this follows from a functor $\underline{A} \in \CM^A \hookrightarrow \CM^A_A$, cf. \cite{Ma:dou}\cite[Lemmas~7.4.4,~7.4.8]{Ma:book}). There is also a natural calculus on $\underline{A}$ 
 when $\Omega(A)$ is strongly bicovariant,  as follows. 
 
 \begin{lemma}{\cite[Prop.~8]{Ma:fuz}}\label{lemtransOmegaA}
Let $\Omega(A)$ be a strongly bicovariant exterior algebra on a coquasitriangular Hopf algebra $A$. Its `transmuted'  braided exterior algebra  $\Omega(\underline{A}):=\underline{\Omega(A)}$ has the same coalgebra and exterior derivative as $\Omega(A)$  and a new product 
\[\omega\bullet \eta = \omega\t \eta\t \CR((S\omega\o)\omega\th, S\eta\o)\]
for all $\omega,\eta\in \Omega(\underline A)$,  where $\CR$ is extended by zero on degree $\ge 1$. Moreover, $\Omega(\underline{A})$ is $A\dcross_\CR A$-covariant for the coaction $\Delta_R a = a\t \tens Sa\o \tens a\th$  on $a\in\underline{A}$. 
\end{lemma}
\begin{proof} (i) As in \cite{Ma:fuz}, we start with the super Hopf algebra map $\pi: \Omega(A)\to A$ given by projecting out degree $\ge 1$. With respect to this, $\Omega(A)$ relatively transmutes to a super-braided-Hopf algebra $\underline{\Omega(A)}$ in $\CM^A$. Here, \cite{Ma:fuz} writes its structure in terms of left-invariant forms according to the description $\Omega(A)=A\rbiprod\Lambda$, but we do not particularly need that here.
Rather, applying transmutation with respect to $\pi$ directly  gives the new product
$\omega\bullet\eta=\omega^{\barnot} \eta\t \CR(\omega^{\baro}, S\pi(\eta\o))$, where the coaction is the right super adjoint coaction of $\Omega(A)$ on itself pushed out along $\pi$. This coaction extends the right adjoint coaction of $A$ on $\underline A$ to one on $\Omega(\underline A)=\underline{\Omega(A)}$. If we extend $\CR$ by zero then this product is equivalent to the formula stated.

(ii) One can also view  $\underline{A}$  as a comodule algebra cotwist for a certain dual 2-cocycle on $A^{\rm op}\tens A$ built from $\CR$, where this Hopf algebra coacts on $A$ by $\Delta_R$ as stated. $\Omega(A)$ being bicovariant means that it is covariant under $\Delta_R$ and comodule algebra cotwists to $\Omega({\underline A})$, while the Hopf algebra $A^{\rm op}\tens A$ Drinfeld-cotwists to $A\dcross_\CR A$. We refer to \cite{Ma:fuz}\cite[Sect.~10.5.3]{Ma:book} for details. \end{proof} 

First we consider the special case where extending $\CR$ by zero on higher degree $\ge 1$ makes $\Omega(A)$ super coquasitriangular. In this case, we can view $\Omega(\underline{A})$ more simply as a super version of the construction of $\underline{A}$,  now giving us a super braided-Hopf algebra in $\CM^{\Omega(A)}_{\Omega(A)}$ by the adjoint coaction and an induced action at the $\Omega(A)$ level. The super bosonisation $\Omega(A) \rbiprod \Omega(\underline{A})$ along the lines of \cite[Prop.~7.4.10]{Ma:book} but now with signs, provides a strongly bicovariant  exterior algebra on $A\rbiprod \underline{A}$ with
\begin{align}\label{cobosOmegaa}
(\omega \tens \eta)(\tau \tens \xi)
&= (-1)^{|\eta||\tau|}\omega \tau\o \tens \eta\t \xi\t \CR((S\eta\o)\eta\th,\tau\t S\xi\o)\\ \label{cobosOmegab}
\Delta_*(\omega \tens \eta) &= (-1)^{|\eta\t|(|\omega\t|+|\eta\o|)} \omega\o \tens \eta\t \tens \omega\t(S\eta\o)\eta\th \tens \eta\fo\\
\label{cobosOmegac}\extd(\omega\tens\eta)&=\extd_A\omega\tens\eta+ (-1)^{|\omega|}\omega\tens\extd_B\eta
\end{align}
for all $\omega, \tau \in \Omega(A)$ and $\eta, \xi \in \Omega(\underline{A})$. It is clear that this could also be viewed as an example of Theorem~\ref{calcbos}.  Corollary~\ref{cobosdif} then tells us that $A\rbiprod\underline A$ coacts differentiably on $\underline A$ under our assumption. The coaction here corresponds to the coaction in the lemma under the isomorphism of $A\dcross_\CR A$ with $A\rbiprod\underline A$. 

If we do not assume that the extension of $\CR$ by zero makes $\Omega(A)$ super coquasitriangular, we still have a natural cross product differential graded algebra $\Omega(A\rbiprod \underline A):=\Omega(A) \rbiprod \Omega(\underline{A})$  defined by (\ref{cobosOmegaa})--(\ref{cobosOmegac}) because  $\Omega(\underline A)$ is a right $A$-module algebra by construction and this pulls back under $\pi$ to a right action of $\Omega(A)$. One can check that $\extd$ is still a super derivation. Likewise the coalgebra is just the cross coproduct by the super right adjoint coaction, and the proof that $\extd$ is a supercoderivation is the same as in Theorem~\ref{calcbos}. Thus, the main difference is that the algebra and coalgebra do not necessarily fit together as a super Hopf algebra, but something more general. Similarly,  from the same data of a strongly bicovariant calculus on a coquasitriangular Hopf algebra, we can set $\Omega(A\dcross_\CR A):=\Omega(A)\dcross_\CR \Omega(A)$ by Corollary~\ref{super dcross coquasi} as a canonical strongly bicovariant calculus on $A\dcross_\CR A$. 

\begin{proposition}\label{isomAbowtie} Let $\Omega(A)$ be a strongly bicovariant exterior algebra on a coquasitriangular Hopf algebra $A$. The Hopf     algebra isomorphism $\varphi:A\dcross_\CR A\cong A\rbiprod\underline A$ defined by $\varphi(a\tens b)=ab\o\tens b\t$ for $a,b\in A$ is differentiable for the stated exterior algebras if and only if the extension of $\CR$ by zero makes $\Omega(A)$ super coquasitriangular. 
 \end{proposition}
\begin{proof} Suppose $\Omega(A)$ is super coquasitriangular then the super version of the bosonisation theory behind $\varphi$ tells us that $\Omega(A)\dcross_\CR \Omega(A)\cong \Omega(A)\rbiprod \Omega(\underline A)$ as super Hopf algebras by $\varphi_*(\omega\tens\eta)=\omega\eta\o\tens\eta\t$ for all $\omega,\eta\in\Omega(A)$. One can check that this is compatible with $\extd$ on both sides. Conversely, suppose that $A$ is coquasitriangular and that $\varphi$ extends to a map $\varphi_*$ of DGAs. Induction on the degree shows that $\varphi_*$ has to have this same form as follows. Let $\varphi_*(\omega\tens 1)=\omega\tens 1$ and $\varphi_*(1\tens\eta)=\eta\o\tens\eta\t$ up to some degree. Then $\varphi_*(a\extd_A\omega\tens 1)=\varphi_*(a)\extd\varphi_*(\omega)=a\extd_A\omega\tens 1$, so the same is true one degree higher. Similarly 
\begin{align*}\varphi_*(1\tens a\extd_A\eta)&=\varphi_*(1\tens a)\extd\varphi_*(1\tens\eta)=(a\o\tens a\t)(\extd_A\eta\o\tens\eta\t+(-1)^{|\eta\o|}\eta\o\tens\extd_A\eta\t)\\
&=a\o\extd_A\eta\o\tens a\t\eta\t+ (-1)^{|\eta\o|}a\o\eta\o\tens\extd_A \eta\t =\Delta_*(a\extd_A\eta)
\end{align*}
as required. The 3rd equality uses the product (\ref{cobosOmegaa}) for each term, but the $\CR$ factors collapse. 

Next, $\varphi_*$ is already  constructed so that it commutes with $\extd$, but one can also check this directly from its stated form. Finally, since $\varphi_*$ is assumed to be a super-algebra map,  the expressions
\begin{align*}
\varphi_*((1\tens \eta)(\omega \tens 1)) =& (-1)^{|\eta||\omega|}\CR(S\eta\o, \omega\o)\phi_*(\omega\t \tens \eta\t)) \CR(\eta\th, \omega\th)\\
=&(-1)^{|\eta\t||\omega\t|+|\eta\th||\omega\t|}\CR(S\eta\o, \omega\o) \omega\t\eta\t \tens \eta\th \CR(\eta\fo,\omega\th), 
\end{align*}
where $|\eta| = |\eta\t|+|\eta\th|$ and $|\omega| = |\omega\t|$, must coincide with
\begin{align*}
\varphi_*(1\tens \eta)\varphi_*(\omega \tens 1)=&(\eta\o \tens \eta\t)(\omega\tens 1)\\
=&(-1)^{|\eta\th||\omega\o|}\eta\o \omega\o \tens \eta\th \CR(S\eta\t, \omega\t)\CR(\eta\fo, \omega\th), 
\end{align*}
where $|\omega\o| = |\omega|$. Applying $\id \tens \underline{\epsilon}$ to both, we obtain
\begin{align*}
\eta\o \omega\o \CR(S\eta\t, \omega\t)\CR(\eta\th, \omega\th) = (-1)^{|\eta\t||\omega\t|}&\CR(S\eta\o,\omega\o) \omega\t\eta\t \CR(\eta\th,\omega\th),
\end{align*}
where the left hand side collapses by the skew pairing properties of $\CR$ to $\eta\omega$. This is then equivalent to the missing super coquasitriangularity axiom, 
\[ \CR(\eta\o,\omega\o)\eta\t\omega\t= (-1)^{|\eta||\omega|}  \omega\o\eta\o \CR(\eta\t,\omega\t)\]  
on noting for the sign that the arguments of $\CR$ have to have degree 0. 
 \end{proof}

We see that, aside from the special case where $\CR$ extends to a super coquasitriangular structure,  the Hopf algebras $A\dcross_\CR A$ and $A\rbiprod \underline A$ are isomorphic but not differentiably (not diffeomorphic) for their specified differential structures.  Similarly,  the coaction of either Hopf algebra on $\underline{A}$ is not necessarily differentiable outside of the special case.  Note, however, that  strongly bicovariant $\Omega(A)$ in the case of $A$ coquasitriangular  can be constructed from $\CR$ and a choice of subcoalgebra in $A$, see  $\CR$\cite{Ma:hod}, and since $\CR_{21}^{-1}$ is also a coquasitriangular structure on $A$, it means that for each such $\Omega(A)$, there is a `conjugate' one $\Omega'(A)$. Therefore we also have a natural variant $\Omega(A\dcross_\CR A):=\Omega'(A)\dcross_\CR\Omega(A)$ by Corollary~\ref{super dcross coquasi} and we conclude by showing in our $R$-matrix setting that with this choice, the coaction is indeed differentiable. 

\begin{example}\label{ex dcross B(R)} Let $A=A(R)[D^{-1}]$ with $R$ $q$-Hecke as in Example~\ref{ex dcross A(R)}. The transmutation $\underline{A}$ consists of the braided matrices $B(R)$ with matrix  $\bu= (u^i{}_j)$ of generators and the braided-matrix relations $R_{21}\bu_1 R\bu_2=\bu_2R_{21}\bu_1R$ as in \cite{Ma:book}, now with $D^{-1}$ adjoined. Here $D$, being grouplike in $A$, is necessarily central as a consequence of the braided commutativity of the transmuted product \cite[Ex~9.4.10]{Ma:book}\cite{Ma:hod} and bosonic (it has trivial braiding with other elements) due to triviality of the right adjoint coaction on $D$. The calculus $\Omega(\underline{A})$ by transmutation of Lemma~\ref{FRTcalc} has bimodule and exterior algebra relations  $\bu_2R_{21}\extd\bu_1 R = R^{-1}\extd \bu_1 R \bu_2$ and $\extd \bu_2R_{21}\extd\bu_1 R =- R^{-1}\extd\bu_1 R\extd \bu_2$ at the level of $B(R)$ cf.\cite[Chapter~10.5.3]{Ma:book}  (the latter was more precisely the transmutation of the conjugate calculus where  $R$ is replaced by $R_{21}^{-1}$). This follows from transmutation formulae in $R\bt_1\bt_2=\bu_1 R \bu_2$ as in \cite[Eqn. (7.37)]{Ma:book} relating the products on the two sides, and the same applies with $\extd$ in the same position on each side. The transmuted exterior algebra is then a comodule algebra under $A\dcross_\CR A$ and the new part is when this is differentiable. For either of the cases in Example~\ref{ex dcross A(R)}, the coaction of $\Omega(A\dcross_{\CR}A)$ on $\Omega(\underline{A})$ as a super-comodule algebra is given by
\[\Delta_{R}\bu = \bu \tens S\bs \tens \bt, \quad \Delta_{R*}\extd \bu = \extd \bu \tens S\bs \tens \bt + \bu \tens \extd S\bs \tens \bt + \bu \tens S\bs \tens \extd \bt.\]
We use the same compact $R$-matrix methods and a notation where we omit $\tens$ but remember that $\bu$ and its differentials commute with $\bs,\bt$ and their differentials. We also note that $S \bs=\bs^{-1}$ so that $\Delta_R\bu=\bs^{-1}(\extd\bu) \bt$. Moreover, a little work using commutation relations in Lemma~\ref{FRTcalc} and Example~\ref{ex dcross A(R)} shows that  $\extd \bs^{-1}=-\bs^{-1}(\extd\bs) \bs^{-1}$ obeys
the useful identities 
\[ \bs_2^{-1}  R  \bt_1=\bt_1R \bs_2^{-1} ,\quad \bt_2 R_{21}\extd \bs_1^{-1}=\extd \bs_1^{-1} R_{21}\bt_2,\quad \bs_2^{-1} R \extd\bt_1=\extd\bt_1R \bs_2^{-1}\]
 (from which one can verify that $\Omega(\underline{A})$  is covariant under $A\dcross_\CR A$, as it must be by Lemma~\ref{lemtransOmegaA}).  Also, depending on the calculus used for $\bs$, we have
\[{\rm Case\ (i)}:\ R_{21}\extd \bs_1^{-1}\bs_2^{-1}=\bs_2^{-1}\extd \bs_1^{-1}R^{-1},\quad {\rm Case\ (ii)}:\  \bs_2^{-1}\extd \bs_1^{-1}R_{21}=R^{-1}\extd \bs_1^{-1}\bs_2^{-1}.\]
Using these, one can check  that the coaction {\em is} always differentiable in Case (ii) but  {\em is not} differentiable in general in Case (i) (unless $R$ is involutive, when this is the same as Case (ii)). For example, in Case (ii), 
\begin{align*} \delta_R(\bu_2R_{21}\extd\bu_1 R)&=\bs_2^{-1}\bu_2\bt_2R_{21}\extd\bs_1^{-1}\bu_1\bt_1R +\bs_2^{-1}\bu_2\bt_2R_{21}\bs_1^{-1} \bu_1\extd\bt_1 R\\
&=\bs_2^{-1}\extd\bs_1^{-1}\bu_2R_{21} \bu_1\bt_2\bt_1 R+\bs_2^{-1}\bs_1^{-1}\bu_2R_{21} \bu_1\bt_2\extd\bt_1 R\\
&=\bs_2^{-1}\extd\bs_1^{-1}\bu_2R_{21} \bu_1R\bt_1\bt_2 +\bs_2^{-1}\bs_1^{-1}\bu_2R_{21} \bu_1R_{21}^{-1}\extd\bt_1 \bt_2\\
&=\bs_2^{-1}\extd\bs_1^{-1}R_{21}\bu_1 R\bu_2  \bt_1\bt_2 + \bs_2^{-1}\bs_1^{-1}R^{-1}\bu_1R \bu_2\extd\bt_1 \bt_2 \\
&=R^{-1}\extd\bs_1^{-1}\bu_1\bs_2^{-1}  R  \bt_1\bu_2\bt_2+R^{-1}\bs_1^{-1}\bu_1\bs_2^{-1} R \extd\bt_1\bu_2 \bt_2\\
&=R^{-1}\extd\bs_1^{-1}\bu_1 \bt_1R \bs_2^{-1} \bu_2 \bt_2+R^{-1}\bs_1^{-1}\bu_1 \extd\bt_1R \bs_2^{-1} \bu_2 \bt_2=
\delta_R(R^{-1}\extd \bu_1 R \bu_2)
 \end{align*}
which, given covariance, shows differentiability at degree 1. Meanwhile, $\Omega(A(R))$ being super coquasitriangular entails $R\bt_1\extd\bt_2=\extd \bt_2\bt_1 R$ or equivalently $(\extd \bt_1)\bt_2=R_{21}\bt_2\extd\bt_1R_{21}^{-1}$ which does not in general agree with our relations in Lemma~\ref{FRTcalc} (unless we are in the involutive case, when Case (i) is also differentiable).  \end{example}

For the $GL_2$ $R$-matrix and a suitable $*$-structure over $\C$, $B(R)$ becomes $q$-Minkowski space as explained in \cite[Chapter~10.5.3]{Ma:book}. The usual $q$-determinant $D$ of $A(R)$ in this case, viewed in $B(R)$, is the braided $q$-determinant.  In fact, we can take any normalisation of $R$ for the present purposes and if we fix a square root of $q$ then we can normalise so that $\CR$ descends to the quotient $D=1$ and accordingly take $A=\C_q[SU_2]$. In this context, $A\dcross_\CR A$ becomes a version of the $q$-Lorentz group. Hence the coaction of the latter at least on the $q$-hyperboloid in $q$-Minkowski space (the braided $q$-determinant 1 quotient) is not differentiable if we take the uniform choice of $\Omega(A\dcross_\CR A)$, but is if we use the conjugate coquasitriangular structure for the left factor. It is worth noting that $A\dcross_\CR A$ is itself coquasitriangular in two natural ways namely $\CR_L$, which restricts to $\CR$ on both factors, and $\CR_D$, which restricts to $\CR_{21}^{-1}$ on the first factor and $\CR$ on the second,  see \cite[Prop. 7.3.1]{Ma:book}. 

\subsection{Exterior algebra on $A\rbiprod V(R)$} Braided linear spaces such as the `quantum plane' provide simpler  $q$-deformed examples of Section~\ref{secbosthm}  and we will compute some of these in detail. We start with general comments at the level of $A=A(R)[D^{-1}]$ with assumed central grouplike $D$ inverted and  $\Omega(A)$ from Lemma~\ref{FRTcalc}. We take the same $\Delta_R{\bx}=\bx\tens\bt$ on $V(R)$ as in Lemma \ref{diff coact A(R) on V(R)}, but now as a coaction of $A$. 

\begin{theorem}\label{calcbos q-Hecke}
Let $A=A(R)[D^{-1}]$ with $R$ $q$-Hecke and the associated $V(R)$ right-covariant braided linear space. Then $\Omega(V(R))$ is a super-braided-Hopf algebra with $x_i,\extd x_i$ primitive in the crossed module category $\CM^{\Omega(A)}_{\Omega(A)}$ with coaction $\Delta_{R*}$ as in Lemma~\ref{diff coact A(R) on V(R)} and 
\[\bx_1 \ra \bt_2 = \bx_1 q^{-1}R_{21}^{-1}, \quad (\extd\bx_1) \ra \bt_2 = \extd \bx_1 q^{-1}R, \quad \bx_1 \ra\extd \bt_2 = (q^{-2}-1)\extd \bx_1 P, \quad \extd  \bx_1 \ra \extd \bt_2 = 0, \]
where $P$ is a permutation matrix. This defines a strongly bicovariant exterior algebra $\Omega(A\rbiprod V(R)):=\Omega(A)\rbiprod \Omega(V(R))$ with relations, coproducts, and antipodes
\[\bx_1 \bt_2 = \bt_2 \bx_1 q^{-1}R_{21}^{-1}, \quad (\extd\bx_1)\bt_2 = \bt_2 \extd \bx_1 q^{-1}R, \quad \bx_1 \extd \bt_2 = (\extd \bt_2)\bx_1 q^{-1}R_{21}^{-1} + (q^{-2}-1)\bt_2\extd \bx_1 P \]
\[\extd\bx_1 \extd \bt_2 = -\extd \bt_2\extd \bx_1 q^{-1}R, \quad \Delta \bx = 1\tens \bx + \bx\tens \bt, \quad \Delta_* \extd \bx = 1\tens \extd\bx + \extd \bx \tens \bt + \bx \tens \extd \bt\] 
\[S\bx = -\bx S\bt, \quad S\extd\bx = -(\extd\bx)S\bt - \bx S\extd \bt.\]
\end{theorem}
\begin{proof} Note that the degree 0 part is essentially the Hopf algebra $A\rbiprod V(R)$ and is just the $q$-Hecke case of the construction of inhomogeneous quantum groups by cobosonisation cf. \cite[Cor~10.2.10]{Ma:book}.  However, the action of $\Omega^1(A)$ is not given as far as we know by previous construction but rather we check directly that we indeed obtain $\Omega(V(R))$ as a super right $\Omega(A)$-crossed module. Thus
\begin{align*}
\Delta_R(\bx_1 \ra \bt_2) =& \bx_1 \ra \bt_2 \tens (S\bt_2)\bt_1\bt_2 = \bx_1 q^{-1} R_{21}^{-1} \tens (S\bt_2)\bt_1\bt_2 = (\id \tens S\bt_2)(\bx_1 q^{-1} R_{21}^{-1} \tens \bt_1\bt_2)\\
=&(\id \tens S\bt_2)(\bx_1 \tens \bt_2\bt_1 q^{-1}R_{21}^{-1}) = \bx_1 \tens (S\bt_2)\bt_2 \bt_1 q^{-1}R_{21}^{-1}) = \bx_1 \tens \bt_1 q^{-1}R_{21}^{-1}\\
\Delta_{R*}((\extd \bx_1) \ra \bt_2) =& (\extd\bx_1) \ra \bt_2 \tens (S\bt_2)\bt_1\bt_2 + \bx_1 \ra \bt_2 \tens (S\bt_2)(\extd\bt_1)\bt_2 \\
=& \extd\bx_1 q^{-1}R \tens (S\bt_2)\bt_1 \bt_2 + \bx_1 q^{-1}R_{21}^{-1}\tens (S\bt_2)(\extd\bt_1)\bt_2\\
=&(\id \tens S\bt_2)(\extd \bx_1 q^{-1}R \tens \bt_1\bt_2 + \bx_1 q^{-1}R_{21}^{-1} \tens (\extd\bt_1)\bt_2)\\
=&(\id \tens S\bt_2)(\extd \bx_1 \tens \bt_2\bt_1 q^{-1}R + \bx_1 \tens \bt_2 \extd\bt_1 q^{-1}R)\\
=&\extd \bx_1 \tens \bt_1 q^{-1}R + \bx_1 \tens \extd \bt_1 q^{-1}R.
\end{align*}
verify that the conditions of Lemma \ref{extension ra} (i) hold and therefore $\Omega(V(R))$ is a super right $\Omega(A)$-crossed module. It is also straightforward to check directly that $\Delta_{R*}(\bx_1 \ra \extd \bt_2)$ and $\Delta_{R*}((\extd \bx_1)\ra \extd \bt_2)$ obey the crossed module axiom. 

Next, one can check that $\Omega(V(R))$ is indeed a super-Hopf algebra in this category with $x_i,\extd x_i$ primitive and with the braiding for $\Omega(V(R))$ in the crossed module category coming out from the action and coaction of $\Omega(A)$ as
\[\Psi(\bx_1 \tens \bx_2) = \bx_2 \tens \bx_1 q^{-1}R_{21}^{-1}, \quad \Psi(\extd \bx_1 \tens \bx_2) = \bx_2 \tens \extd \bx_1 q^{-1}R\] \[\Psi(\bx_1 \tens \extd \bx_2) = \extd \bx_2 \tens \bx_1 q^{-1}R_{21}^{-1} + (q^{-2}-1) \bx_2 \tens \bx_1 P, \quad \Psi(\extd \bx_1 \tens \extd \bx_2)= \extd \bx_2 \tens \bx_1 q^{-1}R. \]
In this way, one verifies that the conditions of Lemma~\ref{extension ra} (ii) hold. We exhibit directly that the construction works by showing that the super coproduct $\Delta_*$ of $\Omega(A\rbiprod V(R))$ is indeed well-defined on degree 1. We let $\lambda = q^{-2}-1$. Then,
\begin{align*}
\Delta_*(\bx_1 \extd \bt_2)=& \extd \bt_2 \tens \bx_1 \bt_2 + \bt_2 \tens \bx_1 \extd \bt_2+ \bx_1 \extd \bt_2 \tens \bt_1 \bt_2 + \bx_1 \bt_2 \tens \bt_1\extd \bt_2\\
=& \extd\bt_2\tens \bt_2\bx_1 q^{-1}R_{21}^{-1} + \bt_2 \tens \extd \bt_2.\bx_1 q^{-1}R_{21}^{-1}+\lambda \bt_2 \tens \bt_2\extd \bx_1 P\\
&+\extd \bt_2 \bx_1 \tens \bt_2\bt_1 q^{-1}R_{21}^{-1} + \lambda\bt_2\extd \bx_1  \tens \bt_2 \bt_1 P + \bt_2 \bx_1 q^{-1}R_{21}^{-1}\tens \bt_1 \extd \bt_2
\end{align*}
\begin{align*}
\Delta_*((\extd \bt_2)\bx_1 q^{-1}&R_{21}^{-1}) = (\extd \bt_2 \tens \bt_2 \bx_1 + \extd \bt_2 .\bx_1 \tens \bt_2\bt_1 + \bt_2 \tens (\extd\bt_2)\bx_1 + \bt_2\bx_1 \tens (\extd \bt_2)\bt_1)q^{-1}R_{21}^{-1}\\
=&\Big(\extd \bt_2 \tens \bt_2 \bx_1 + \extd \bt_2 .\bx_1 \tens \bt_2\bt_1 + \bt_2 \tens \(\extd\bt_2)\bx_1\Big)q^{-1}R_{21}^{-1} + q^{-1}\bt_2\bx_1 \tens R\bt_1\extd \bt_2\\
\Delta_*(\lambda\bt_2\extd \bx_1 P) =& (\bt_2 \tens \bt_2 \extd \bx_1 + \bt_2 \extd \bx_1 \tens \bt_2 \bt_1 + \bt_2\bx_1 \tens \bt_2 \extd \bt_1)\lambda P
\end{align*}
from which we find that $\Delta_*(\bx_1 \extd \bt_2 - \extd \bt_2.\bx_1 q^{-1}R_{21}^{-1} - \lambda\bt_2\extd \bx_1 P)=0$ using $\lambda P = q^{-1}(R_{21}^{-1}-R)$. We  also exhibit that this bimodule relation is compatible with the graded Leibniz rule as it must, 
\begin{align*}
\extd (\bx_1 \extd \bt_2 -& (\extd \bt_2)\bx_1 q^{-1}R_{21}^{-1}-\lambda\bt_2\extd\bx_1 P)\\ 
=& \extd \bx_1 \extd \bt_2 + \extd \bt_2\extd \bx_1 q^{-1}R_{21}^{-1} - \lambda\extd \bt_2 \extd \bx_1 P =\extd\bt_2 \extd \bx_1\Big(q^{-1}(-R+R_{21}^{-1})-\lambda P\Big) =0.
\end{align*}
Similarly for the other relations. 
\end{proof}

We also know by Corollary~\ref{cobosdif} that the right $A\rbiprod V(R)$-coaction $\Delta_R \bx= \bx\tens\bt+ 1\tens \bx\in V(R)\tens A\rbiprod V(R)$ on $V(R)$ underlying our view of the former as an inhomogeneous quantum group is differentiable for the exterior algebras above. 

\subsection{Calculations for the smallest $R$-matrices}\label{sec 4.2}

The construction in Theorem \ref{calcbos q-Hecke} includes the standard $q$-deformation $R$-matrix for the $SL_n$ series for all $n$ as these are known to be $q$-Hecke when normalised correctly, see \cite{Ma:book}. In this case  $A=\C_q[GL_n]$ and $V(R)=\C_q^n$ is the standard quantum-braided $n$-plane with relations $x_jx_i=qx_ix_j$ for all $j>i$. Thus we obtain $\Omega(\C_q[GL_n]\rbiprod\C_q^n)$ such that the canonical coaction of $\C_q[GL_n]\rbiprod\C_q^n$ on $\C_q^n$ is differentiable. In this section, we exhibit $n=1$ and $n=2$ explicitly. 

For $n=1$, $R=(q)$, we have $D=t$ and $A=\C_q[GL_1]=\C[t,t^{-1}]$  the algebraic circle with $\Delta t = t\tens t$ and with strongly bicovariant exterior algebra structure
 \[(\extd t)t=q^2t\extd t, \quad (\extd t)^2=0, \quad \Delta_* \extd t = t\tens \extd t + \extd t \tens t\] 
by Lemma~\ref{FRTcalc}, and implied relations for $t^{-1}$. (This is the standard bicovariant calculus on a circle for a free parameter $q$).  We have $B=V(R)=\C[x]$ with calculus $(\extd x)x = q^2x \extd x$,  $(\extd x)^2=0$, which is  $A$-covariant with $\Delta_Rx=x\tens t$ by Lemma~\ref{diff coact A(R) on V(R)}. 

\begin{proposition}\label{ExB+} Let $B=V(R)=\C[x]$ be a braided-Hopf algebra in $\CM^{A}_{A}$ as part of $\Omega(B)$ a super braided-Hopf algebra in $\CM^{\Omega(A)}_{\Omega(A)}$ with
\[x \ra t = q^{-2}x, \quad x {\ra} \extd t = (q^{-2}-1)\extd x, \quad \extd x \ra t = \extd x, \quad \extd x \ra \extd t = 0,\]
\[\Delta_R x = x\tens t , \quad \Delta_{R*} \extd x = \extd x \tens t + x \tens \extd t,\]
\[(\extd x)x = q^2x \extd x, \quad (\extd x)^2=0, \quad \underline{\Delta}x = 1\tens x + x \tens 1, \quad \underline{\Delta}_*\extd x = 1 \tens \extd x + \extd x \tens 1.\]
The super bosonisation defines a strongly bicovariant exterior algebra $\Omega(\C_q[B_+]):=\Omega(A) \rbiprod \Omega(B)$ with relations and comultiplication
\[xt=q^{-2}tx, \quad (\extd x)t = t\extd x, \quad  (\extd t)x = q^2 x\extd t + (q^2-1)t\extd x , \quad \extd x\extd t= -\extd t\extd x\]
\[\Delta x = 1\tens x + x\tens t, \quad \Delta_* \extd x = 1\tens \extd x + \extd x \tens t + x \tens \extd t\]
where we identify $A\rbiprod B=\C_q[B_+]$, the quantisation of the positive Borel subgroup of $SL_2$ (a quotient of $\C_q[SL_2]$).  Moreover, the coaction $\Delta_R:\C[x]\to\C[x]\tens\C_q[B_+]$ given by $\Delta_R x=1\tens x+ x\tens t$ is differentiable. 
\end{proposition}
\begin{proof} This is read off immediately from Theorem~\ref{calcbos q-Hecke} but the key facts are also simple enough to verify directly. \end{proof}

\begin{remark} The Hopf algebra $\C_q[B_+]$ is also called the Sweedler-Taft algebra (but we think of it as a $q$-deformed coordinate algebra). One can also think of it as  $U_q(b_+)$ and in this case we recover the exterior algebra  $\Omega(U_q(b_+))$ previously found in \cite{Oec}. In addition, we can work with $q$ a primitive $n$th odd root of unity where now $A=\C_q[t]/(t^n-1)$ and $B=\C[x]/(x^n)$, giving us the the exterior algebra of the reduced quantum group $c_q[B_+]$ with additional relations $t^n=1$ and $x^n=0$.
\end{remark}

We now compute the rather more complicated $n=2$ case where  $A=\C_q[GL_2]$. We recall that this has generators $t^1{}_1 = a$, $t^1{}_2 = b$, $t^2{}_1 = c$, $t^2{}_2=d$ with relations, invertible determinant, coproduct and antipode \[ba = qab, \quad ca=qac, \quad db= qbd, \quad dc=qcd\]
\[cb=bc, \quad da-ad=(q-q^{-1})bc, \quad ad-q^{-1}bc = da-qcb =D\]
\[\Delta\begin{pmatrix}
a~~ & b \\ c~~ & d
\end{pmatrix} = \begin{pmatrix}
a~~ & b \\ c~~ & d
\end{pmatrix} \tens \begin{pmatrix}
a~~ & b \\ c~~ & d
\end{pmatrix}, \quad S\begin{pmatrix}
a~~ & b \\ c~~ & d
\end{pmatrix} = D^{-1} \begin{pmatrix}
d~~ & -qb \\ -q^{-1}c~~ & a
\end{pmatrix}.\]
Next, $A$ has an obvious 1-parameter family of coquasitriangular structures 
\begin{equation}\label{Rsl2}
\CR(\bt_1\tens \bt_2)=R_\alpha
= q^{\alpha}\begin{pmatrix}
q ~~& 0~~ & 0~~ & 0\\ 0~~ & 1~~ & q-q^{-1}~~ & 0\\ 0~~ & 0~~ & 1~~ & 0\\ 0~~ & 0~~ & 0~~ & q
\end{pmatrix}
\end{equation}
with $R=R_0$ the $q$-Hecke normalisation. The choice of $\CR$ means a 1-parameter family $\Omega_\alpha(\C_q[GL_2])$ if we use the standard construction for differentials on  coquasitriangular Hopf algebras.  We refer to \cite{Ma:hod} for a recent general treatment,  specialising  which in our case gives the exterior algebra in the form $\C_q[GL_2]\rbiprod\Lambda$ where the space of left-invariant 1-forms $\Lambda^1$ has basis  $e_a,e_b,e_c,e_d$  and cross relations
\[e_a \begin{pmatrix}
a~~ & b\\ c~~& d
\end{pmatrix} = q^{2\alpha}\begin{pmatrix}
q^2 a~~ & b \\
q^2 c~~ & d 
\end{pmatrix} e_a, \quad [e_b, \begin{pmatrix}
a~~ & b\\ c~~& d
\end{pmatrix}]_{q^{1+2\alpha}} = q^{1+2\alpha}\lambda\begin{pmatrix}
0~~ &   a \\
0~~ &   c 
\end{pmatrix}e_a \]
\[[e_c, \begin{pmatrix}
a~~ & b\\ c~~ & d
\end{pmatrix}]_{q^{1+2\alpha}} = q^{1+2\alpha}\lambda\begin{pmatrix}
 b~~ & 0 \\
 d~~ & 0
\end{pmatrix}e_a,\quad [e_d, \begin{pmatrix}
a \\ c
\end{pmatrix}]_{q^{2\alpha}} = q^{2\alpha}\lambda\begin{pmatrix}
 b \\ d
\end{pmatrix} e_b\]
\[[e_d, \begin{pmatrix}
b \\ d
\end{pmatrix}]_{q^{2+2\alpha}} = q^{2\alpha}\lambda\begin{pmatrix}
 a e_c + \lambda b e_a \\
 c e_c + \lambda d e_a
\end{pmatrix} \]
with $\lambda = q-q^{-1}$.  The calculus is inner with $\theta = e_a + e_d$, which defines $\extd$. When $\alpha=-\frac{1}{2}$ the calculus descends to the quotient $D=1$ giving the standard 4D calculus on $\C_q[SL_2]$ as in \cite{Wor} but otherwise we are in the same family but with a different $q$-factor in the commutation relations. On the other hand,  the crossed module braiding $\Psi$ on $\Lambda^1$ as given in \cite{Ma:hod} has an equal number of $R_\alpha$ and its appropriate inverse, so does not depend on the normalisation factor in $R_\alpha$. Hence the left-invariant exterior algebra $\Lambda = T\Lambda^1/\ker(\id-\Psi)$ is the same as for the standard 4D calculus on $\C_q[SL_2]$, namely the usual graded-commutative Grassmannian algebra on $e_a, e_b, e_c$ and 
\[e_a e_d + e_d e_a + q^{-1}\lambda e_c e_b = 0, \quad e_d e_c +q^2e_c e_d + q^{-1}\lambda e_a e_c =0\]
\[e_b e_d + q^2 e_d e_b + q^{-1}\lambda e_b e_a =0, \quad e_a^2 = e_b^2 = e_c^2=0, \quad e_d^2 = q^{-1}\lambda e_c e_b\]
This recalls the full structure of the exterior algebras $\Omega_\alpha(\C_q[GL_2])=\C_q[GL_2]\rbiprod \Lambda$.

\begin{lemma}\label{4Dcalcisom} The 4D strongly bicovariant calculus $\Omega(\C_q[GL_2])$ in  Lemma~\ref{FRTcalc} has relations
\[(\extd a)a = q^2 a\extd a, \quad (\extd a) b = qb\extd a, \quad (\extd a)c = q c\extd a,\quad (\extd a) d=d\extd a\]
\[(\extd b)a = q a\extd b+(q^2-1)b \extd a, \quad (\extd b)b = q^2 b \extd b, \quad (\extd b)c = c\extd b + (q-q^{-1})d \extd a, \quad (\extd b)d = qd\extd b\]
\[(\extd c)a = qa\extd c + (q^2-1)c\extd a, \quad (\extd c)b = b\extd c + (q-q^{-1})d\extd a , \quad (\extd c)c=q^2c\extd c, \quad (\extd c)d = qd \extd c\]
\[(\extd d)a=a\extd d + (q-q^{-1})(b\extd c + c\extd b + (q-q^{-1})d\extd a ), \quad (\extd d)b = qb\extd d + (q^2-1)d\extd b\]
\[(\extd d)c=qc\extd d + (q^2-1)d\extd c, \quad (\extd d)d =q^2d\extd d\]
along with implied relations for $D^{-1}$, and is isomorphic to the $\alpha=0$ member of the standard family of 4D bicovariant calculi $\Omega_\alpha(\C_q[GL_2])$. 
\end{lemma}
\begin{proof} The displayed relations are a routine calculation from Lemma~\ref{FRTcalc}. 
Next, working in this calculus, we consider the basis of left-invariant 1-forms $\omega_a = \varpi(a)$, $\omega_b = \varpi(b)$, $\omega_c = \varpi(c)$, $\omega_d = \varpi(d)$, where $\varpi(a)=(Sa\o)\extd a\t$ is the quantum Maurer-Cartan form. Explicitly, we have basic forms and their relations
\[\omega_a = D^{-1}(d\extd a- qb\extd c), \quad \omega_b=D^{-1}(d\extd b - qb \extd d)\]
\[\omega_c = D^{-1}(a\extd c-q^{-1}c\extd a), \quad \omega_d = D^{-1}(a\extd d - q^{-1}c\extd b)\]
\[\omega_a \begin{pmatrix}
a~~ & b \\ c~~ & d
\end{pmatrix} = \begin{pmatrix}
q^2 a~~ & b \\ q^2 c~~ & d
\end{pmatrix}\omega_a, \quad [\omega_b, \begin{pmatrix}
a~~ & b \\ c~~ & d
\end{pmatrix}]_q = \lambda \begin{pmatrix}
b~~ & 0 \\ d~~ & 0 
\end{pmatrix}\omega_a, \quad [\omega_b, \begin{pmatrix}
a~~ & b \\ c~~ & d
\end{pmatrix}]_q = \lambda \begin{pmatrix}
0~~ & a \\ 0~~ & c 
\end{pmatrix}\omega_a\]
\[[\omega_d, \begin{pmatrix}
a \\ c
\end{pmatrix}] = \begin{pmatrix}
q\lambda b \omega_c + \lambda^2 a\omega_a\\
q\lambda d \omega_c  + \lambda^2 c\omega_a
\end{pmatrix}, \quad [\omega_d, \begin{pmatrix}
b\\ d
\end{pmatrix}]_{q^2} = \lambda\begin{pmatrix}
a \\ c
\end{pmatrix}\omega_b\]
where $\lambda = q-q^{-1}$, and exterior derivative 
\[\extd a = a \omega_a + b \omega_c, \quad \extd b = a \omega_b + b \omega_d, \quad \extd c = c \omega_a + d\omega_c , \quad c\omega_b + d\omega_d.\]
It is then a straightforward calculation to prove that $\varphi : \Omega(\C_q[GL_2]) \to \Omega_0(\C_q[GL_2])$ given by the identity map for elements of degree 0 and 
\[\varphi(\omega_a) = q\lambda e_a, \quad \varphi(\omega_b) = \lambda e_c, \quad \varphi(\omega_c)=\lambda e_b, \quad \varphi(\omega_d) = q\lambda e_d + \lambda^2 e_a\]
is an isomorphism. \end{proof}

We next consider the quantum-braided plane $B=\C_q^2$ generated by $x_1, x_2$ with relation $x_2x_1=qx_1x_2$ and viewed initially as a braided-Hopf algebra in the category of right $\C_q[GL_2]$-comodules and with exterior algebra  $\Omega(\C_q^2)$ and standard relations and coaction from Lemma~\ref{diff coact A(R) on V(R)}, 
\begin{align}\label{planecalc} (\extd x_i)x_i &= q^2 x_i \extd x_i, \quad (\extd x_1)x_2 = qx_2 \extd x_1, \quad (\extd x_2) x_1 = qx_1 \extd x_2 + (q^2-1)x_2 \extd x_1\nonumber\\
(\extd x_i)^2 &= 0, \quad \extd x_2\extd x_1 = -q^{-1}\extd x_1  \extd x_2,\quad \Delta_R x_i = x_j \tens t^j{}_i. \end{align}
We also view $\C_q[GL_1]$, covered above in our family with  $n=1$, as the algebraic circle $\C[D,D^{-1}]$ with its 1-dimensional strongly bicovariant differential calculus with relations $(\extd D)D=q^2D\extd D$ and coproduct $\Delta D=D\tens D$.

\begin{proposition}\label{new 4D calc} The 4D strongly bicovariant calculus $\Omega(\C_q[GL_2])$ constructed from  Lemma~\ref{FRTcalc} is the universal strongly
 bicovariant exterior algebra such that
 \begin{enumerate}\item the canonical right coaction of $\C_q[GL_2]$ on  $\C_q^2$ is differentiable; 
 \item the determinant  inclusion $\C_q[GL_1]\hookrightarrow \C_q[GL_2]$ is differentiable.
\end{enumerate}
\end{proposition}
\begin{proof} Requiring only $\Delta_{R*}((\extd x_i)x_i)=\Delta_{R*}(q^2 x_i \extd x_i)$, we obtain all the relations stated in Lemma~\ref{4Dcalcisom} except for those stated for $(\extd a)d, (\extd b)c$, $(\extd c)b$, $(\extd d)a$, $(\extd c)\extd b$ and $(\extd d)\extd a$, but we also get the following additional relations
\[(\extd c) b +q^{-1}(\extd a)d = b\extd c + q d\extd a, \quad (\extd b) c +q^{-1}(\extd a)d = c\extd b + q d\extd a\]
\[(\extd c)b + q(\extd d)a = q^2 b\extd c + (q^2-1)c\extd b + q a\extd d + q(q^2-1)d \extd a\]
\[(\extd b)c + q(\extd d)a = q^2 c\extd b + (q^2-1)b\extd c + q a\extd d + q(q^2-1)d \extd a.\]

Clearly $\C[D,D^{-1}]$ is a sub-Hopf algebra of $\C_q[GL_2]$ with $D$ the $q$-determinant. That the inclusion extends to differentials amounts to requiring $(\extd D)D = q^2 D\extd D$ in the bigger calculus. If we impose this then the above conditions imply that
\[a(\extd a)d^2 - ad(\extd a)d + q bcd (\extd a) - bc(\extd a) d=0.\]
One can simplify this further by moving the elements of degree 0 to the right by using the relations in the lemma already known,  to obtain $[\extd a, d]D =0$ and hence $[\extd a, d]=0$. The remaining relations of $\Omega^1(\C_q[GL_2])$  in Lemma~\ref{4Dcalcisom} then follow. We also find that
\[\extd D = a\extd d - q^{-1}b\extd c -q^{-1}c\extd b +q^{-2}d\extd a\]
\[(\extd t^i{}_j)D = q^2 D \extd t^i{}_j, \quad (\extd D)t^i{}_j = t^i{}_j \extd D + (q^2-1)D\extd t^i{}_j.\]

By applying $\extd$ to the stated bimodule relations, we also obtain 
\[(\extd a)^2=(\extd b)^2 = (\extd c)^2 = (\extd d)^2= (\extd D)^2=0\]
\[\extd b\extd a = -q^{-1}\extd a\extd b, \quad \extd c\extd a = -q^{-1}\extd a\extd c, \quad \extd d\extd b = -q^{-1}\extd b\extd c, \quad \extd d\extd c= -q^{-1}\extd c\extd d\]
\[\extd c\extd b = -\extd b\extd c + (q-q^{-1})\extd a\extd d , \quad \extd d\extd a= -\extd a\extd d, \quad \extd D\extd t^i{}_j = -q^{-2}\extd t^i{}_j \extd D\]
for the degree 2 relations. \end{proof}

Universal here means we used only the conditions stated to derive the relations, but there could be quotients with the same two properties.  It seems likely that similar results apply for all $\C_q[GL_n]$, but this is beyond our scope here (it would require rather more machinery  than we have recalled here, such as properties of the quantum Killing form in \cite{Ma:hod}). We are now ready to state our example of Theorem~\ref{calcbos q-Hecke}. 

\begin{proposition}\label{propexP} Let $A=\C_q[GL_2]$ with $\Omega(\C_q[GL_2])$ its strongly bicovariant exterior algebra  in Lemma~\ref{4Dcalcisom}.  Let $B= \C_q^2$ be viewed as a braided-Hopf algebra in $\CM^A_A$ and $\Omega(\C_q^2)$ as a super braided Hopf algebra in the category of $\Omega(\C_q[GL_2])$-crossed modules with (co)action and coproduct
\[ x_1 \ra \begin{pmatrix}
a~~ & b \\ c~~ & d\end{pmatrix} = \begin{pmatrix}q^{-2} x_1~~ & -q^{-1}\lambda x_2\\0~~ & q^{-1}x_1\end{pmatrix},\quad  x_2 \ra \begin{pmatrix}a~~ & b \\ c~~ & d\end{pmatrix} = \begin{pmatrix}q^{-1} x_2~~ & 0\\ 0~~ & q^{-2}x_2
\end{pmatrix},\]
\[x_1 {\ra} \begin{pmatrix}\extd a~~ & \extd b\\\extd c~~ & \extd d \end{pmatrix} = \begin{pmatrix} -q^{-1}\lambda\extd x_1~~ \  & -q^{-1}\lambda \extd x_2\\ 0~~ \  & 0\end{pmatrix}, \quad \extd x_1 \ra \begin{pmatrix} a~~  & b\\ c~~ & d
\end{pmatrix} =\begin{pmatrix}\extd x_1~~ \  & 0 \\ 0~~ \  & q^{-1}\extd x_1\end{pmatrix}, \]
\[x_2 {\ra} \begin{pmatrix}\extd a~~ & \extd b\\\extd c~~ & \extd d\end{pmatrix} = \begin{pmatrix} 0~~ \  & 0 \\ -q^{-1}\lambda\extd x_1~~ \  & -q^{-1}\lambda \extd x_2\end{pmatrix}, \quad \extd x_2 \ra \begin{pmatrix},a~~ & b\\ c~~ & d
\end{pmatrix} = \begin{pmatrix}
q^{-1}\extd x_2~~ \  & 0\\ q^{-1}\lambda\extd x_1~~ \  & \extd x_2
\end{pmatrix}, \]
\begin{align*}
\extd x_i \ra \extd t^k{}_l = 0,\quad &\Delta_R x_1 = x_1 \tens a + x_2 \tens c,\quad  \Delta_R x_2 = x_1 \tens b + x_2 \tens c \label{act on C_q^2},\\
&\Delta_{R*} \extd x_1 = \extd x_1 \tens a + \extd x_2 \tens c + x_1 \tens \extd a + x_2 \tens \extd c,\\ 
&\Delta_{R*} \extd x_2 = \extd x_1 \tens b + \extd x_2 \tens d + x_1 \tens \extd b + x_2 \tens \extd d,\\
&\underline{\Delta}x_i=x_i\tens 1+1\tens x_i,\quad \underline{\Delta}_*\extd x_i=\extd x_i\tens 1+1\tens \extd x_i,
\end{align*}
where $\lambda = q-q^{-1}$. Its super bosonisation is a strongly bicovariant exterior algebra $\Omega(\C_q[GL_2]\rbiprod\C^2):=\Omega(\C_q[GL_2])\rbiprod \Omega(\C_q^2)$ with sub-exterior algebras $\Omega(\C_q[GL_2])$, $\Omega(\C_q^2)$  and cross relations and super coproduct
\[x_1 \begin{pmatrix}
a~~ & b\\ c~~ & d
\end{pmatrix} = \begin{pmatrix}
q^{-2}a x_1~~ \ & q^{-1} bx_1  -q^{-1}\lambda ax_2 \\
q^{-2}c x_1~~ \ & q^{-1}dx_1  -q^{-1}\lambda cx_2
\end{pmatrix}, \quad x_2 \begin{pmatrix}
a~~ & b\\ c~~ & d
\end{pmatrix} = \begin{pmatrix}
q^{-1}ax_1~~ \ & q^{-2}bx_2\\
q^{-1}cx_2~~ \ & q^{-2}dx_2
\end{pmatrix},\]
\[\extd x_1 \begin{pmatrix}
a~~ & b\\ c~~ & d
\end{pmatrix} = \begin{pmatrix}
a\extd x_1 ~~ \ & q^{-1}b\extd x_1\\
c\extd x_1 ~~ \ & q^{-1}d\extd x_1
\end{pmatrix}, \quad \extd x_2 \begin{pmatrix}
a~~ & b\\ c~~ & d
\end{pmatrix} = \begin{pmatrix}
q^{-1}a\extd x_2 +q^{-1}\lambda b\extd x_1 ~~ \ & b \extd x_2\\
q^{-1}c\extd x_2 +q^{-1}\lambda c\extd x_2 ~~ \ & d \extd x_2
\end{pmatrix},\]
\[x_1 \begin{pmatrix}
\extd a~~ &\extd b\\ \extd c~~ &\extd d
\end{pmatrix} = \begin{pmatrix}
q^{-2}(\extd a)x_1  -q^{-1}\lambda a\extd x_1~~ \  \ & q^{-1}(\extd b)x_1  -q^{-1}\lambda((\extd a)x_2+a\extd x_2)\\
q^{-2}(\extd c)x_1  -q^{-1}\lambda c\extd x_1 ~~ \  \ & q^{-1}(\extd d)x_1 -q^{-1}\lambda((\extd c)x_2 + c\extd x_2) 
\end{pmatrix},\]
\[x_2 \begin{pmatrix}
\extd a~~ &\extd b\\ \extd c~~ &\extd d
\end{pmatrix} = \begin{pmatrix}
q^{-1}(\extd a)x_2 -q^{-1}\lambda b\extd x_1 ~~ \ \ & q^{-2}(\extd b)x_2  -q^{-1}\lambda b\extd x_2\\
q^{-1}(\extd c)x_2 -q^{-1}\lambda d\extd x_1 ~~ \ \ & q^{-2}(\extd d)x_2  -q^{-1}\lambda d\extd x_2
\end{pmatrix},\]
\[\extd x_1 \begin{pmatrix}
\extd a ~~ &\extd b\\ \extd c~~ &\extd d
\end{pmatrix} = -\begin{pmatrix}
\extd a\extd x_1 ~~ \ & q^{-1}\extd b\extd x_1\\
\extd c\extd x_1 ~~ \ & q^{-1}\extd d\extd x_1
\end{pmatrix},\]
\[\extd x_2 \begin{pmatrix}
\extd a ~~ &\extd b\\ \extd c~~ &\extd d
\end{pmatrix} = -\begin{pmatrix}
q^{-1}\extd a\extd x_2 +q^{-1}\lambda \extd b\extd x_1~~ \ \ & \extd b\extd x_2\\
q^{-1}\extd c\extd x_2 +q^{-1}\lambda \extd d\extd x_1~~ \  \ & \extd d\extd x_2
\end{pmatrix},\]
\[\Delta x_i = 1\tens x_i + \Delta_R(x_i), \quad \Delta_*(\extd x_i) = 1\tens \extd x_i + \Delta_{R*}(\extd x_i)\]
and the coproduct of $\Omega(\C_q[GL_2])$. Moreover, the canonical coaction $\Delta_R:\C_q^2\to \C_q^2\tens\C_q[GL_2]\rbiprod\C_q^2$ is differentiable. 
\end{proposition}
\begin{proof} We apply Theorem~\ref{calcbos q-Hecke}. The $\C_q[GL_2]$-crossed module structure has coaction $\Delta_R x_i = x_j \tens t^j{}_i$ and right action $x_i \ra t^k{}_l = q^{-1} x_j (R_{21}^{-1})^j{}_i{}^k{}_l $ which computes as shown. The rest follows similarly by computation for our choice of $R$ as in (\ref{Rsl2}) with the $q$-Hecke normalisation $\alpha=0$. The coaction of $\C_q[GL_2]\rbiprod\C_q^2$ at the end is $\Delta_R x_i=1\tens x_i+ x_j\tens t^j{}_i$.  \end{proof}

\begin{remark}  $\C_q[GL_2]\rbiprod\C_q^2$ is a quantum deformation of a maximal parabolic $P\subset SL_3$ and one can check that this is indeed isomorphic to a quotient of $\C_q[SL_3]$. We have found its strongly bicovariant calculus and moreover it coacts differentiably on $\C_q^2$ by our results. 
The same construction still works when $q$ is a primitive $n$-th odd  root of unity. Here we take $A=c_q[GL_2]=\C_q[GL_2]/(a^n-1, b^n, c^n, d^n-1)$ and $B=c_q^2=\C_q^2/(x_1^n, x_2^n)$ with $c_q[GL_2]\rbiprod c_q^2$ and strongly bicovariant exterior algebra as above with the additional relations $c_q[GL_2]$ and $c_q^2$.
\end{remark}

\section{Differentials on bicrossproduct Hopf algebras}\label{secbicross}

We conclude the paper with quantum differentials for the second case where the Hopf algebra has a cross coproduct coalgebra as typical for inhomogeneous quantum group coordinate algebras, namely the bicrossproduct  construction \cite{Ma:phy, Ma:book}.

\subsection{Exterior algebras by super bicrossproduct}
Let $H$ and $A$ be  Hopf algebras with $A$ a left $H$-module algebra by left action $\la$ and $H$ a right $A$-comodule coalgebra by $\beta : H \to H \tens A$ denoted $\beta(h) = h^{\barnot}\tens h^{\baro}$. We suppose in this section that $\la$ and $\beta$ are compatible as in \cite[Thm.~6.2.2]{Ma:book} such that the cross (or smash) product by $\la$ and cross coproduct by $\beta$ form a bicrossproduct Hopf algebra $A \bicross H$. In the same spirit as  Section \ref{secbiprod}, we now construct a strongly bicovariant exterior algebra on $A\bicross H$ by a super bicrossproduct.

To this end, suppose that $\Omega(A)$ and $\Omega(H)$ are strongly bicovariant exterior algebras and that $\la$ is differentiable. This  entails that $\Omega(A)$ is $H$-covariant i.e. an $H$-module algebra by an action $\la:H\tens \Omega(A)\to \Omega(A)$ commuting with $\extd$ and that this extends further to a super $\Omega(H)$-module algebra by  $\la:\Omega(H)\underline{\tens} \Omega(A)\to \Omega(A)$ such that
\begin{equation}\label{action alpha}
\extd_A(\eta \la \omega) = (\extd_H \eta) \la \omega + (-1)^{|\eta|}\eta \la (\extd_A \omega)
\end{equation}
for all $\omega\in \Omega(A)$ and $\eta \in \Omega(H)$. 

On the dual side, we assume that $\beta$ extends to a degree-preserving super coaction $\beta_* : \Omega(H)\to \Omega(H)\underline{\tens} \Omega(A)$, denoted by $\beta_*(\eta) =\eta^{\barnot *} \tens \eta^{\baro *}$, such that $\Omega(H)$ is a super $\Omega(A)$-comodule coalgebra and
\begin{equation}\label{coaction beta}
\beta_*(\extd_H \eta) = \extd_H \eta^{\barnot *} \tens \eta^{\baro *} +(-1)^{|\eta^{\barnot *}|} \eta^{\barnot *}\tens \extd_A \eta^{\baro *}. 
\end{equation}

If, moreover, $\la$ and $\beta_*$ obey the super bicrossproduct conditions:
\begin{align}
\epsilon(\eta \la \omega) &= \epsilon(\eta)\epsilon(\omega)\label{bicross1}\\
\Delta_*(\eta \la \omega) &= (-1)^{(|\eta\o{}^{\baro *}
|+|\eta\t|)|\omega\o|} \eta\o{}^{\barnot *} \la \omega\o \tens \eta\o{}^{\baro *}(\eta\t \la \omega\t) \label{bicross2}\\
\beta_*(\eta\xi) &= (-1)^{(|\eta\o{}^{\baro *}
|+|\eta\t|)|\xi^{\barnot *}|} \eta\o{}^{\barnot *}\xi^{\barnot *} \tens \eta\o{}^{\baro *}(\eta\t \la \xi^{\baro *}) \label{bicross3}\end{align}
\begin{equation}
 (-1)^{|\omega||\eta\t{}^{\baro *}|+|\eta\o||\eta\t{}^{\barnot *}|} \eta\t{}^{\barnot *}\tens (\eta\o \la \omega)\eta\t{}^{\baro *} =  \eta\o{}^{\barnot *}\tens \eta\o{}^{\baro *}(\eta\t \la \omega)\label{bicross4}
\end{equation}
then we have a bicrossproduct super Hopf algebra  $\Omega(A)\bicross \Omega(H)$  with product and coproduct
\[(\omega\tens \eta)(\tau \tens \xi)= (-1)^{|\eta\t||\tau|}\omega(\eta\o \la \tau)\tens \eta\t \xi\]
\[\Delta_*(\omega\tens \eta)= (-1)^{|\omega\t||\eta\o{}^{\barnot *}|}\omega\o \tens \eta\o{}^{\barnot *}\tens \omega\t\eta\o{}^{\baro *}\tens \eta\t\]
for all $\omega, \tau \in \Omega(A)$ and for all $\eta, \xi \in \Omega(H)$. We omit the proof since this is similar to the usual version \cite{Ma:phy, Ma:book} with some extra signs. 

\begin{theorem}\label{calc bicross}
Let $A, H$ be Hopf algebras forming a bicrossproduct $A\bicross H$ and let $\Omega(A)$ and $\Omega(H)$ be strongly bicovariant exterior algebras with $\la, \beta_*$ obey the conditions (\ref{action alpha})-(\ref{bicross4}). Then $\Omega(A\bicross H):=\Omega(A)\bicross \Omega(H)$ is a strongly bicovariant exterior algebra on $A\bicross H$ with differential
\[\extd(\omega\tens \eta)=\extd_A \omega \tens \eta + (-1)^{|\omega|}\omega \tens \extd_H \eta,\]
for all $\omega \in \Omega(A)$ and $\eta \in \Omega(H)$.
\end{theorem}
\begin{proof}
This is clear since the algebra structure here is the left-right reversal of a right-handed super cross product and the super coalgebra structure is a super right cross coproduct, both of which were already covered in Theorem \ref{calcbos} with the same form of $\extd$.
\end{proof}

In practice, we typically only need to know that the bicrossproduct action $\la$ and coaction $\beta$ extend to degree 1 to construct the super bicrossproduct, since the extension to higher degrees is determined. Note also that if  $\Omega(A)$ is an $H$-module algebra and $\Omega(H)$ is the maximal prolongation of $\Omega^1(H)$, then by the left-hand reversal of Lemma \ref{extension ra}(i), $\Omega(A)$ is a super $\Omega(H)$-module algebra such that  (\ref{action alpha}) holds. 

\begin{lemma}
Let $A, H$ be Hopf algebras forming a bicrossproduct $A \bicross H$. Let $\Omega(A)$ be a super left $\Omega(H)$-module algebra such that (\ref{action alpha}) holds, suppose that the coaction  extends to a well-defined map $\beta_* : \Omega^1(H)\to \Omega^1(H)\tens A \oplus H \tens \Omega^1(A)$ by 
\[ (1)\quad\quad 
\beta_*(h\extd_H g) = h\o{}^{\barnot} \extd_H g^{\barnot} \tens h\o{}^{\baro} (h\t \la g^{\baro})+ h\o{}^{\barnot} g^{\barnot} \tens h\o{}^{\baro}(h\t\la \extd_A g^{\baro})\]
for all $h,g\in H$, that the action obeys  $\eps(\eta\la\omega)=\eps(\eta)\eps(\omega)$ for all $\eta\in \Omega(H)$, $\omega\in \Omega(A)$ and that
\[(2)\qquad
\Delta_*(h \la \extd_A a) = h\o{}^{\barnot} \la \extd_A a\o \tens h\o{}^{\baro}(h\t \la a\t)+ h\o{}^{\barnot} \la a\o \tens h\o{}^{\baro}(h\t \la \extd_A a\t)\]
\[(3)\quad\quad\quad\quad\quad\quad\quad\quad\quad
h\t{}^{\barnot} \tens (h\o \la \extd_A a)h\t{}^{\baro} = h\o{}^{\barnot}\tens h\o{}^{\baro}(h\t \la \extd_A a)\quad\quad\quad\quad\]
for all $h\in H$ and $a \in A$. If both $\Omega(H),\Omega(A)$ are maximal prolongations  then $\beta_*$ extends to all degrees obeying (\ref{bicross2})-(\ref{bicross4}) and form $\Omega(A)\bicross \Omega(H)$ by Theorem \ref{calc bicross}.
\end{lemma}
\begin{proof}
(i) First we first check that $\beta_*((\extd_H h)g)$ also satisfies (\ref{bicross3}) for products from the other side,
\begin{align*}
\beta_*((\extd_H h)g) =& \Delta_{R*}(\extd_H(hg) - h\extd_H g)\\
=& \extd_H(hg)^{\barnot} \tens (hg)^{\baro} + (hg)^{\barnot}\tens \extd_A(hg)^{\baro} - h\o{}^{\barnot} \extd_H g^{\barnot} \tens h\o{}^{\baro} (h\t \la g^{\baro})\\
&- h\o{}^{\barnot} g^{\barnot} \tens h\o{}^{\baro}(h\t\la \extd_A g^{\baro})\\
=& \extd_H(h\o{}^{\barnot}g^{\barnot}) \tens h\o{}^{\baro}(h\t \la g^{\baro}) + h\o{}^{\barnot}g^{\barnot} \tens \extd_A(h\o{}^{\baro}(h\t \la g^{\baro}))\\
&-h\o{}^{\barnot} \extd_H g^{\barnot} \tens h\o{}^{\baro} (h\t \la g^{\baro})- h\o{}^{\barnot} g^{\barnot} \tens h\o{}^{\baro}(h\t\la \extd_A g^{\baro})\\
=& (\extd_H h\o{}^{\barnot})g^{\barnot} \tens h\o{}^{\baro} (h\t \la g^{^{\baro}}) + h\o{}^{\barnot}g^{\barnot}\tens (\extd_A h\o{}^{\baro})(h\t \la g^{\baro})\\
&+h\o{}^{\barnot}g^{\barnot} \tens h\o{}^{\baro}((\extd_H h\t) \la g^{^{\baro}})\\
=&((\extd_H h)\o{}^{\barnot})g^{\barnot} \tens (\extd_H h)\o{}^{\baro}((\extd_H h)\t \la g^{\baro})
\end{align*}
for $h,g\in H$. Next, we prove that $\beta_*$ extends by (\ref{bicross3}) to $\Omega^2(H)$ for the maximal prolongation. Suppose $h\extd_H g = 0$  in $\Omega^1(H)$ (a sum of such terms understood) then $\beta_*(h\extd_H g ) = 0$ tells us that
\[h\o{}^{\barnot} \extd_H g^{\barnot} \tens h\o{}^{\baro} (h\t \la g^{\baro})=0, \quad  h\o{}^{\barnot} g^{\barnot} \tens h\o{}^{\baro}(h\t\la \extd_A g^{\baro})=0.
\]
Applying $\extd_H \tens \id$ to the first equation, then the $\Omega^2(H)\tens A$-part of $\beta_*(\extd_H h \extd_H g)$ is
\begin{align*}
& \extd_H h\o{}^{\barnot} \extd_H g^{\barnot} \tens h\o{}^{\baro} (h\t \la g^{\baro})=0.
\end{align*}
Applying $\id \tens \extd_A$ to the second equation, the $H\tens \Omega^2(A)$-part of $\beta_*(\extd_H h \extd_H g)$ is
\begin{align*}
h\o{}^{\barnot} g^{\barnot} \tens\big( (\extd_A h\o{}^{\baro})(h\t \la \extd_A g^{\baro}) + h\o{}^{\baro}((\extd_H h\t) \la \extd_A g^{\baro}) \big)=0.
\end{align*}
Finally, applying $\extd_H \tens \id$ to the second equation and $\id \tens \extd_A$ to the first equation and subtracting them, the $\Omega^1(H)\underline{\tens} \Omega^1(A)$-part of $\beta_*(\extd_H h \extd_H g)$ is
\begin{align*}
(\extd_H h\o{}^{\barnot})g^{\barnot} \tens& h\o{}^{\baro} (h\t \la \extd_A g^{\baro})\\&-h\o{}^{\barnot}\extd_H g^{\barnot} \tens ((\extd_A h\o{}^{\baro})(h\t \la g^{\baro})+h\o{}^{\baro}(\extd_H h\t \la g^{\baro}))=0.
\end{align*}
Since $\la : \Omega(H)\underline{\tens}\Omega(A)\to \Omega(A)$ is defined and $\Omega(H)$ is the maximal prolongation of $\Omega^1(H)$, $\beta_*$ can be extended further to higher degree obeying (\ref{bicross3}).

(ii) Next we check that $\Delta_*((\extd_H h) \la a)$ satisfies (\ref{bicross2}), 
\begin{align*}
\Delta_*((\extd_H h)\la a) &= \Delta_*(\extd_A (h\la a) - h\la \extd_A a)\\
=& \extd_A(h\la a)\o\tens (h\la a)\t +(h\la a)\o\tens \extd_A(h\la a)\t -\Delta_*(h\la \extd_A a)\\
=& \extd_A(h\o{}^{\barnot}\la a\o) \tens h\o{}^{\baro}(h\t \la a\t) +(h\o{}^{\barnot}\la a\o) \tens \extd_A( h\o{}^{\baro}(h\t \la a\t))\\
&-h\o{}^{\barnot} \la \extd_A a\o \tens h\o{}^{\baro}(h\t \la a\t)- h\o{}^{\barnot} \la a\o \tens h\o{}^{\baro}(h\t \la \extd_A a\t)\\
=& \extd_H h\o{}^{\barnot}\la a\o \tens h\o{}^{\baro} (h\t \la a\t) + (h\o{}^{\barnot} \la a\o)\tens (\extd_A h\o{}^{\baro})(h\t \la a\t)\\
&+h\o{}^{\barnot} \la a\o)\tens h\o{}^{\baro} ((\extd_H h\t) \la a\t) \\
=& (\extd_H h)\o{}^{\barnot}\la a\o \tens (\extd_H h)\o{}^{\baro}((\extd_H h)\t \la a\t)
\end{align*}
and one can find further that 
\begin{align*}
\Delta_*((h\extd_H g) \la a) =& (h\o{}^{\barnot} \extd_H g\o{}^{\barnot}) \la a\o \tens h\o{}^{\baro}(h\t\la g\o{}^{\baro})((h\t g\t)\la a\t)\\
&+(h\o{}^{\barnot} g\o{}^{\barnot}) \la a\o \tens h\o{}^{\baro}(h\t\la \extd_A g\o{}^{\baro})((h\t g\t)\la a\t)\\
&+(h\o{}^{\barnot} g\o{}^{\barnot}) \la a\o \tens h\o{}^{\baro}(h\t\la g\o{}^{\baro})((h\t \extd_H g\t)\la a\t)
\end{align*}
for all $a \in A$ as also required for (\ref{bicross2}). We can extend further to $\Delta_*((\extd_H h) \la \extd_A a)$ and prove that  (\ref{bicross2}) holds as
\begin{align*}
\Delta_*((\extd_H h) &\la \extd_A a) = \Delta_*(\extd_A(h \la \extd_A a)) \\
=& \extd_A (h\la \extd_A a)\o \tens (h\la \extd_A a)\t + (-1)^{|(h\la \extd_A a)\o|} (h\la \extd_A a)\o \tens \extd_A(h\la \extd_A a)\\
=&\extd_A(h\o{}^{\barnot} \la \extd_A a\o) \tens h\o{}^{\baro}(h\t \la a\t) - h\o{}^{\barnot} \la \extd_A a\o \tens \extd_A (h\o{}^{\baro}(h\t \la a\t))\\
&+\extd_A(h\o{}^{\barnot} \la a\o)\tens h\o{}^{\baro}(h\t \la \extd_A a\t)+ h\o{}^{\barnot} \la a\o\tens \extd_A( h\o{}^{\baro}(h\t \la \extd_A a\t))\\
=&(\extd_H h\o{}^{\barnot})\la \extd_A a\o \tens h\o{}^{\baro} (h\t \la a\t) - h\o{}^{\barnot}\la \extd_A a\o \tens (\extd_A h\o{}^{\baro})(h\t \la a\t)\\
&-h\o{}^{\barnot}\la \extd_A a\o \tens h\o{}^{\baro}((\extd_H h\t)\la a\t) + (\extd_H h\o{}^{\barnot})\la a\o \tens h\o{}^{\baro}(h\t \la \extd_A a\t)\\
&+h\o{}^{\barnot}\la a\o \tens (\extd_A h\o{}^{\baro})(h\t \la \extd_A a\t) + h\o{}^{\barnot} \la a\o \tens h\o{}^{\baro} ((\extd_H h\t)\la \extd_A a\t)\\
=&(-1)^{(|(\extd_H h)\o{}^{\baro}|+|(\extd_H h)\t|)|(\extd_A a)\o|} (\extd_H h)\o{}^{\barnot} \la (\extd_A a)\o \tens (\extd_H h)\o{}^{\baro}((\extd_H h)\t \la (\extd_A a)\t).
\end{align*}
This then extends to all degrees  since $\Delta_*$ of $\Omega(A)$ and  $\la : \Omega(H)\underline{\tens} \Omega(A) \to \Omega(A)$ are defined by assumption and $\beta_* : \Omega(H)\to \Omega(H)\underline{\tens} \Omega(A)$ is now defined since $\Omega(H)$ is the maximal prolongation of $\Omega^1(H)$. 

(iii) Finally, we check that $\beta_*$ obeys (\ref{bicross4}). In fact by applying $\extd_H \tens \id + \id \tens \extd_A$ to 
\[h\t{}^{\barnot}\tens (h\o \la a)h\t{}^{\baro} = h\o{}^{\barnot}\tens h\o{}^{\baro}(h\t \la a)
\]
combined with assumption (3), we have
\begin{align*}
\extd_H& h\t{}^{\barnot}\tens (h\o \la a)h\t{}^{\baro} + h\t{}^{\barnot} \tens ((\extd_H h\o)\la a) h\t{}^{\baro} + h\t{}^{\barnot}\tens (h\o\la a)\extd_A h\t{}^{\baro}\\
&=\extd_H h\o{}^{\barnot}\tens h\o{}^{\baro}(h\t \la a) +  h\o{}^{\barnot}\tens \extd_A h\o{}^{\baro}(h\t \la a) + h\o{}^{\barnot}\tens h\o{}^{\baro}(\extd_H h\t \la  a) 
\end{align*} 
for all $h\in H$, $a\in A$, which is equivalent to
\[(\extd_H h)\t{}^{\barnot}\tens ((\extd_H h) \o \la a)(\extd_H h)\t{}^{\baro} = (\extd_H h)\o{}^{\barnot}\tens (\extd_H h)\o{}^{\baro}((\extd_H h)\t \la a).\]
Furthermore, one can find that
\[(h\extd_H g)\t{}^{\barnot}\tens ((h\extd_H g) \o \la a)(h\extd_H g)\t{}^{\baro} = (h \extd_H g)\o{}^{\barnot}\tens (h\extd_H g)\o{}^{\baro}((h\extd_H g)\t \la a).\]

We  extend this by applying $\extd_H \tens \id + \id \tens \extd_A$ to the assumption (3), where we have
\begin{align*}
\extd_H& h\t{}^{\barnot}\tens (h\o \la \extd_A a)h\t{}^{\baro} + h\t{}^{\barnot} \tens \big( ((\extd_H h\o)\la \extd_A a)h\t{}^{\baro} - (h\o \la \extd_A a)\extd_A h\t{}^{\baro} \big)\\
&= \extd_H h\o{}^{\barnot}\tens h\o{}^{\baro}(h\t \la \extd_A a) + h\o{}^{\baro}\tens \big((\extd_A h\o)(h\t \la \extd_A a) + h\o{}^{\baro}((\extd_H h\t)\la \extd_A a)\big)
\end{align*}
which is equivalent to
\begin{align*}
(-1)^{|(\extd_H h)\t{}^{\baro}||\extd_A a|}(\extd_H h)\t{}^{\barnot}\tens& ((\extd_H h)\o \la \extd_A a)(\extd_H h)\t{}^{\baro}\\
&= (\extd_H h)\o{}^{\barnot}\tens (\extd_H h)\o{}^{\baro}((\extd_H h)\t \la \extd_Aa).
\end{align*}

Since $\la$ and $\beta_*$ are defined for $\Omega(H)$, one can extend the above equation further to higher degree by applying $\extd_H \tens \id + (-1)^{|\ |}\id \tens \extd_A$ to the lower degree equation. In this way, $\beta_*$ obeys (\ref{bicross4}), which completes the proof. \end{proof}

This lemma assists with the data for Theorem~\ref{calc bicross}. Moreover, as part of the theory of bicrossproduct Hopf algebras, there is a covariant right coaction $\Delta_R : H \to H \tens A\bicross H$ given by $\Delta_R h = h\o{}^{\barnot} \tens h\o{}^{\baro}\tens h\t$ for all $h\in H$.

\begin{corollary}
If $\la$ and $\beta_*$ obey the conditions in Theorem~\ref{calc bicross} then $\Delta_R : H \to H \tens A\bicross H$ as above is differentiable.
\end{corollary}
\begin{proof}
Since the coaction $\Omega(H)\to \Omega(H)\underline{\tens} \Omega(A)$ and action $\Omega(H)\underline{\tens}\Omega(A)\to \Omega(A)$ exist globally as assumed in Theorem~\ref{calc bicross}, it is clear that $\Delta_{R*} \eta=\eta\o{}^{\barnot *}\tens \eta\o{}^{\baro *}\tens \eta\t$ is well-defined on $\eta\in\Omega(H)$ and gives a coaction of $\Omega(A\bicross H)$ on $\Omega(H)$. For example, on degree 1 we have
\[\Delta_{R*}(\extd_H h) = \extd_H h\o{}^{\barnot} \tens h\o{}^{\baro}  \tens h\t + h\o{}^{\barnot} \tens \extd_A h\o{}^{\baro}  \tens h\t + h\o{}^{\barnot} \tens h\o{}^{\baro}  \tens \extd_H h\t .\]
\end{proof}
\subsection{Examples of bicrossproduct exterior algebras}\label{secexbicross}

We now turn to some examples of exterior algebras constructed as super bicrossproducts. We work over $\C$ due to the physical context (typically supplemented by an appropriate $*$-algebra real form). Our warm-up example
is the `Planck scale' Hopf algebra $\C[g,g^{-1}]\bicross \C[p]$ generated by $g,g^{-1},p$ with
\[[p,g]=\lambda (1-g)g, \quad \Delta g = g\tens g, \quad \Delta p = 1\tens p + p \tens g\]
where $\lambda = \imath\frac{\hbar}{\gamma}$ is an imaginary constant built from Planck's constant and a curvature length scale $\gamma$, and $g=e^{-\frac{x}{\gamma}}$ where $x$ is the spatial position in \cite{Ma:book} (whereas we work algebraically with $g$ as a generator).  This turned out to be a Drinfeld twist of $U(b_+)$ and we will obtain the same calculus as found in \cite{MaOec} by twisting methods.  Let $\C[r]$ be another copy of $\C[p]$ with standard exterior algebra $\Omega(\C[r])$ given by
\[[r,\extd r]=\lambda\extd r, \quad (\extd r)^2=0.\]
As part of the theory of bicrossproducts, there is a natural coaction $\Delta_R : \C[r]\to\C[r]\tens  \C[g,g^{-1}]\bicross \C[p]$ given by $\Delta_R r = 1\tens p + r \tens g$ making $\C[r]$ a comodule-algebra. This coaction extends to $\Omega(\C[r])$ by $\Delta_R\extd r=\extd r\tens g$. 

\begin{proposition}\label{explanck} There is a uniquely determined strongly bicovariant exterior algebra $\Omega(\C[g,g^{-1}]\bicross \C[p])$ such that $\Delta_R$ is differentiable. This has relations 
\[ [\extd g,g]=0, \quad [p,\extd p]=\lambda\extd p, \quad [\extd p, g]=\lambda g\extd g, \quad [\extd g,p]=\lambda (1-g)\extd g, \]
\[(\extd p)^2= (\extd g)^2=0, \quad \extd p \extd g = -\extd g \extd p.\]
Moreover, $\Omega(\C[g,g^{-1}]\bicross \C[p]) = \Omega(\C[g,g^{-1}])\bicross \Omega(\C[p])$  and has coproduct
\[\Delta g = g\tens g, \quad \Delta_* \extd g = \extd g \tens g + g\tens \extd g,\]
\[\Delta p = 1\tens p + p\tens g, \quad \Delta_* \extd p = 1\tens \extd p + \extd p \tens g + p \tens \extd g.\]
\end{proposition}
\begin{proof} If it exists then $\Delta_{R*}\extd r=\extd r\tens g+ 1\tens\extd p+r\tens\extd g$. For this to extend in a  well-defined way to $\Omega^1(\C[p])$ as in (\ref{diff coact1}), we 
require $\Delta_{R*}(\extd r.r-r\extd r + \lambda \extd r)=0$ which gives 
\[ [p,\extd p]=\lambda \extd p,\quad [\extd g, g]=0,\quad [\extd g,p]+[\extd p,g]+\lambda \extd g =0.\]
By requiring $\Delta_{R*}(\extd r)^2=0$ we find that $[g,\extd p]=\lambda g\extd g$ which implies $[\extd g,p]$ as stated. We also find the relations on degree 2 as stated. One can check that they obey the Leibniz rule, making $\Omega(\C[g,g^{-1}]\bicross \C[p])$ an exterior algebra. 

Moreover, one can find that $\Omega(\C[p])$ is a super right $\Omega(\C[g,g^{-1}])$-comodule and $\Omega(\C[g,g^{-1}])$ is a super right $\Omega(H)$-module by the  actions and coactions
\[p \la g = \lambda(1-g)g,\quad \extd p \la g = -\lambda g \extd g, \quad p \la \extd g = \lambda (1-g)\extd g, \quad \extd p \la \extd g = 0,\]
\[\beta (p) = p\tens g, \quad \beta_* (\extd p) = \extd p \tens g + p \tens \extd g.\]
It is a routine calculation to show that these obey the conditions for Theorem \ref{calc bicross}. It is then easy to check that the super bicrossproduct $\Omega(\C[g,g^{-1}])\bicross \Omega(\C[p])$ gives the same exterior algebra as $\Omega(\C[g,g^{-1}]\bicross\C[p])$ and with the stated coproduct. \end{proof}

Being uniquely determined here                                                                                                                                                               and below means in the universal sense that we use only the differentiability to derive the calculus relations (one could have quotients). The next most complicated example in this context is the quantum Poincar\'e group $\C[\mathrm{Poinc}_{1,1}]=\C[SO_{1,1}]\bicross U(b_+)$ for which we follow the construction in \cite[Chapter~9.1]{QRG}. Thus  $\C[SO_{1,1}]$ is the  `hyperbolic Hopf algebra' generated by $c=\cosh \alpha$ and $s=\sinh \alpha$ with relations $c^2-s^2=1$ and $cs=sc$, where we work algebraically with $c,s$ and do not need $\alpha$ itself. This is a Hopf algebra with 
\[\Delta \begin{pmatrix}
c ~~& s\\ s~~& c
\end{pmatrix} = \begin{pmatrix}
c ~~& s\\ s~~& c
\end{pmatrix} \tens \begin{pmatrix}
c ~~& s\\ s~~& c
\end{pmatrix}, \quad \epsilon \begin{pmatrix}
c ~~& s\\ s~~& c
\end{pmatrix} = \begin{pmatrix}
1 ~~& 0 \\ 0~~ & 1
\end{pmatrix}, \quad S \begin{pmatrix}
c ~~& s\\ s~~& c
\end{pmatrix} = \begin{pmatrix}
c ~~& -s \\ -s ~~& c
\end{pmatrix}.\]
Meanwhile, we take $ U(b_+)$ as generated by $a_0$ and $a_1$ with relation $[a_0, a_1]= \lambda a_0$ for some $\lambda \in \C$, and primitive comultiplication $\Delta a_i = 1\tens a_i + a_i \tens 1$ for $i=0,1$.  $\C[SO_{1,1}]$ coacts on $U(b_+)$ and $U(b_+)$ acts on $\C[SO_{1,1}]$ by 
\[a_0 \la \begin{pmatrix}
c \\ s
\end{pmatrix} = \lambda \begin{pmatrix}
s^2 \\ sc
\end{pmatrix}, \quad a_1 \la \begin{pmatrix}
c \\ s
\end{pmatrix} = \lambda \begin{pmatrix}
cs-s \\ c^2-c
\end{pmatrix}, \quad \Delta_R \begin{pmatrix}
a_0 ~~& a_1
\end{pmatrix} = \Delta_R \begin{pmatrix}
a_0 ~~& a_1
\end{pmatrix} \tens \begin{pmatrix}
c ~~& s \\ s~~& c.
\end{pmatrix} \]
Thus their bicrossproduct $\C_\lambda[\mathrm{Poinc_{1,1}}] = \C[SO_{1,1}]\bicross U(b_+)$ contains $\C[SO_{1,1}]$ as sub-Hopf algebra, $U(b_+)$ as subalgebra, and one has the additional cross-relations and coproducts\cite[Chapter~9.1]{QRG}
\[[a_0, \begin{pmatrix}
c \\ s
\end{pmatrix}] = \lambda s\begin{pmatrix}
s \\ c
\end{pmatrix}, \quad [a_1, \begin{pmatrix}
c \\ s
\end{pmatrix}] = \lambda (c-1)\begin{pmatrix}
s \\ c
\end{pmatrix}, \quad \Delta a_i = 1\tens a_i + \Delta_R a_i.\]

Finally, we let $U(b_+)$ be another copy with generators $t,x$ in place of $a_0,a_1$ with exterior algebra\cite{Sit}
\[[\extd x, x]=\lambda \theta', \quad [\extd x, t]=0, \quad [\extd t, x] = \lambda \extd x, \quad [\extd t, t] = \lambda (\extd t-\theta'), \quad [\theta', x]=0, \quad [\theta', t]=-\lambda \theta'\]
\[\extd \theta'=0, \quad (\extd x)^2=(\extd t)^2=0, \quad \extd x \extd t = -\extd t \extd x, \quad \extd x. \theta' = -\theta' \extd x, \quad \extd t .\theta' = -\theta'\extd t.\]
It is known that this is covariant under $\C_\lambda[{\rm Poinc}_{1,1}]$ by coaction 
\[ \Delta_R : U(b_+) \to U(b_+) \tens \C[\mathrm{Poinc}_{1,1}],\quad \Delta_R (t\ \, x) = 1\tens (a_0 \ \,  a_1)+ (t\ \,  x)\tens \begin{pmatrix}c ~~& s\\ s ~~& c\end{pmatrix}\]
in the same matrix notation as for coproducts above. We are thinking of $U(b_+)$ here as a quantisation of 2D Minkowski spacetime with its known 3D differential calculus. 

\begin{proposition}\label{expoinc} There is a uniquely determined strongly bicovariant exterior algebra $\Omega(\C[\mathrm{Poinc}_{1,1}])$ such that the above coaction on $U(b_+)$ is differentiable. It contains $\Omega(U(b_+))$ as sub-exterior algebras and has the additional relations on degree 1 and 2,
\[(\extd c)s = s\extd c, \quad (\extd s)c= c\extd s, \quad (\extd c)c=c\extd c= s\extd s=(\extd s)s, \]
\[[\extd a_0, \begin{pmatrix}
c \\ s
\end{pmatrix}] = \lambda c \begin{pmatrix}
\extd c \\ \extd s
\end{pmatrix}, \quad [\extd a_1, \begin{pmatrix}
c \\ s
\end{pmatrix}]= \lambda s \begin{pmatrix}
\extd c \\ \extd s
\end{pmatrix},\]
\[[a_0, \begin{pmatrix}
\extd c \\ \extd s
\end{pmatrix}] = \lambda s \begin{pmatrix}
\extd s \\ \extd c
\end{pmatrix}, \quad [a_1, \begin{pmatrix}
\extd c \\ \extd s
\end{pmatrix}] = \lambda (c-1)\begin{pmatrix}
\extd s \\ \extd c
\end{pmatrix},\quad [\theta', \begin{pmatrix}
c \\ s
\end{pmatrix}] = \lambda (c-1)\begin{pmatrix}
\extd c \\ \extd s
\end{pmatrix},\]
\[\{\extd a_0, \begin{pmatrix}
\extd c\\ \extd s
\end{pmatrix} \} = \begin{pmatrix}
0 \\ \lambda \extd s \extd c
\end{pmatrix},  \quad \{\extd a_1, \begin{pmatrix}
\extd c\\ \extd s
\end{pmatrix} \} = \begin{pmatrix}
\lambda \extd c \extd s \\ 0
\end{pmatrix}.\]
Moreover, $\Omega(\C[\mathrm{Poinc}_{1,1}])=\Omega(\C[SO_{1,1}])\bicross \Omega(U(b_+))$  with  coproduct that of $\C_\lambda[{\rm Poinc}_{1,1}]$ on degree 0 and
\[\Delta_* \extd c = \extd c\tens c + \extd s \tens s + c\tens \extd c + s\tens \extd s, \quad \Delta \extd s = \extd c \tens s + \extd s \tens c + c\tens \extd s+ s\tens \extd c,\]
\[\Delta_* \extd a_0 = 1 \tens \extd a_0 + \extd a_0 \tens c + \extd a_1 \tens s + a_0 \tens \extd c + a_1 \tens \extd s,\]
\[\Delta_* \extd a_1 = 1 \tens \extd a_1 + \extd a_0 \tens s + \extd a_1 \tens c + a_0 \tens \extd s + a_1 \tens \extd c,\]
\[\Delta_* \theta' = 1 \tens \theta' + \theta'\tens 1 - 1\tens \extd a_0 - \extd a_0 \tens 1 + \Delta_* \extd a_0.\]
\end{proposition}
\begin{proof} If  $\Delta_{R*}$ exists obeying (\ref{diff coact1}), we will need
\[\Delta_{R*}\extd t = 1\tens \extd a_0 + \extd t \tens c + \extd x \tens s + t\tens \extd c + x\tens \extd s,\]
\[\Delta_{R*}\extd x = 1\tens \extd a_1 + \extd t \tens s + \extd x \tens c + t\tens \extd s + x\tens \extd c,\]
\[\Delta_{R*}\theta' = 1\tens \theta' -1\tens \extd a_0 + \theta' \tens 1 -\extd t \tens 1 + \Delta_{R*}\extd t.  \]
By requiring $\Delta_{R*}([\extd t,x])= \Delta_{R*}(\lambda \extd x)$ and $\Delta_{R*}([\extd x, t]) = 0$, we find 
\[(\extd s)c = c\extd s, \quad (\extd c)s = s\extd c, \quad (\extd c)c + (\extd s)s - c \extd c - s\extd s = 0.\]
\[[\extd a_1, c] = [a_0, \extd s], \quad [\extd a_1, s]-[a_0, \extd c]+ \lambda ((\extd s) s - c \extd c) =0,\]
\[[a_1, \extd c] = [\extd a_0, s]-\lambda \extd s, \quad [\extd a_0, c]-[a_1, \extd s]+ \lambda ((\extd c)c - s\extd s - \extd c)=0.\]
Additionally, from $\extd(c^2 -s^2)=0$ combined with the condition $\extd c.c +\extd s.s -c\extd c-s\extd s=0$ above, we find that $(\extd c)c= s\extd s$ and $(\extd s)s = c\extd c$. But then by requiring $\Delta_{R*}([\extd x, x])=\Delta_{R*}(\lambda \theta')$ we also find that $(\extd c)c=c\extd c$ and $(\extd s)s = s\extd s$. Thus we have $(\extd c)c=c\extd c = s\extd s = (\extd s)s$. Also, by requiring $\Delta_{R*}(\extd x)^2=0$ and $\Delta_{R*}(\extd t)^2=0$, one can find that
\[[\extd a_0, \begin{pmatrix}
c \\ s
\end{pmatrix}] = \lambda c \begin{pmatrix}
\extd c \\ \extd s
\end{pmatrix}, \quad [\extd a_1, \begin{pmatrix}
c \\ s
\end{pmatrix}]= \lambda s \begin{pmatrix}
\extd c \\ \extd s
\end{pmatrix}\]
leading to the stated bimodule relations. The relations involving $\theta'$ are obtained by using the other  relations. One can also find that the stated relations on degree 2 hold and $\Omega(\C_\lambda[\rm{Poinc}_{1,1}])$ contains $\Omega(U(b_+))$ from the above calculation.
 
Moreover, $\Omega(\C[SO_{1,1}])$ is a super  ${\Omega(U(b_+))}$-module and $\Omega(U(b_+))$ is a super ${\Omega(\C[SO_{1,1}])}$-comodule with actions and coactions
\[\extd a_0 \la \begin{pmatrix}
c \\ s
\end{pmatrix} = \lambda c \begin{pmatrix}
\extd c \\ \extd s
\end{pmatrix}, \quad \extd a_1 \la \begin{pmatrix}
c \\ s
\end{pmatrix} = \lambda s \begin{pmatrix}
\extd c \\ \extd s
\end{pmatrix} ,\]
\[a_0 \la \begin{pmatrix}
\extd c \\ \extd s
\end{pmatrix} = \lambda s \begin{pmatrix}
\extd s \\ \extd c
\end{pmatrix}, \quad a_1 \la \begin{pmatrix}
\extd c \\ \extd s
\end{pmatrix} = \lambda (c-1) \begin{pmatrix}
\extd s \\ \extd c
\end{pmatrix},\]
\[\beta_*(\extd a_0) = \extd a_0 \tens c + \extd a_1 \tens s + a_0 \tens \extd c + a_1 \tens \extd s,\]
\[\beta_*(\extd a_1) = \extd a_0 \tens s + \extd a_1 \tens c + a_0 \tens \extd s + a_1 \tens \extd c,\]
\[\theta \la \begin{pmatrix}
c \\ s
\end{pmatrix} = \lambda (c-1)\begin{pmatrix}
\extd c \\ \extd s
\end{pmatrix}, \quad \beta_* (\theta') = \theta'\tens 1 - \extd a_0\tens 1 + \beta_*( \extd a_0),\]
and one can check that these obey the conditions for Theorem~\ref{calc bicross}. We can then construct the super bicrossproduct and identify \sloppy $\Omega(\C_\lambda[\mathrm{Poinc}_{1,1}]) = \Omega(\C[SO_{1,1}])\bicross\Omega(U(b_+))$ with coproducts as stated. 
\end{proof}

\end{document}